\documentclass[12pt]{article}
\usepackage[a4paper,left=3cm,right=3cm,top=2cm,bottom=4cm,bindingoffset=5mm]{geometry}

\usepackage{amscd}

\usepackage{enumerate}

\usepackage{amsmath, amssymb, latexsym}
\usepackage{amsfonts}
\usepackage{amsthm}
\usepackage{relsize}
\usepackage{setspace}
\usepackage{geometry}
\usepackage{url}
\usepackage{xspace}
\usepackage{tocloft}
\usepackage{graphics}
\usepackage{graphicx}
\usepackage{lscape}
\usepackage{microtype}

\usepackage[usenames, dvipsnames]{color}
\usepackage[utf8]{inputenc}
\usepackage{tikz}

\usepackage[pagebackref=true]{hyperref}
\usepackage[alphabetic, nobysame]{amsrefs}
\usepackage[english]{babel}

\usepackage{authblk}

\usepackage{tikz}
\usetikzlibrary{arrows} 
\tikzset{>=latex}
\usetikzlibrary{automata,topaths}
\usepackage{mathrsfs}

\newtheorem{theorem}{Theorem}[section]
\newtheorem{proposition}[theorem]{Proposition}

\newtheorem{cor}[theorem]{Corollary}
\newtheorem{lemma}[theorem]{Lemma}

\theoremstyle{definition}
\newtheorem{definition}[theorem]{Definition}
\newtheorem{example}[theorem]{Example}
\newtheorem{remark}[theorem]{Remark}

\theoremstyle{problem}

\newtheoremstyle{citing}% name
  {}%      Space above, empty = `usual value'
  {}%      Space below
  {\itshape}% Body font
  {}%         Indent amount (empty = no indent, \parindent = para indent)
  {\bfseries}% Thm head font
  {.}%        Punctuation after thm head
  {.5em}%     Space after thm head: " " = normal interword space;
  {\thmnote{#3}}% Thm head spec

\theoremstyle{citing}
\newtheorem*{varthm}{}% all text supplied in the note

\def\calnc         { {\mathcal {NC}}}
\def\scrh          {\mathscr H}
\def\Z      {\mathbb Z}
\def\N      {\mathbb N}

\def\s      {\sigma}
\def\ord    {\textup{ord}}
\def\f    {\textup{fi}}
\let\emptyset\varnothing

\title{Subgroup graph methods\\
 for presentations of finitely generated groups\\
  and the connectivity of associated simplicial complexes}

\author{Cora Welsch}
\affil{
Mathematisches Institut, Universit\"at M\"unster, \\
Einsteinstra\ss e 62, 48149 M\"unster,\\
 E-mail: cora.welsch@uni-muenster.de}

\date{\today}

\begin{document}

\maketitle

\begin{abstract}
In this article we generalize the theory of subgroup graphs of subgroups of free groups, developed by I. Kapovich and A. Myasnikov, based on a work by J. Stallings, to finite index subgroups of finitely generated groups. Given a presentation of a finitely generated group $G=\left\langle\,  X \, | \, R \, \right\rangle$ and a finite connected $X$-regular graph $\Gamma$ which fulfills the relators $R$, we associate to $\Gamma$ a finite index subgroup $H$ of $G$. Conversely, the Schreier coset graph of $H$ with respect to $X$ and $G$ is such a graph $\Gamma$. 
Firstly, we study and prove various properties of $H$ in relation to the graph $\Gamma$.  Secondly, we prove that for many finitely generated infinite groups the order and nerve complexes that we associate to $G$ are contractible.  
In particular, this is the case for free and free abelian groups, Fuchsian groups of genus $g \geq 2$, infinite right angled Coxeter groups, Artin and pure braid groups, infinite virtually cyclic groups, Baumslag-Solitar groups as well as the (free) product of at least two of these, and all finite index subgroups of these groups. 
\end{abstract}

\section{Introduction}
\label{Intro}
In this article we generalize the theory of subgroup graphs of subgroups of free groups, developed by I. Kapovich and A. Myasnikov in \cite{KM02}, based on J. Stallings' work \cite{S}, to finite index subgroups of finitely generated groups. To a folded graph whose edges are labeled by elements of $X$ they associate a subgroup of the free group $F(X)$. Conversely, they prove that for every subgroup $H\leq F(X)$ there exists a based $X$-graph $(\Gamma,v)$, unique up to isomorphism, which is connected, folded, a core graph with respect to $v$ and which has language $H=L(\Gamma,v)$.  They call this graph the subgroup graph $\Gamma_X(H)$. Furthermore, they show that it is the core of the Schreier coset graph of $H$ with respect to $F(X)$ at the base-vertex $H$ (see Definitions \ref{defxgraph}, \ref{deflang}, \ref{deffolded}, \ref{defcore} and \ref{def1}).

We develop a theory of subgroup graphs of finite index subgroups of finitely generated groups as follows. Let $\phi\colon F(X) \rightarrow G\cong F(X) / \left\langle\!\left\langle\, R\, \right\rangle\!\right\rangle_{F(X)}$ be the canonical epimorphism.

\begin{varthm}[Theorem \ref{thm2}] \textup{(Subgroup Graph)}\mbox{}\\
Let $G$ be a group with a presentation $G= \langle \,X\, |\, R\, \rangle$, where $X$ is finite and $R$ is not necessarily finite. 
\begin{enumerate}[(1)]
\item Let $\Gamma$ be an $X$-regular connected graph with $n$ vertices. Let $v_0$ be a base-vertex of $\Gamma$. Assume that $\Gamma$ fulfills the defining relators $R$. 
Then $\phi(L(\Gamma,v_0))$ is a subgroup of $G$ of index $[G:\phi(L(\Gamma,v_0))]=n$.
\item Let $H\leq G$ be a subgroup of index $[G:H]=n\in \N$. Then there exists a based $X$-graph $(\Gamma,v_0)$ (unique up to a canonical isomorphism of based $X$-graphs)  such that
\begin{enumerate}[(i)]
\item 
$\Gamma$ is $X$-regular and connected (see Definition \ref{defregular});
\item $\Gamma$ fulfills the defining relators $R$  (see Definition \ref{deffulfil});
\item $\Gamma$ has $n$ vertices; 
\item  $\phi(L(\Gamma, v_0))=H$.
\end{enumerate}
\end{enumerate}
In this situation we call $\Gamma$ the subgroup graph of $H$ with respect to $X$ and $R$. We denote it by $\Gamma_{X,R}(H)$ or briefly by $\Gamma(H)$. The base-vertex $v_0$ is denoted by $1_H$.
In fact, $\Gamma_{X,R}(H)=\Gamma_X(H')$, where $H':=L(\Gamma_{X,R}(H),1_H)\leq F(X)$.
\end{varthm}

The subgroup graph $\Gamma_{X,R}(H)$ is the Schreier coset graph of $H$ with respect to $X$ and $G$, with $H$ as the base-vertex. 
The important part of the theorem above is part (1). It shows that if we can construct such a graph, which is easily done, then there exists such a subgroup $H\leq G$. Furthermore,  
the subgroup graph $\Gamma_{X,R}(H)$ provides useful information about various properties of $H$. For example:

\begin{varthm}[Proposition \ref{propmorsub2}] \textup{(Morphisms and Subgroups)}\mbox{}\\
Let $\pi\colon \Gamma\rightarrow \Gamma' $ be a morphism of $X$-graphs, let $v\in V(\Gamma)$ and $v'=\pi(v)$. Let $G =\langle \,X\, |\, R\, \rangle$ be a group with $X$ finite and $R$  not necessarily finite. Suppose that $\Gamma$ and $\Gamma'$ are connected finite $X$-regular graphs which fulfill the defining relators $R$. Put $K=L(\Gamma,v)$ and $H=L(\Gamma', v')$. Then $\phi(K) \leq \phi(H)\leq G$.
\end{varthm}

\begin{varthm}[Proposition \ref{con2}] \textup{(Conjugate Subgroups)}\mbox{}\\
Let $H$ and $K$ be subgroups of finite index in the group $G=\langle \,X\, |\, R\, \rangle$, where $X$ is finite and $R$ is not necessarily finite. Then $H$ is conjugate to $K$ in $G$ if and only if the subgroup graphs $\Gamma_{X,R}(H)$ and $\Gamma_{X,R}(K)$ are isomorphic as $X$-graphs. 
\end{varthm}

 \begin{varthm}[Theorem \ref{norm2}] \textup{(Normal Subgroups)}\mbox{}\\
Let $H$ be a finite index subgroup of the group $G= \langle \,X\, |\, R\, \rangle$, where $X$ is finite and $R$ is not necessarily finite.  Then $H$ is normal in $G$ if and only if the based $X$-graphs $(\Gamma_{X,R}(H),1_H)$ and $(\Gamma_{X,R}(H),v)$ are isomorphic for all $v\in V(\Gamma_{X,R}(H))$. 
\end{varthm}

\begin{varthm}[Theorem \ref{thmnormalizer}] \textup{(Normalizer)}\mbox{}\\
Let $G =\langle \,X\, |\, R\, \rangle$ be a group with $X$ finite and $R$ not necessarily finite. Let $H$ be a finite index subgroup of $G$. Let $p_v$ be the reduced path in $\Gamma_{X,R}(H)$ from $1_H$ to $v$ with label $\mu(p_v)=g_v$. Then $g_v\in N_G(H)$ if and only if $(\Gamma_{X,R}(H),1_H)$ and $(\Gamma_{X,R}(H),v)$ are isomorphic as based $X$-graphs. Furthermore, let $V$ be the set of vertices of $\Gamma_{X,R}(H)$ with $(\Gamma_{X,R}(H),1_H)$ isomorph to $(\Gamma_{X,R}(H),v)$ as based $X$-graphs. Then
$$N_G(H)=\bigcup\limits_{v\in V} Hg_v. $$
\end{varthm}

\begin{varthm}[Proposition \ref{propint}] \textup{(Intersection)}\mbox{}\\
Let $H$ and $K$ be finite index subgroups of the  group $G=\langle \,X\, |\, R\, \rangle$, where $X$ is finite and $R$ is not necessarily  finite.   
Let $\Gamma(H) \times_1 \Gamma( K)$ be the connected component of the product graph $\Gamma_{X,R}(H) \times \Gamma_{X,R}(K)$ containing $1_H \times 1_K$ (see Definition \ref{defproductgraph}). Then $(\Gamma( H) \times_1 \Gamma( K), 1_{ H} \times 1_{ K})$ is the subgroup graph of  
$H \cap K<G$.
\end{varthm}

As a consequence of the theory of subgroup graphs developed in this article, we are able to answer questions related to the connectivity of order complexes. 

In \cite{B00} K.S. Brown asked two questions about the connectivity of the coset poset $\mathscr C (G)$ of a finite group $G$. The coset poset $\mathscr C (G)$ is the collection of all left cosets of all proper subgroups of $G$, the cosets being ordered by inclusion. The connectivity of the coset poset $\mathscr C (G)$ is the connectivity of the order complex $\Delta \mathscr C (G)$ (see Definition \ref{deforder}).

Brown's first question, whether there is any finite nontrivial group whose coset poset is contractible, was recently answered in the negative by J.  Shareshian and R. Woodroofe \cite{SW14}. Brown's second question, to find finite groups whose coset posets are simply connected, was investigated by D.A. Ramras in \cite{R05}. Ramras also studied $\mathscr C (G)$ for infinite groups $G$. He showed that the coset poset is contractible if $G$ is not finitely generated. 

We consider a different poset of cosets. Specifically, $P_\f (G)$ is the collection of all right cosets of all proper finite index subgroups of $G$. If $G$ is finite, then $\mathscr C (G)=P_\f(G)$, since $xH=(xHx^{-1})x$. We are interested in the connectivity of the order complex $\Delta P_\f (G)$ for the poset $(P_\f (G), \subseteq)$. We study the case when $G$ is an infinite finitely generated group. 
In fact, we consider the connectivity of the nerve complex $\calnc (G, \scrh_\f)$ (see Definition \ref{defnerve}), which is homotopy equivalent to $\Delta P_\f (G)$. The set $\scrh_\f$ is the set of all proper finite index subgroups of $G$. To investigate its connectivity, we use the subgroup graphs developed in this article.  
We obtain the following main results. 

The $X$-graphs $\Gamma_p$ below are subgroup graphs with $p$ vertices such that there exists an $X$-word $w_p$ with $V(\Gamma_p)=\{t(p_i) \mid o(p_i)=v, 0 \leq i < p\}$ for each reduced path $p_i$ with label $\mu(p_i)=w_p^i$.

\begin{varthm}[Theorem \ref{thmcon}] \mbox{}\\
Let $G= \langle \,X\, |\, R\, \rangle$ be a  group with $X$ finite and $R$  not necessarily finite. Suppose that 
there exists a collection $\{\Gamma_p\}_{p\in P}$  of $X$-graphs $\Gamma_p$  as in Proposition \ref{propGp} such that $P$ is a set of infinitely many primes. Suppose also that there exists a freely reduced $X$-word $w$ with $w_p=w$ uniformly for all $p\in P$. Then the nerve complex $\calnc (G, \scrh_\f)$ and  the order complex $\Delta P_\f (G)$ are contractible. 
\end{varthm}

\begin{varthm}[Theorem \ref{thmconsub}] \mbox{}\\
Let $G=\left\langle\,  X \, | \, R\, \right\rangle$ be a group with $X$ finite and $R$  not necessarily finite. 
Let $H$ be a finite index subgroup of $G$. Suppose that $G$ is a group such that there exists a collection $\{\Gamma_p\}_{p\in P}$ of subgroup graphs as in Theorem \ref{thmcon}. Then the nerve complex $\calnc(H, \scrh_\f )$ and the order complex $\Delta P_\f (H)$ are contractible.
\end{varthm}

\begin{varthm}[Theorem \ref{thmast1}] \mbox{}\\
Let $G_1=\left\langle \, X_1 \, | \, R_1 \,\right\rangle$ and $G_2=\left\langle \,  X_2 \, | \, R_2 \,\right\rangle$ be  groups with $X_1$, $X_2$ finite and $R_1$, $R_2$ not necessarily finite. Suppose that there exist proper subgroups $H_1<  G_1$ and $H_2 < G_2$ such that $\{1, w_1, w_1^2, ..., w_1^{n_1-1}\}$ and $\{1,w_2, w_2^2,..., w_2^{n_2-1}\}$ are full sets of coset representatives of $H_1$, $H_2$ in $G_1$, $G_2$ for some $w_1\in G_1$, $w_2\in G_2$. Then the nerve complex $\calnc (G_1 \ast G_2, \scrh_\f )$ and the order complex $\Delta P_\f (G_1 \ast G_2)$  are contractible. 
\end{varthm}

\begin{varthm}[Theorem \ref{thmast2}] \mbox{}\\
Let $G_1=\left\langle \, X_1 \, | \, R_1 \,\right\rangle$ and $G_2=\left\langle \,  X_2 \, | \, R_2 \,\right\rangle$ be  groups with $X_1$, $X_2$ finite and $R_1$, $R_2$ not necessarily finite. Suppose that $G_1$ is a group such that there exists a collection $\{\Gamma_p\}_{p\in P}$  of $X_1$-graphs as in Theorem \ref{thmcon}.  
Then the nerve complex $\calnc(G, \scrh_\f )$ and the order complex $\Delta P_\f (G)$ are contractible for $G=G_1\ast G_2$, $G=G_1\times G_2$ or $G=G_2 \rtimes G_1$.
\end{varthm}

\begin{varthm}[Corollary \ref{corende}] \mbox{}\\
Let $G$ be either the free product $G_1 \ast ...\ast G_n$, the direct product $G_1 \times ... \times G_n$, the semidirect product $G_1 \rtimes G_2$ or a finite index subgroup of the group $G_1$. Suppose that each $G_i$, for $1\leq i \leq n$, is one of the following finitely generated groups: \\
 a free group; a free abelian group; a Fuchsian group of genus $g \geq 2$;  an infinite right angled Coxeter group; an Artin group; a pure braid group; a Baumslag-Solitar group or an infinite virtually cyclic group. \\
Then the nerve complex $\calnc (G, \scrh_\f)$ and the order complex $\Delta P_\f (G)$ are contractible.
\end{varthm}

\begin{varthm}[Theorem \ref{thmamalgam}] \mbox{}\\
Let $G_1$ and $G_2$ be finitely generated groups. Suppose there exist proper normal subgroups  $H_1 \lhd  G_1$ and $H_2 \lhd G_2$  of finite index such that  $H_1 \cong H_2$ and $\{1,w_1,...,w_1^{n_1-1}\}$ and $\{1,w_2,...,w_2^{n_2-1}\}$ are full sets of coset representatives of $H_1< G_1$ and $H_2<G_2$ for some $w_1\in G_1$, $w_2\in G_2$. Then the nerve complex $\calnc(G, \scrh_\f)$ and the order complex $\Delta P_\f(G)$ are contractible for $G=G_1 \ast_D G_2$ with $\{1\}\leq D \leq H_1$.
\end{varthm}

The article is organized as follows.  Sections \ref{subfree}  is an introduction to the theory of subgroup graphs of subgroups of free groups. We recall the relevant material from \cite{KM02} without proofs, in order to make our article self-contained. 

Section \ref{ssub} develops our theory of subgroup graphs of finite index subgroups of finitely generated groups, see Theorem \ref{thm2}. Section \ref{appsub} gives applications of this theory. Most of them are results of \cite{KM02} which we generalize to finite index subgroups of finitely generated groups. In addition, we give  a method to detect the existence of Hall subgroups and normal subgroups as well as a method to compute the normalizer of a finite index subgroup.

In Section \ref{secconnect} and \ref{secsuf} our main results are stated and proved. There we use  the subgroup graph to prove the connectivity of the order complex $\Delta P_\f (G)$, where $G$ is a finitely generated group.
In Section \ref{secconnect} we consider special classes of infinite finitely generated groups. In Section \ref{secsuf} we prove sufficient conditions for finitely generated groups $G_1$, $G_2$ such that $\Delta P_\f(G_1\ast G_2)$, $P_\f(G_1\times G_2)$, $P_\f(G_1\rtimes G_2)$, $\Delta P_\f (G_1\ast _D G_2)$ and $\Delta P_\f (H)$ are contractible for $H\leq G_1$ a finite index subgroup.

%%%%%%%

\subsubsection*{Acknowledgment} This is part of the author's Ph.D. thesis. The author greatly indebted to Linus Kramer for suggesting the problem and to Corina Ciobotaru for reading some parts.

%%%%%%%%%%%%%%%%%%%%%%%%%%%%%%%%%%%

\section{Subgroup graphs of subgroups of free groups}
\label{subfree}

%%%%%%%%%%%%%%%%%%%%%%%%%%%%%

Following the article of Ilya Kapovich and Alexei Myasnikov \cite{KM02}, this section is meant to recall the theory of subgroup graphs for free groups. Their approach is more combinatorial and computational than the topological one by J. Stallings \cite{S}. Both approaches study subgroups of free groups. Let $F(X)$ be a finitely generated free group. To a based folded connected $X$-graph $(\Gamma,v)$  which is a core graph  with respect to $v$ they associate a subgroup $H\leq F(X)$, by taking the language $L(\Gamma,v)=H$ (see Definitions \ref{defxgraph}, \ref{deflang}, \ref{deffolded} and \ref{defcore}). 
Conversely, for every subgroup $H\leq F(X)$ there exists a  based $X$-graph $(\Gamma',v')$, unique up to isomorphism, which is folded, connected, a core graph with respect to $v'$ and has language $L(\Gamma', v')=H$. 
The correspondence is unique up to isomorphism of based $X$-graphs.  
Therefore we call $(\Gamma',v')$ the subgroup graph $\Gamma_X(H)$ of $H\leq F(X)$ with base-vertex $v'=1_H$. 
 The graph $\Gamma_X(H)$ is the core of the Schreier coset graph of $H$ with respect to $F(X)$ at the base-vertex $H$ (see Definition \ref{def1}).  
If $H$ is a subgroup of finite index in $F(X)$, then $\Gamma_X(H)$ is the Schreier coset graph of $H$ with respect to $F(X)$. \\

\begin{definition}{($X$-Graph, see \cite[2.1]{KM02})}\label{defxgraph}\mbox{}\\
Let $X$ be a finite set that is called an \textit{alphabet}. Let $\Gamma$ be a (finite or infinite) directed multi-edge graph with vertex set $V(\Gamma)$ and $E(\Gamma)$ the set of directed edges. We denote by $o(e)$ the \textit{origin} and by $t(e)$ the \textit{terminus} of the edge $e$.  We say $e$ is an edge from $o(e)$ to $t(e)$.  If $o(e)=t(e)$ then $e$ is called a \textit{loop}. 

A graph $\Gamma$ is called an \textit{$X$-labeled directed graph} (or \textit{$X$-digraph}, or \textit{$X$-graph}) if every directed edge $e \in E(\Gamma)$ is labeled by a letter from $X$, which is denoted by $\mu(e)$. 
 
A map $\pi\colon \Gamma \rightarrow \Gamma'$, between two $X$-graphs is called a \textit{morphism} of $X$-graphs if $\pi$ takes vertices to vertices, directed edges to directed edges, preserves labels of directed edges and has the property that $o(\pi(e))=\pi(o(e))$, $t(\pi(e))=\pi(t(e))$ for every edge $e$ of $\Gamma$. 
Furthermore,  a morphism of based $X$-graphs $(\Gamma,v)$, $(\Gamma',v')$ maps $v$ to $v'$.
\end{definition}

In this article $X^{-1}$ will always denote the set of all formal inverses of the elements in $X$ and $X$ will always be finite. By adding edges to an $X$-graph $\Gamma$, we obtain a new graph $\widehat{\Gamma}$.

\begin{definition}{(See \cite[2.2]{KM02})\\}
Let $\Gamma$ be an $X$-graph. We define the \textit{$(X\cup X^{-1})$-graph $\widehat{\Gamma}$} as follows:
$V(\widehat{\Gamma}):=V(\Gamma)$ and for every edge $e \in E(\Gamma)$ we introduce the formal inverse $e^{-1}$ of $e$, whose label is $\mu(e)^{-1}$. 
The endpoints of $e^{-1}$ are $o(e^{-1}):=t(e)$ and $t(e^{-1}):=o(e)$. 
For a new edge $e^{-1}$, we set $(e^{-1})^{-1}=e$. We have that $E(\widehat{\Gamma}):=E(\Gamma) \cup E(\Gamma)^{-1}$. For an example see Figure \ref{Gdach}. 
\end{definition} 
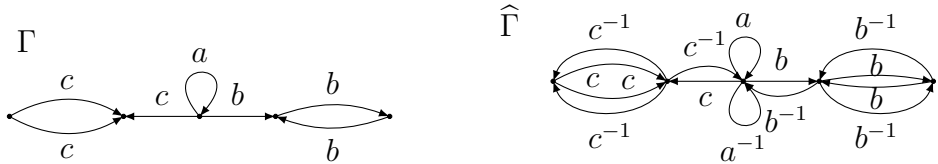
\begin{figure}[h!]
	\centering
  \begin{tikzpicture}     

\coordinate[label=left:$\Gamma$](0) at (0,1);
\coordinate (1) at (-0.5,0);
\coordinate (2) at (1,0);
\coordinate (3) at (2,0);
\coordinate (4) at (3,0);
\coordinate (5) at (4.5,0);
 \filldraw[black](1) circle (0.8pt)
 (2) circle (0.8pt)
 (3) circle (0.8pt)
 (4) circle (0.8pt)
 (5) circle (0.8pt)
 ;
\path[->,min distance=1cm] (3) edge[in=50,out=130,above] node {$a$}(3);
\draw [->] (4) to[bend left=20] node[above] {$b$} (5);
\draw [->] (5) to[bend left=20] node[below] {$b$} (4);
\draw [->] (1) to[bend left=30] node[above] {$c$} (2);
\draw [->] (1) to[bend right=30] node[below] {$c$} (2);
%\draw [->] (1) to node[above] {$c$} (2);
\draw [->] (3) to node[above] {$c$} (2);
\draw [->] (3) to node[above] {$b$} (4);
\end{tikzpicture}
\hspace{10mm}
\begin{tikzpicture}     

\coordinate[label=left:$\widehat{\Gamma}$](0) at (-0.8,0.8);
\coordinate (1) at (-0.50,0);
\coordinate (2) at (1,0);
\coordinate (3) at (2,0);
\coordinate (4) at (3,0);
\coordinate (5) at (4.5,0);
 \filldraw[black](1) circle (0.8pt)
 (2) circle (0.8pt)
 (3) circle (0.8pt)
 (4) circle (0.8pt)
 (5) circle (0.8pt)
 ;
\path[->,min distance=1cm] (3) edge[in=50,out=130,above] node {$a$}(3);
\draw [->] (4) to[bend left=10] node[below] {$b$} (5);
\draw [->] (5) to[bend left=10] node[above] {$b$} (4);
\draw [->] (1) to[bend left=30] node[below right] {$c$} (2);
\draw [->] (1) to[bend right=30] node[above left] {$c$} (2);

%\draw [->] (1) to node[above] {$c$} (2);
\draw [->] (3) to node[below] {$c$} (2);
\draw [->] (3) to node[above] {$b$} (4);
\path[->,min distance=1cm] (3) edge[in=310,out=230,below] node {$a^{-1}$}(3);
\draw [->] (5) to[bend right=70] node[above] {$b^{-1}$} (4);
\draw [->] (4) to[bend right=70] node[below] {$b^{-1}$} (5);
%\draw [->] (2) to[bend left=40] node[below] {$c^{-1}$}
 (1);
 \draw [->] (2) to[bend left=70] node[below] {$c^{-1}$} (1);
\draw [->] (2) to[bend right=70] node[above] {$c^{-1}$} (1);
\draw [->] (2) to[bend left=40] node[above] {$c^{-1}$} (3);
\draw [->] (4) to[bend left=40] node[below] {$\ b^{-1}$} (3);
\end{tikzpicture}
	\caption{An $X$-graph $\Gamma$ and the $X\cup X^{-1}$-graph $\widehat{\Gamma}$ for $X=\{a,b,c\}$. }
	\label{Gdach}
\end{figure}

\begin{definition}{(Path, see \cite[2.2]{KM02})\\}
Let $\Gamma$ be an $X$-graph. A \textit{path} $p$ in $\Gamma$ is, by definition, a sequence of edges $e_1,...,e_k $, where $e_i \in E(\widehat{\Gamma})$ and $o(e_i)=t(e_{i-1})$. The origin of $p$ is $o(p):=o(e_1)$ and its terminus is $t(p):=t(e_k)$. The label of $p$ is, by definition, $\mu(p):=\mu(e_1)\cdots \mu(e_k)$, the word in the free monoid  generated by $X\cup X^{-1}$. We call $p$ a path from $o(e_1)$ to $t(e_k)$. 
 If $v$ is a vertex of $\Gamma$, we consider the sequence $p=v$ to be a path with $o(p)=t(p)=v$ and $\mu(p)=1$ (the empty word). 
\end{definition}

\begin{definition}{(Reduced Word and Reduced Path, see \cite[2.6]{KM02})\\}
A \textit{freely reduced word} in the alphabet $X \cup X^{-1}$ is a word without any subwords $xx^{-1}$ or $x^{-1}x$ for $x \in X$.  
A path $p$ in an $X$-graph $\Gamma$ is said to be \textit{reduced} if $p$ does not contain subpaths of the form $e, e^{-1}$ for $e\in E(\widehat{\Gamma})$.
\end{definition}

Let $w$ be a word in the alphabet $X \cup X^{-1}$ (or  an $X$-word, or a word in $X^{\pm 1}$). We  denote by  $\overline{w}$ the freely reduced $X$-word  obtained by removing $xx^{-1}$, $x^{-1}x$ successively. 
\textit{The free group} on $X$, denoted $F(X)$, is the collection of all freely reduced words in $X^{\pm 1}$. The multiplication in the free group $F(X)$ is defined as $$ f \cdot g:=\overline{fg}$$ 
for all $f, g\in F(X)$.

\begin{definition}{(Language, see \cite[2.7]{KM02})}\label{deflang}\mbox{}\\
Let $\Gamma$ be an $X$-graph and let $v$ be a vertex of $\Gamma$. We define the \textit{language of $\Gamma$ with respect to} $v$ to be 
$$L(\Gamma, v)=\{\mu(p) \mid p \textrm{ is a reduced path in } \Gamma \textrm{ with } o(p)=t(p)=v \}.$$
\end{definition}  
\begin{figure}[h!]
	\centering
  \begin{tikzpicture}     
\coordinate[label=left:$\Gamma$] (0) at (0,0.6);
\coordinate[label=below:$u$] (u) at (0,0);
\coordinate[label=below:$v$] (v) at (2,0); 
\coordinate[label=below:$w$] (w) at (4,0);
%\coordinate[label=below:$\Gamma$] () at (0,1);
 \filldraw[black]
 (u) circle (0.8pt)
 (v) circle (0.8pt)
 (w) circle (0.8pt);
\draw [->] (u) to[bend left=30] node[above] {$x$} (v);
\draw [->] (u) to[bend right=30] node[below] {$x$} (v);
\draw [->] (v) to node[above] {$y$} (w);

\end{tikzpicture}
	\caption{An $\{x,y\}$-graph $\Gamma$. }
	\label{sprache}
\end{figure}
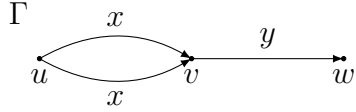
For example, in Figure \ref{sprache} the languages of $\Gamma$ with respect to the vertex $u, v$ or $w$ are: 
	$L(\Gamma, u)=\{ (xx^{-1})^n \mid n \in \mathbb N\}$, 
	$L(\Gamma, v)=\{ (x^{-1}x)^n \mid n \in \mathbb N\}$ and $L(\Gamma, w)=\{ 1, y^{-1}(x^{-1}x)^ny \mid n \in \mathbb N_{>0}\}$.

Note that $\mu(p)$ may have subwords of the form $xx^{-1}$ or $x^{-1}x$ for some $x\in X$ even if $p$ is a reduced path. Hence the words in the language $L(\Gamma, v)$ of an $X$-graph are not necessarily freely reduced.  

\begin{proposition}\textup{(See \cite[3.1]{KM02})\\}
Let $\Gamma$ be an $X$-graph and let $v$ be a base-vertex of $\Gamma$. Then the set 
$$\overline{L(\Gamma, v)}=\{\overline{w} \mid w \in L(\Gamma, v) \}$$
is a subgroup of the free group $F(X)$.
\end{proposition}

As we will see, the language of a folded $X$-graph consists only of freely reduced words.  

\begin{definition}{(Folded $X$-Graph, see \cite[2.3]{KM02})}\label{deffolded}\mbox{}\\
Let $\Gamma$ be an $X$-graph. We say that the $X$-graph $\Gamma $ is \textit{folded} if for each vertex $v$ of $\Gamma$ and each letter $x\in X$ there is at most one edge in $\Gamma$ with origin $v$ and label $x$ and at most one edge with terminus $v$ and label $x$.
\end{definition}

For example, in Figure \ref{folding}, the graphs $\Gamma, \Gamma' $ and $\Gamma''$  are not folded but the graph $ \Gamma'''$ is folded.

Suppose $\Gamma$ is an $X$-graph and $e, e'$ are two edges of $\Gamma$ with $o(e)=o(e')$ (or $t(e)=t(e')$) and the same label $x\in X$. Then, informally speaking, \textit{folding}  $\Gamma$ at $e, e'$ means identifying $e$ and $e'$ in a single new edge with label $x$. 
 For a more precise definition see \cite[2.4]{KM02} and \cite[3.2]{S}.

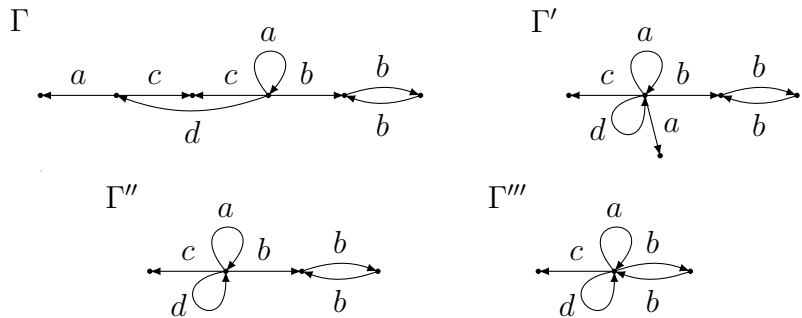
\begin{figure}[h!]
	\centering
  \begin{tikzpicture}     

\coordinate[label=left:$\Gamma$](0) at (0,1);
\coordinate (1) at (0,0);
\coordinate (2) at (1,0);
\coordinate (3) at (2,0);
\coordinate (4) at (3,0);
\coordinate (5) at (4,0);
\coordinate (6) at (5,0);
 \filldraw[black](1) circle (0.8pt)
 (2) circle (0.8pt)
 (3) circle (0.8pt)
 (4) circle (0.8pt)
 (5) circle (0.8pt)
  (6) circle (0.8pt)
  (0,-1) circle (0pt)
 ;
\path[->,min distance=1cm] (4) edge[in=50,out=130,above] node {$a$}(4);
\draw [->] (5) to[bend left=20] node[above] {$b$} (6);
\draw [->] (6) to[bend left=20] node[below] {$b$} (5);
\draw [->] (4) to[bend left=20] node[below] {$d$} (2);
\draw [->] (2) to node[above] {$a$} (1);
\draw [->] (2) to node[above] {$c$} (3);
\draw [->] (4) to node[above] {$c$} (3);
\draw [->] (4) to node[above] {$b$} (5);
\end{tikzpicture}
\hspace{10mm}
\begin{tikzpicture}     

\coordinate[label=left:$\Gamma'$](0) at (0,1);
\coordinate (1) at (0,0);
\coordinate (2) at (1,0);
\coordinate (3) at (2,0);
\coordinate (4) at (3,0);
\coordinate (5) at (1.2,-0.8);

 \filldraw[black](1) circle (0.8pt)
 (2) circle (0.8pt)
 (3) circle (0.8pt)
 (4) circle (0.8pt)
 (5) circle (0.8pt)
 
 ;
\path[->,min distance=1cm] (2) edge[in=50,out=130,above] node {$a$}(2);
\path[->,min distance=1cm] (2) edge[in=270,out=190,left] node {$d$}(2);
\draw [->] (3) to[bend left=20] node[above] {$b$} (4);
\draw [->] (4) to[bend left=20] node[below] {$b$} (3);

\draw [->] (2) to node[above] {$c$} (1);
\draw [->] (2) to node[right] {$a$} (5);
\draw [->] (2) to node[above] {$b$} (3);
\end{tikzpicture}

\begin{tikzpicture}     

\coordinate[label=left:$\Gamma''$](0) at (0,1);
\coordinate (1) at (0,0);
\coordinate (2) at (1,0);
\coordinate (3) at (2,0);
\coordinate (4) at (3,0);

 \filldraw[black](1) circle (0.8pt)
 (2) circle (0.8pt)
 (3) circle (0.8pt)
 (4) circle (0.8pt)

 ;
\path[->,min distance=1cm] (2) edge[in=50,out=130,above] node {$a$}(2);
\path[->,min distance=1cm] (2) edge[in=270,out=190,left] node {$d$}(2);
\draw [->] (3) to[bend left=20] node[above] {$b$} (4);
\draw [->] (4) to[bend left=20] node[below] {$b$} (3);

\draw [->] (2) to node[above] {$c$} (1);
\draw [->] (2) to node[above] {$b$} (3);
\end{tikzpicture}
\hspace{10mm}
\begin{tikzpicture}     

\coordinate[label=left:$\Gamma'''$](0) at (0,1);
\coordinate (1) at (0,0);
\coordinate (2) at (1,0);
\coordinate (3) at (2,0);

 \filldraw[black](1) circle (0.8pt)
 (2) circle (0.8pt)
 (3) circle (0.8pt)

 ;
\path[->,min distance=1cm] (2) edge[in=50,out=130,above] node {$a$}(2);
\path[->,min distance=1cm] (2) edge[in=270,out=190,left] node {$d$}(2);
\draw [->] (3) to[bend left=20] node[below] {$b$} (2);
\draw [->] (2) to[bend left=20] node[above] {$b$} (3);
\draw [->] (2) to node[above] {$c$} (1);

\end{tikzpicture}
	\caption{Folding of an $\{a,b,c,d\}$-graph $\Gamma$. Every step $\Gamma \dashrightarrow \Gamma'$, $\Gamma' \dashrightarrow \Gamma''$, and $\Gamma'' \dashrightarrow \Gamma'''$ is a folding.}
	\label{folding}
\end{figure}

With the definition of a folded $X$-graph at hand we may associate  a subgroup to the language of an $X$-graph.

\begin{lemma}\textup{(See \cite[2.9]{KM02})\\}
Let $\Gamma$ be a folded $X$-graph and $v$ be a vertex of $\Gamma$. Then all the words in the language $L(\Gamma, v)$ are freely reduced. 
\end{lemma}

\begin{cor}\textup{(See \cite[3.2]{KM02})}\label{cor1} \mbox{}\\
Suppose $\Gamma$ is a folded $X$-graph. Then $L(\Gamma, v)=\overline{L(\Gamma, v)}$ is a subgroup of the free group $F(X)$.
\end{cor}

To each folded based $X$-graph $(\Gamma, v)$ we have thus associated a subgroup of the free group by considering the language $L(\Gamma, v)$ of the $X$-graph. 
But two different folded based $X$-graphs  can have the same language. For example, each of these four  based graphs $(\Gamma,u)$, $(\Gamma,v)$, $(\Gamma',v')$  and $(\Gamma',u')$ in Figure \ref{core} has language $\{x^{2z} \mid z \in \Z\}$. 

\begin{figure}[h!]
	\centering
  \begin{tikzpicture}     
\coordinate[label=left:$\Gamma$] (0) at (0,0.6);
\coordinate[label=below:$u$] (u) at (0,0);
\coordinate[label=below:$v$] (v) at (2,0); 
\coordinate[label=below:$w$] (w) at (4,0);
%\coordinate[label=below:$\Gamma$] () at (0,1);
 \filldraw[black]
 (u) circle (0.8pt)
 (v) circle (0.8pt)
 (w) circle (0.8pt);
\draw [->] (v) to[bend right=30] node[above] {$x$} (u);
\draw [->] (u) to[bend right=30] node[below] {$x$} (v);
\draw [->] (v) to node[above] {$y$} (w);

\end{tikzpicture}
\hspace{1cm}
\begin{tikzpicture}     
\coordinate[label=left:$\Gamma'$] (0) at (0,0.6);
\coordinate[label=below:$u'$] (u) at (0,0);
\coordinate[label=below:$v'$] (v) at (2,0); 

%\coordinate[label=below:$\Gamma$] () at (0,1);
 \filldraw[black]
 (u) circle (0.8pt)
 (v) circle (0.8pt);
\draw [->] (v) to[bend right=30] node[above] {$x$} (u);
\draw [->] (u) to[bend right=30] node[below] {$x$} (v);

\end{tikzpicture}
	\caption{An $X$-graph $\Gamma$ and an $X$-graph $\Gamma'=Core(\Gamma,v)$ for $X=\{x, y\}$.}
	\label{core}
\end{figure}
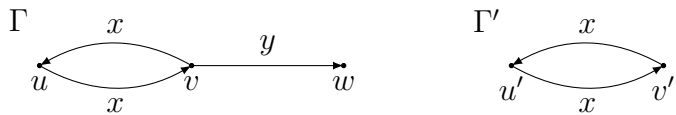

\begin{definition}{(Core Graph, see \cite[3.5]{KM02})}\label{defcore} \mbox{}\\
Let $\Gamma$ be an $X$-graph and let $v$ be a vertex. Then the \textit{core of  $\Gamma$ at} $v$ is defined as 
$$Core(\Gamma, v)= \bigcup \{ p \mid p \textrm{ is a reduced path in } \Gamma \textrm{ with } o(p)=t(p)=v \}.$$
If $Core(\Gamma, v)=\Gamma$ we say that $\Gamma$ is a \textit{core graph with respect to $v$}. 
\end{definition}

For example, the graph $\Gamma$ in Figure \ref{core} is a core graph only with respect to $w$. For the vertex $u$ and $v$ we have that $Core(\Gamma, u)=Core(\Gamma, v)= \Gamma'$.

\begin{theorem}\textup{(Subgroup Graph, see \cite[5.3]{KM02})}\label{thm}\mbox{}\\
Let $H\leq F(X)$ be a subgroup. 
There exists a  based $X$-graph $(\Gamma, v)$ (unique up to canonical isomorphism of based $X$-graphs) such that
\begin{enumerate}[(i)]
\item 
the graph $\Gamma$ is folded and connected;
\item 
the graph $\Gamma$ is a core graph with respect to $v$;
\item 
$L(\Gamma, v)=H$.
\end{enumerate}
In this situation we call $\Gamma$ the subgroup graph of $H$ with respect to $X$ and denote it by $\Gamma_X(H)$ or  briefly by $\Gamma(H)$. The base-vertex $v$ is denoted $1_H$. 
\end{theorem}

\begin{remark}
The subgroup graph $\Gamma_X(H)$ is the core of the Schreier coset graph of $H$ with respect to $X$ and $F(X)$,  see the proof of Theorem 5.1 \cite{KM02}. 
\end{remark}

By Theorem \ref{thm}, the graph $\Gamma(H)$ is unique (up to isomorphism). Consequently, if $(\Gamma,v)$ is a folded connected core $X$-graph with language $L(\Gamma,v)=H$, then $(\Gamma,v)\cong(\Gamma(H),1_H)$.

\begin{definition}\textup{(Schreier Coset Graph)}\label{def1} \mbox{}\\
Let $G$ be a group with a finite generating set $X$. Let $H$ be a subgroup of $G$. Let $\Gamma$ be the following $X$-graph. The vertex set of $\Gamma$ is the set of right cosets of the subgroup $H$ in $G$. For two cosets $Hg$ and $Hg'$ and each letter $x\in X$ we introduce a directed edge with  origin $Hg$, terminus $Hg'$ and label $x$ whenever $Hgx=Hg'$. This graph is called the \textit{Schreier coset graph} of $H$ with respect to $X$ and $G$.
\end{definition}

There are three important classes of subgroups of the free group $F(X)$: finite index subgroups, finitely generated subgroups and infinitely generated ones. In their article \cite{KM02},  Kapovich and Myasnikov show that each of these three classes have subgroup graphs with specific properties.

\begin{lemma}\textup{(See \cite[5.4]{KM02})\\}
For a subgroup $H\leq F(X)$ the subgroup graph $\Gamma(H)$ is finite if and only if $H$ is finitely generated. 
\end{lemma}

The next notion distinguishes finitely generated subgroups of $F(X)$ from those of finite index (which are also finitely generated). 

\begin{definition}{(Regular $X$-Graph, see  \cite[8.1]{KM02})}\label{defregular}\mbox{}\\
An $X$-graph $\Gamma$ is said to be \textit{$X$-regular}  if for every vertex $v$ of $\Gamma$ and every $x$ in $X\cup X^{-1}$ there is exactly one edge in $\widehat{\Gamma}$ with origin $v$ and label $x$.

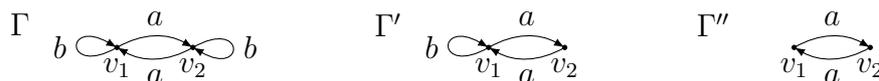
\begin{figure}[!h]
	\centering
  \begin{tikzpicture}     
\coordinate[label=left:$ \Gamma$](0) at (0,0.3);
\coordinate[label=below:$v_1$] (2) at (1,0);
\coordinate [label=below:$v_2$](3) at (2,0);

 \filldraw[black]
 (2) circle (0.8pt)
 (3) circle (0.8pt)

 ;
 \path[->,min distance=8mm] (2) edge[in=150,out=210,left] node {$b$}(2);
\path[->,min distance=8mm] (3) edge[in=330,out=30,right] node {$b$}(3);
%\path[->,min distance=1cm] (2) edge[in=130,out=50,above] node {$a$}(2);
%\path[->,min distance=1cm] (3) edge[in=50,out=130,above] node {$a$}(3);
\draw [->] (3) to[bend left=30] node[below] {$a$} (2);
\draw [->] (2) to[bend left=30] node[above] {$a$} (3);
\end{tikzpicture}
\hspace{10mm}
\begin{tikzpicture}     
\coordinate[label=left:$ \Gamma'$](0) at (0,0.3);
\coordinate[label=below:$v_1$] (2) at (1,0);
\coordinate [label=below:$v_2$](3) at (2,0);

 \filldraw[black]
 (2) circle (0.8pt)
 (3) circle (0.8pt)

 ;
\path[->,min distance=8mm] (2) edge[in=150,out=210,left] node {$b$}(2);

\draw [->] (3) to[bend left=30] node[below] {$a$} (2);
\draw [->] (2) to[bend left=30] node[above] {$a$} (3);
\end{tikzpicture}
\hspace{10mm}
\begin{tikzpicture}     
\coordinate[label=left:$\Gamma''$](1) at (0.3,0.3);
\coordinate[label=below:$v_1$] (2) at (1,0);
\coordinate [label=below:$v_2$](3) at (2,0);

 \filldraw[black]
 (2) circle (0.8pt)
 (3) circle (0.8pt)
 ;
\draw [->] (3) to[bend left=30] node[below] {$a$} (2);
\draw [->] (2) to[bend left=30] node[above] {$a$} (3);

\end{tikzpicture}
	\caption{$\{a,b\}$-regular and not $\{a,b\}$-regular graphs.}
	\label{regular}
\end{figure}
\end{definition}

Figure \ref{regular} contains the following examples. 
The graph $\Gamma$ is $\{a,b\}$-regular, the graph $\Gamma'$ and $\Gamma''$ are not $\{a,b\}$-regular and the graph $\Gamma''$ is $\{a\}$-regular. 

We can reformulate the notion of being $X$-regular. An $X$-graph $\Gamma$ is $X$-regular if for every vertex $v$ of $\Gamma$ and every $x \in X$ there is exactly one edge with label $x$ and origin $v$ and exactly one edge with label $x$ and terminus $v$. 

By the above reformulation of $X$-regularity and the definition
of a folded $X$-graph, we obtain:

\begin{lemma}\label{lem1.1} \mbox{}\\
An $X$-regular graph is folded. 
\end{lemma}

\begin{lemma}\label{lem1.2} \mbox{}\\
A connected finite $X$-regular graph $\Gamma$  is a core graph with respect to every vertex of $V(\Gamma)$. 
\end{lemma}

The subgroup graph $\Gamma(H)$ of $H$ is the core of the Schreier coset graph of $H$ with respect to $X$ and $G$. Thus the subgroup graph $\Gamma(H)$ is the Schreier coset graph of $H$ if and only if the Schreier coset graph is a core graph. By Lemma \ref{lem1.2}, this is the case if the Schreier coset graph of $H$ is finite thus $H$ has finite index in $G$.

\begin{proposition}\textup{(See \cite[8.3]{KM02})} \label{prop1} \mbox{}\\
Let $H$ be a subgroup of the free group $F(X)$. Then the index $[F(X):H]$ is finite if and only if the subgroup graph $\Gamma(H)$ is a finite $X$-regular graph. 
In this case $[F(X):H]=|V(\Gamma(H))|$.
\end{proposition}

%%%%%%%%%%%%%%%%%%%%%%%%%%%%%%%%%%%%%%%%%%%%%%%%%%%%%%%%%%%%%%%%%%%%%%%%%%%%%%%%%%%

\section{Subgroup graphs of finite index subgroups of finitely generated groups}
\label{ssub}

%%%%%%%%%%%%%%%%%%%%%%%%%%%%%%%%%%%%%%%%%%%%%%%%%%%%%%%%%%%%%%%%%%%%%%%%%%%

In this section we generalize the theory of subgroup graphs of subgroups of free groups, as described in Section \ref{subfree},  to finite index subgroups of finitely generated groups.  
Suppose that $G$ is a finitely generated group with a presentation $\langle \,X\, |\, R\, \rangle$, where  $X$ is finite and $R$ is not necessarily finite. To  a finite connected based $X$-regular graph $(\Gamma,v)$ which fulfills the defining relators $R$ we associate a finite index subgroup $H\leq G$, by taking $H:=\phi(L(\Gamma, v))$ (see Definitions \ref{defregular}, \ref{deffulfil}).  
Conversely, for every finite index subgroup $H\leq G$ there exists a  finite connected $X$-regular based graph $(\Gamma',v')$, unique up to isomorphism, which fulfills the defining relators $R$ and has $\phi(L(\Gamma',v'))=H$.  
The correspondence is unique up to isomorphism of based $X$-graphs. 
Therefore we call $(\Gamma',v')$ the subgroup graph $\Gamma_{X,R}(H)$ of $H\leq G$.
It is the Schreier coset graph of $H$ with respect to $X$ and $G$.   
Notice that the subgroup graph $\Gamma_{X,R}(H)$ depends on the presentation of the group $G$. Therefore the subgroup graphs $\Gamma_{X,R}(H)$ and $\Gamma_{X',R'}(H)$ of $H$ may not be isomorphic for different generators $X$, $X'$ and relators $R$, $R'$. However, the number of vertices $|V(\Gamma_{X,R}(H))|=[G:H]$ is independent of $X$ and $R$.\\

For the rest of the article, let us fix some notation: Let $G=\langle \,X\, |\, R\, \rangle$ be a finitely generated group, where $X$ is finite and $R$ is a subset of words of the free group $F(X)$, not necessarily finite. We denote the normal closure of $R$ in $F(X)$ by $N:= \langle \! \langle\, R\, \rangle \! \rangle_{F(X)}=\left\langle\, \overline{wrw^{-1}} \, |\, w\in F(X), r\in R\, \right\rangle$. Let $\phi\colon F(X) \rightarrow G$ be the canonical epimorphism such that $G \cong F(X)/N$. 

Our aim is to develop a theory of subgroup graphs of finite index subgroups of finitely generated groups.  A natural way to define a subgroup graph of $H\leq G$ is to use the subgroup graph $\Gamma_X(H')$ of  $H'=\phi^{-1}(H)\leq F(X)$.  Then $\phi(L(\Gamma_X(H'),1_{H'}))=H$ and $N\leq H'$. Moreover, if $[F(X):H']=n$, then $[G:H]=n$. Since $H'$ is a finite index subgroup in $F(X)$, the subgroup graph $\Gamma_X(H')$ is finite, connected and $X$-regular. If $N\leq L(\Gamma,v)\leq F(X)$, then $L(\Gamma,v)=\phi^{-1}(H)$. Consequently, to determine the subgroup graph $\Gamma(H)$, we need a condition such that $N\leq L(\Gamma(H),v)$.

\begin{definition}\textup{(Fulfilling $X$-Graph)}\label{deffulfil}\mbox{}\\
Let $G=\langle \,X\, |\, R\, \rangle $ be a presentation of a group, where $X$ is finite and $R$ is not necessarily finite. We say that an $X$-graph $\Gamma$ \textit{fulfills the defining relators $R$} if for all vertices $v \in V(\Gamma)$ the following holds: 
   if $p_r$ is a reduced path with origin $v$ and label $r \in R$, then the terminus of $p_r$ is $v$. 
\end{definition}

 \begin{example}
 We consider the graphs of Figure  \ref{fulfil}. 
 The $\{x, y\}$-regular graph $\Gamma$ fulfills the relators $x^2$, $y^2$ and  $(xy)^3$. Indeed, the reduced paths with label $x^2$ are: $v_1 \xrightarrow{x} v_2 \xrightarrow{x} v_1$; $v_2 \xrightarrow{x} v_1 \xrightarrow{x} v_2$; $v_3\xrightarrow{x} v_3\xrightarrow{x} v_3$. For the relator  $(xy)^3$, an example of a reduced path with that label is $v_1 \xrightarrow{x} v_2 \xrightarrow{y}  v_3 \xrightarrow{x} v_3 \xrightarrow{y} v_2 \xrightarrow{x} v_1 \xrightarrow{y}  v_1$. Notice that the graph $\Gamma$ does not fulfill the relator $(xy)^2$. In fact, the reduced path with label $(xy)^2$ and origin $v_1$ has terminus $v_2$. 
 
 Instead, the $\{x, y\}$-regular graph $\Gamma'$ fulfills the relator $(xy)^2$. The reduced paths with label $(xy)^2$ are: $u_1 \xrightarrow{x} u_2 \xrightarrow{y}  u_1 \xrightarrow{x} u_2 \xrightarrow{y}  u_1$; $u_2\xrightarrow{x} u_3 \xrightarrow{y}  u_3 \xrightarrow{x} u_4 \xrightarrow{y}  u_2$; $u_3 \xrightarrow{x} u_4 \xrightarrow{y}  u_2 \xrightarrow{x} u_3 \xrightarrow{y}  u_3$; $u_4 \xrightarrow{x} u_1\xrightarrow{y} u_4 \xrightarrow{x} u_1 \xrightarrow{y} u_4$. The graph $\Gamma'$ also fulfills the relators $x^4$ and $y^3$. 
 
 Note that if a graph fulfills a relator $w^n$, then it also fulfills the relator $w^{nz}$ for $z\in \Z$.  
 \end{example}
 
 \begin{figure}[h!]
\centering
       
\begin{tikzpicture}     

\coordinate[label=below:$v_1$] (u) at (0,0);
\coordinate[label=below:$v_2$] (v) at (1.5,0); 
\coordinate[label=below:$v_3$] (w) at (3,0);
\coordinate[label=above:$\Gamma$] () at (0,1);
 \filldraw[black]
 (u) circle (0.8pt)
 (v) circle (0.8pt)
 (w) circle (0.8pt);
\draw [->] (u) to[bend left=30] node[above] {$x$} (v);
\draw [->] (v) to[bend left=30] node[below] {$x$} (u);
\draw [->] (w) to[bend left=30] node[below] {$y$} (v);
\draw [->] (v) to[bend left=30] node[above] {$y$} (w);
\path[->,min distance=1cm] (u) edge[in=140,out=220,left] node {$y$}(u);
\path[->,min distance=1cm] (w) edge[in=40,out=320,right] node {$x$}(w);

\end{tikzpicture}
\hspace{10mm}
\begin{tikzpicture}     
\coordinate[label=below:$u_1$] (1) at (0,0);
\coordinate[label=below:$u_4$] (2) at (1.5,0); 
\coordinate[label=above:$u_3$] (3) at (1.5,1.5);
\coordinate[label=above:$u_2$] (4) at (0,1.5);
\coordinate[label=above:$\Gamma'$] () at (-1,1.5);
 \filldraw[black](1) circle (0.8pt)
 (2) circle (0.8pt)
 (3) circle (0.8pt)
 (4) circle (0.8pt);
\draw [->] (1) to node[above] {$y$} (2);
\draw [->] (2) to node[right] {$y$} (4);
\draw [->] (4) to node[right] {$y$} (1);
\path[->,min distance=1cm] (3) edge[in=320,out=30,right] node {$y$}(3);
\draw [->] (2) to[bend left=30] node[below] {$x$} (1);
\draw [->] (3) to[bend left=30] node[right] {$x$} (2);
\draw [->] (4) to[bend left=30] node[above] {$x$} (3);
\draw [->] (1) to[bend left=30] node[left] {$x$} (4);
\end{tikzpicture}
	\caption{Two $\{a,b\}$-graphs which fulfill different relators.} 
	\label{fulfil}
\end{figure}
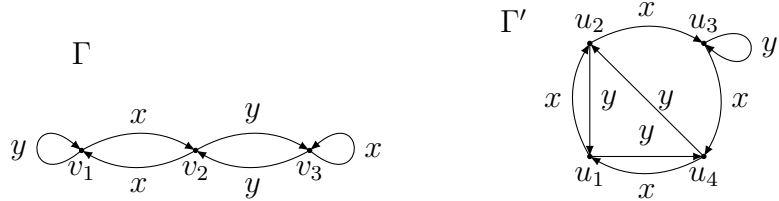

\begin{remark}
Let $p_v$ be a reduced path from $1_H$ to $v$ in $\Gamma_X(H)$ and $\mu (p_v)=g_v$. 
The subgroup graph $\Gamma_X(H)$ of a finite index subgroup $H$ of the free group $F(X)$ is the Schreier coset graph of $H$. Therefore $H\backslash F(X)=\{Hg_v \mid v\in V(\Gamma_X(H))\}$. Moreover,
$$Hg_v= \{\mu(p)\ \mid p \textrm{ is a reduced path with } o(p)=1_H \textrm{ and } t(p)=v \}.$$ 
\end{remark}

Using the above remark we can prove the following proposition. 

\begin{proposition}\label{prop21}\mbox{}\\ 
Let $\Gamma$ be a finite $X$-regular connected graph with base-vertex $v_0$. Then 
$\Gamma $ fulfills the defining relators $R$ if and only if $\langle\!\langle\,  R\, \rangle\!\rangle_{F(X)} \leq L(\Gamma, v_0)$.
\end{proposition}

\begin{proof} 
Since $\Gamma$ is a finite connected $X$-regular graph, $\Gamma$ is also folded and a core graph, by Lemmas \ref{lem1.1} and \ref{lem1.2}. By Corollary \ref{cor1}, the language $H=L(\Gamma, v)$ is a subgroup of $F(X)$. Following Theorem \ref{thm}, $\Gamma$ is the subgroup graph $\Gamma_X(H)$ of $H$, where $v_0=1_H$. 
By Proposition \ref{prop1}, the subgroup $H$ has finite index in $F(X)$ and $\Gamma_X (H)$ is a Schreier coset graph. 

Assume that $\Gamma$ fulfills the relators $R$. Let $w$ be in $F(X)$.  Since $\Gamma$ is $X$-regular, there exists a unique reduced path $p_w$ with $\mu(p_w)=w$ and $o(p_w)=v_0$. Let $v$ be the terminus of the path $p_w$. 
Since $\Gamma $ fulfills the defining relators $R$, the reduced path $p_r$ with $\mu(p_r)=r\in R$ and $o(p_r)=v$ has terminus $v$. 
The path $p_w^{-1}$ with $\mu(p_w^{-1})=w^{-1}$ and $o(p_w^{-1})=v$ ends in $v_0$. Consequently, the path $p_wp_rp_w^{-1}$ is a path from $v_0$ to $v_0$, which may not be reduced. Let $p$ be the path which we obtain by deleting successively all subpaths $ee^{-1}$ from $p_wp_rp_w^{-1}$, where $e\in E(\widehat{\Gamma})$. Then the path $p$ is reduced. By definition $\mu(p)\in L(\Gamma, v_0)$. Since the graph $\Gamma$ is $X$-regular, $\mu(p)$ is a freely reduced word. Therefore $\mu(p)=\overline{wrw^{-1}}$. Analogously, there exists a reduced path with label $\overline{wrw^{-1}}$ from $v$ to $v$ for each $v\in V(\Gamma)$.  
Moreover, there exists a reduced path for every combination of paths  $v_0\xrightarrow{wrw^{-1}} v_0$ or paths $v_0\xrightarrow{u} v \xrightarrow{wrw^{-1}w'r'w'^{-1}...} v \xrightarrow{u^{-1}} v_0$.  Therefore  $\langle\!\langle\,  R\, \rangle\!\rangle_{F(X)} \leq L(\Gamma, v_0)$.

Assume now that $\left\langle\!\left\langle\,  R\, \right\rangle\!\right\rangle_{F(X)} \leq L(\Gamma, v_0)=H\leq F(X)$ and let $v$ be a vertex of $\Gamma$ and $r\in R$. Since $\Gamma$ is $X$-regular, there  exists a unique reduced path $p_r$ with $\mu(p_r)=r$ and $o(p_r)=v$. Let $v'$ be the terminus of $p_r$. 
Let $w$ be the label of a reduced path from $v_0$ to $v$. The graph $\Gamma$ is a Schreier coset graph. Thus the vertex $v$ is the coset $Hw$ and the vertex $v'$ is the coset $Hw'$ with $w':=\overline{wr}$. 
By assumption, we have $\overline{wrw^{-1}}\in H$, hence $H\overline{wrw^{-1}}=H$ which is equivalent to $Hw'=H\overline{wr}=H\overline w =Hw$, hence $v=v'$. 
\end{proof}

Finally, we can state the main theorem of this section. Part (\ref{thm2_(1)}) is the analogue to Corollary \ref{cor1} and part (\ref{thm2_(2)}) the analogue to Theorem \ref{thm}.

\begin{theorem}\textup{(Subgroup Graph)}\label{thm2}\mbox{}\\
Let $G$ be a group with a presentation $G=\langle \,X\, |\, R\, \rangle$, where $X$ is finite and $R$ is not necessarily finite.
\begin{enumerate}[(1)]
\item 
\label{thm2_(1)}
Let $\Gamma$ be an $X$-regular connected graph with $n$ vertices.  Let $v_0$ be a base-vertex of $\Gamma$. Assume that $\Gamma$ fulfills the defining relators $R$. 
Then $\phi(L(\Gamma,v_0))$ is a subgroup of $G$ of index $[G:\phi(L(\Gamma,v_0))]=n$.
\item
\label{thm2_(2)}
 Let $H\leq G$ be a subgroup of index $[G:H]=n\in \N$. Then there exists a based $X$-graph $(\Gamma,v_0)$ (unique up to a canonical isomorphism of based $X$-graphs)  such that
\begin{enumerate}[(i)]
\item 
$\Gamma$ is $X$-regular and connected;
\item $\Gamma$ fulfills the defining relators $R$; 
\item $\Gamma$ has $n$ vertices; 
\item  $\phi(L(\Gamma, v_0))=H$.
\end{enumerate}
\end{enumerate}

In this situation we call $\Gamma$ the subgroup graph of $H$ with respect to $X$ and $R$. We denote it by $\Gamma_{X,R}(H)$ or briefly by $\Gamma(H)$. The base-vertex $v_0$ is denoted by $1_H$.
In fact, $\Gamma_{X,R}(H)=\Gamma_X(H')$, where $H':=L(\Gamma_{X,R}(H),1_H)\leq F(X)$.
\end{theorem}

\begin{proof}
Let us prove part (\ref{thm2_(1)}). 
The graph $(\Gamma, v_0)$ is the subgroup graph of the subgroup $ L(\Gamma, v_0)=:H'\leq F(X)$. Since $\Gamma$ is $X$-regular and has $n$ vertices, $H'$ is a subgroup of index $n$ in $F(X)$. The graph $\Gamma$ fulfills the defining relators $R$. Thus  $N \leq H'$, and consequently $[G: \phi(H')]=[F(X):H']=n$.

Let us prove  part (\ref{thm2_(2)}).
Let $\phi^{-1}(H)=:H'\leq F(X)$. By Theorem \ref{thm}, there is a  subgroup graph $\Gamma_X(H')$ for $H'$ (unique up to isomorphism of based $X$-graphs) which is folded, connected, a core graph and has $L(\Gamma_X(H'), 1_{ H'})= H'$. Thus  $\phi(L(\Gamma_X(H'),1_{H'}))=\phi(H')=\phi(\phi^{-1}(H))=H$.  But $[G:H]=[F(X):H']$,
since $N =\phi^{-1} (1_G) \leq \phi^{-1}(H)= H'$. By Proposition \ref{prop1},  $\Gamma_X ( H')$ is an $X$-regular graph with $n$ vertices. By Proposition \ref{prop21}, $\Gamma$ fulfills the defining relators $R$. Thus $(\Gamma_{X,R}(H),1_H)=(\Gamma,v_0)\cong (\Gamma_X(H'),1_{H'})$.
\end{proof}

In Section \ref{secconnect} and \ref{secsuf} we use part (\ref{thm2_(1)}) of the theorem above. We construct an $X$-graph satisfying the properties of (\ref{thm2_(1)}) which gives us a finite index subgroup. 
By the proof of the Theorem \cite[5.1]{KM02}, we see that the subgroup graph $\Gamma(H)$ is the Schreier coset graph of $H$ with respect to $X$ and $G$.

\begin{remark} 
As $\Gamma(H)$ is the Schreier coset graph for the finite index subgroup $H\leq G$, we obtain the cosets of $H\backslash G$ from the subgroup graph $\Gamma(H)$. 
Let $p_v$ be a reduced path from $1_H$ to $v\in V(\Gamma(H))$ and $\phi(\mu(p_v))=g_v \in G$. Then 
$$Hg_v= \{\phi(\mu(p)) \mid p \textrm{ is a reduced path with } o(p)=1_H \textrm{ and } t(p)=v \}.$$ 
\end{remark}

We end this section with examples of subgroup graphs of finite index subgroups. As a finite group has only finite index subgroups, our theory gives us all subgroups of a finite group.

\begin{example}
Figure \ref{S3} shows all finite $X$-regular graphs which fulfill the defining relators of the presentation $\left\langle\, s_1,s_2 \, | \, s_1^2, s_2^2, (s_1s_2)^3 \, \right\rangle$ of the symmetric group $S_3$. 

The graph $\Gamma$ is the subgroup graph $\Gamma(S_3)$ of the symmetric group $S_3$. 
The language of $\Gamma'$ is the same for both of its vertices and $\phi(L(\Gamma',v_1))=\left\langle s_1s_2 \right\rangle=A_3$. Therefore $\Gamma(A_3)=\Gamma'$ is the subgroup graph of $A_3$. 
The graph $\Gamma''$ gives us three different subgroups:
 $H_1=\phi(L(\Gamma'',v_1))=\left\langle s_1 \right\rangle$,
  $H_2=\phi(L(\Gamma'',v_2))=\left\langle s_1s_2s_1 \right\rangle$    and
   $H_3=\phi(L(\Gamma'',v_3))=\left\langle s_2 \right\rangle$ in $S_3$. 
   Hence the subgroup graphs of $H_1, H_2$ and $H_3$ are $(\Gamma(H_i),1_{H_i})=(\Gamma'',v_i)$. 
   The language of the graph $\Gamma'''$ is the same for all its vertices.  Thus the graph $\Gamma'''$ is the subgroup graph $\Gamma(\{1_{S_3}\})$ of the trivial group. 
\end{example}

  \begin{figure}[h!]
\centering
  \begin{tikzpicture}     
\coordinate[label=left:$ \Gamma$](0) at (0,1);
\coordinate[label=below:$v_1$] (1) at (0,0);

 \filldraw[black]
 (1) circle (0.8pt)
 ;
\path[->,min distance=1cm] (1) edge[in=140,out=220,left] node {$s_1$}(1);
\path[->,min distance=1cm] (1) edge[in=40,out=320,right] node {$s_2$}(1);
\end{tikzpicture}
\hspace{20mm}
\begin{tikzpicture}     
\coordinate[label=left:$ \Gamma'$](0) at (0,0.8);
\coordinate[label=left:$v_1$] (1) at (0,0);
\coordinate [label=right:$v_2$](2) at (1.5,0);
\coordinate [label=below:](4) at (4,0);

 \filldraw[black]
 (2) circle (0.8pt)
 (1) circle (0.8pt)

 ;
\draw [->] (2) to[bend left=70] node[below] {$s_2$} (1);
\draw [->] (1) to[bend left=70] node[above] {$s_1$} (2);
\draw [->] (2) to[bend right=10] node[below] {$s_2$} (1);
\draw [->] (1) to[bend right=10] node[above] {$s_1$} (2);
\end{tikzpicture}

\vspace{8mm}
\begin{tikzpicture}     
\coordinate[label=left:$ \Gamma''$](0) at (0,1);
\coordinate[label=below:$v_1$] (1) at (0,0);
\coordinate [label=below:$v_2$](2) at (1,0);
\coordinate [label=below:$v_3$](3) at (2,0);
\coordinate [label=below:](4) at (1,-1);
 \filldraw[black]
 (2) circle (0.8pt)
 (1) circle (0.8pt)
 (3) circle (0.8pt)

 ;
 \path[->,min distance=1cm] (1) edge[in=140,out=220,left] node {$s_1$}(1);
\path[->,min distance=1cm] (3) edge[in=40,out=320,right] node {$s_2$}(3);
\draw [->] (3) to[bend left=20] node[below] {$s_1$} (2);
\draw [->] (2) to[bend left=20] node[above] {$s_1$} (3);
\draw [->] (2) to[bend left=20] node[below] {$s_2$} (1);
\draw [->] (1) to[bend left=20] node[above] {$s_2$} (2);
\end{tikzpicture}
\hspace{10mm}
\begin{tikzpicture}
\coordinate[label=left:$ \Gamma'''$](0) at (-2,1.8);
\coordinate[label=left:$v_1$](4) at (180:1.5);
\coordinate[label=above left:$v_2$](5) at (120:1.5); 
\coordinate[label=above right:$v_3$](6) at (60:1.5);
\coordinate[label=right:$v_4$](7) at (0:1.5);
\coordinate[label=below right:$v_5$](8) at (300:1.5); 
\coordinate[label=below left:$v_6$](9) at (240:1.5);
 \filldraw[black](4) circle (0.8pt)
 (5) circle (0.8pt)
 (6) circle (0.8pt)
 (7) circle (0.8pt)
 (8) circle (0.8pt)
 (9) circle (0.8pt)
;
\draw [->]  (4) to[bend left=25] node[left] {$s_1$} (5);

\draw [->]  (4) to[bend left=25] node[right] {$s_2$} (9);
\draw [->]  (6) to[bend left=25] node[below] {$s_2$} (5);

\draw [->]  (6) to[bend left=25] node[right] {$s_1$} (7);
\draw [->]  (8) to[bend left=25] node[left] {$s_2$} (7);
\draw [->]  (8) to[bend left=25] node[below] {$s_1$} (9);

\draw [->]  (5) to[bend left=25] node[right] {$s_1$} (4);

\draw [->]  (9) to[bend left=25] node[left] {$s_2$} (4);
\draw [->]  (5) to[bend left=25] node[above] {$s_2$} (6);

\draw [->]  (7) to[bend left=25] node[left] {$s_1$} (6);
\draw [->]  (7) to[bend left=25] node[right] {$s_2$} (8);
\draw [->]  (9) to[bend left=25] node[above] {$s_1$} (8);

\end{tikzpicture}
	\caption{All $\{s_1, s_2\}$-regular graphs which fulfill the defining relators of the presentation $\left\langle\, s_1,s_2 \, | \, s_1^2, s_2^2, (s_1s_2)^3 \, \right\rangle$ of the symmetric group $S_3$. }
	\label{S3}
\end{figure}
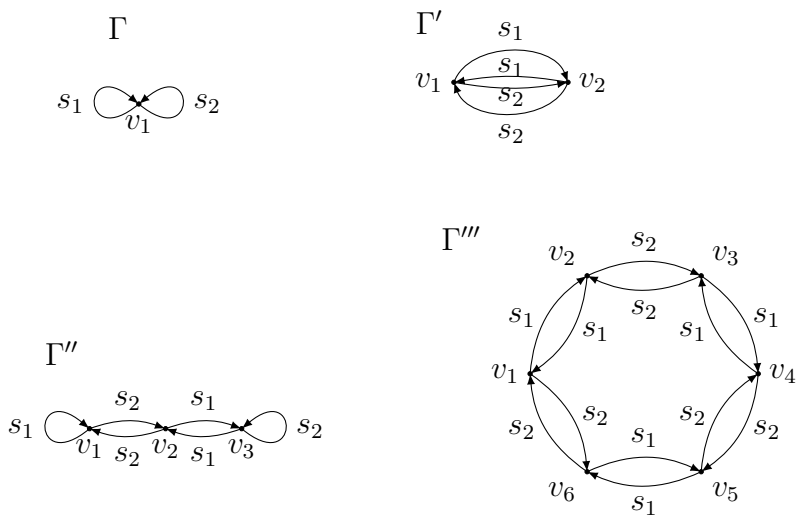

%%%%%%%%%%%%%%%%%%%%%%%%%%%%%%%%%%%%%%%%%%%%%%%%%%%%%

 \section{Applications of subgroup graphs of finite index subgroups}
 \label{appsub}
 
 %%%%%%%%%%%%%%%%%%%%%%%%%%%%%%%%%%%%%%%%%

 In this section we extend the results of \cite{KM02} to applications for finite index subgroups of finitely generated groups. 
 Furthermore, we add some results.
 Subsection \ref{sec:HallAndSylowSubgroups} uses the subgroup graph to detect a Hall subgroup of a finite group.
In Subsection \ref{sec:Mor}  we use the subgroup graphs of two finite index subgroups to determine if one is a subgroup of the other. 
 Subsection \ref{sec:GeneratingSystems}  provides a generating system for a finite index subgroup using its subgroup graph. 
 Subsection \ref{sec:ConjugacyClass} shows that $\{\phi(L(\Gamma(H),v)) \mid v\in \Gamma(H) \}$ is the conjugacy class of the finite index subgroup $H\leq G$.   
 In Subsection \ref{sec:NormalSubgroup} we prove that if $(\Gamma(H),1_H)\cong(\Gamma(H),v)$ for all $v\in V(\Gamma(H))$, then the finite index subgroup $H<G$ is normal in $G$. 
 Moreover, we provide the normalizer of a finite index subgroup. 
 Subsection \ref{sec:IntersectionOfSubgroups} proves that the connected component of $1_H \times 1_K$ of the product graph $\Gamma(H) \times \Gamma(K)$ is the subgroup graph of the intersection $H\cap K$ of two finite index subgroups $H$, $K$ of $G$. 
 Furthermore, $\Gamma(H)\times \Gamma(K)$ provides the intersection $Hg \cap Kg'$ for all $g,g'\in G$. 
 In Subsection \ref{sec:Malnormal} we prove that a subgroup $H$ of a finite group is malnormal if and only if $L(\Gamma(H)\times \Gamma(H), v\times u)=N$ for all $u\times v$ not in the connected component of $1_H\times 1_K$.  
 
Note that $F(X)=\langle \,X\, |\, \emptyset \, \rangle$ with $X$ finite is a presentation as required in all theorems for infinite groups. Consequently, every theorem for infinite groups in this article holds for the free group $F(X)$.

\subsection{Hall and Sylow subgroups}
\label{sec:HallAndSylowSubgroups}

First, we consider Hall subgroups. A \textit{Hall subgroup} of a finite group $G$ is a subgroup $H$ whose order is coprime to its index $[G:H]$. All Sylow subgroups are Hall subgroups.

\begin{proposition}\textup{(Hall and Sylow Subgroups)\\}
 Let $G =\langle \,X\, |\, R\, \rangle$ be a finite group of order $|G|=n$. There exists a Hall subgroup $H$ of order $d$ (hence $d$ and $\frac{n} {d}$ are coprime) if and only if there exists a connected $X$-regular graph $\Gamma$ which fulfills the defining relators $R$ and has $m=\frac{n}{d}$ vertices. 
\end{proposition}

\begin{proof}
Let $H$ be a Hall subgroup with $|H|=d$ and $\Gamma(H)$ its subgroup graph. Then put $\Gamma:= \Gamma_{X,R}(H)$. Hence $|V(\Gamma )|=[G:H]=\frac{n}{d}$.

 Let $v\in V(\Gamma)$. By Theorem \ref{thm2}\,(1), the graph $(\Gamma,v)$ is a subgroup graph of the subgroup $\phi(L(\Gamma,v))=H$ of order $d$ and $H$ has index $m$ in $G$.
\end{proof}

\subsection{Morphisms and subgroups}
\label{sec:Mor}

The first application we extend is the following. 

\begin{proposition}\textup{(See \cite[4.1]{KM02})}\label{propsub} \mbox{}\\
Let $\pi\colon \Gamma\rightarrow \Gamma' $ be a morphism of $X$-graphs, let $v\in V(\Gamma)$ and $v'=\pi(v)$. Suppose that $\Gamma$ and $\Gamma'$ are folded $X$-graphs. Put $K=L(\Gamma,v)$ and $H=L(\Gamma', v')$. Then $K \leq H\leq F(X)$.
\end{proposition} 

We now state the analogue of Proposition \ref{propsub}. 

\begin{proposition}\textup{(Morphisms and Subgroups)}\label{propmorsub2} \mbox{}\\
Let $\pi\colon \Gamma\rightarrow \Gamma' $ be a morphism of $X$-graphs, let $v\in V(\Gamma)$ and $v'=\pi(v)$. Let $G =\langle \,X\, |\, R\, \rangle$ be a  group with $X$ finite and $R$  not necessarily finite. Suppose that $\Gamma$ and $\Gamma'$ are connected finite $X$-regular graphs which fulfill the defining relators $R$. Put $K=L(\Gamma,v)$ and $H=L(\Gamma', v')$. Then $\phi(K) \leq \phi(H)\leq G$.
\end{proposition}

\begin{proof}
By Theorem \ref{thm2}\,(1),  $\phi(K)$ and $ \phi(H)$ are subgroups of $G$. 
By Proposition \ref{propsub}, we have $K \leq H \leq F(X)$. Since $\phi\colon F(X) \rightarrow F(X) / N= G$ is an epimorphism,  $\phi(K) \leq \phi(H) \leq \phi(F(X))=G$ holds. 
\end{proof}

\subsection{Generating systems}
\label{sec:GeneratingSystems}

The next result provides a free basis for the language of an $X$-graph. 
 
Recall that in a connected graph a subgraph is called a \textit{spanning tree} if this subgraph is a tree and contains all vertices of the original graph. If $T$ is a spanning tree, then for any two vertices $u,u'$ of $T$ there is a unique reduced path in $T$ from $u$ to $u'$, which will be denoted $[u,u']_T$.

\begin{proposition}\textup{(Free Basis, see \cite[6.1]{KM02})}\label{properz}\mbox{}\\
Let $\Gamma$ be a folded $X$-graph and let $v$ be a vertex of $\Gamma$. Let $T$ be a spanning tree of $\Gamma$. Let $T^+$ be the set of those edges of $\Gamma$ which lie outside of $T$. For each $e\in T^+$ put $p_e=[v,o(e)]_Te[t(e),v]_T$ (so that $p_e$ is a reduced path from $v$ to $v$ and its label is a freely reduced word in $X \cup X^{-1}$). Also for each $e \in T^+$ put $[e]=\mu(p_e)=\overline{\mu (p_e)}$. Put 
$$ Y_T=\{[e]\mid e\in T^+\}.$$
Then $Y_T$ is a free basis for the subgroup $H=L(\Gamma, v)$ of $F(X)$. 
\end{proposition}

\begin{figure}[h!]
\centering
  \begin{tikzpicture}

\coordinate[label=left:$\Gamma$] () at (-0.5,2);
\coordinate[label=below:$v_1$](1) at (0,0);
\coordinate[label=above:$v_4$](2) at (1.8,1.8); 
\coordinate[label=below:$v_2$](3) at (1.4,0);
\coordinate[label=below:$v_3$](4) at (2.6,0);
\coordinate[label=below:$v_5$](5) at (2.8,1); 
\coordinate[label=above:$v_6$](6) at (0,1.5);
 \filldraw[black](1) circle (0.8pt)
 (2) circle (0.8pt)
 (3) circle (0.8pt)
 (4) circle (0.8pt)
 (5) circle (0.8pt)
 (6) circle (0.8pt)
;
\draw [->,line width=1pt]  (1) to node[right] {$b$} (2);
\draw [->,line width=1pt]  (2) to node[above] {$a$} (6);
\draw [->]  (2) to node[left] {$b$} (4);
\draw [->,line width=1pt]  (5) to node[above] {$a$} (2);
\draw [->,line width=1pt]  (4) to node[above] {$d$} (3);
%\draw [->]  (6) to node[left] {$a$} (5);

\draw [->,line width=1pt]  (1) to[bend left=10] node[above] {$a$} (3);
\draw [->]  (3) to[bend left=20] node[below] {$a$} (1);
\draw [->]  (6) to[bend left=10] node[right] {$c$} (1);
\draw [->]  (6) to[bend right=40] node[left] {$b$} (1);
%\draw [->]  (3) to[bend left=10] node[right] {$b$} (6);
%\draw [->]  (6) to[bend left=10] node[left] {$b$} (3);
\path[->,min distance=1cm] (5) edge[in=30,out=330,right] node {$d$}(5);
\end{tikzpicture}
%\hspace{2mm}
	\caption{A folded $\{a,b,c,d\}$-graph with a spanning tree $T$, marked by thicker arrows.}
	\label{fig5}
\end{figure}
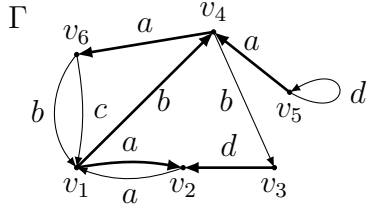

The free basis for the language $L(\Gamma,v_1)$ of the folded $\{a,b,c,d\}$-graph $\Gamma$ shown in Figure \ref{fig5} with $T$ as a spanning tree is the set $$Y_T=\{bab, bac, a^2, b^2da^{-1}, ba^{-1}dab^{-1}\}.$$

Proposition \ref{properz} provides a free basis for the language of a subgroup graph. We use this to get a generating system and even a presentation  for a finite index subgroup $H$ of  $G =\langle \,X\, |\, R\, \rangle$ from its subgroup graph $\Gamma_{X,R}(H)$. Recall that $N:=\left\langle \!\left\langle\, R\, \right\rangle\!\right\rangle_{F(X)}$.

\begin{proposition}\textup{(Generating System and Presentation)\\}
Let $G=\langle \,X\, |\, R\, \rangle$ be a group with $X$ finite and $R$ not necessarily finite. Let $\Gamma$ be a finite connected $X$-regular graph which fulfills the defining relators $R$. Let $H'= L(\Gamma, v)$ and let $S$ be a free basis for $H'$ which we get by Proposition \ref{properz}. Then $\phi(S)$ generates $H=\phi(H')$. Moreover, $\langle \,S\, |\, R'\, \rangle$ is a presentation for the finite index subgroup $H$ of $G$, with $\langle \!\langle\, R' \,\rangle\! \rangle_{F(S)}=N$.
\end{proposition}

 \begin{proof}
The set $S\subset F(X)$ is a free basis for $H'$. Therefore there is an epimorphism $\phi|_{F(S)}\colon F(S) \rightarrow F(S) / N,\ w \mapsto  wN$. Since $N\leq H'=\langle S \rangle$, there exists a subset $R' \subseteq N$ such that $\langle\!\langle\, R' \,\rangle\!\rangle_{F(S)}=N$. 
  Hence $\langle\, S\, |\, R'\, \rangle$ is a presentation for $H$.
 \end{proof}

\begin{example} 
We give examples for the proposition above. We consider the graphs in Figure \ref{dieder}, which are all $\{a, b\}$-regular graphs which fulfill the defining relators of the presentation $\left\langle\,  a, b \, | \, a^3, b^2, (ab)^2 \, \right\rangle$ of the dihedral group $D_3$. 

The set $\{a,b\}$ generates $L(\Gamma, v)=F(a,b)$. Hence $\Gamma$ is the subgroup graph $\Gamma( D_3)$ of $D_3$. 

The language of the graph   $\Gamma'$ is the same for both of its vertices. We have a free basis $\{a, bab^{-1}, b^2\}$ for the language $L(\Gamma',v)$. Since $\phi(bab^{-1})=a^2$ and $\phi(b^2)=1$, the graph $\Gamma'$ is the subgroup graph $\Gamma(\left\langle a \right\rangle)$ of the subgroup $\left\langle a \right\rangle < D_3$. 

For the graph $\Gamma''$ the languages are different for each vertex. 
The language of  $\Gamma''$ with respect to $v_1$ is $F(b, aba^{-2}, a^2ba^{-1}, a^3 )$. Since $\phi(aba^{-2})=\phi(a^2ba^{-1})=b$ and $\phi(a^3)=1$, the subgroup graph of $\left\langle b \right\rangle < D_3$ is $(\Gamma'',v_1)$. The second language for $\Gamma''$ is $L(\Gamma'',v_2)=F(a^3, ab, ba^{-1}, a^2ba^{-2})$. 
We have $ab=\phi(ba^{-1})=\phi(a^2ba^{-2})$. Therefore the graph $(\Gamma'',v_2)$ is the subgroup graph of $\left\langle ab \right\rangle <D_3$.
 The   language
  $L(\Gamma'', v_3)$ has  $\{ a^3, aba^{-1}, ba^{-2}, a^2b \}$ as a free basis  and $\phi$ gives us the subgroup $\left\langle a^2b \right\rangle < D_3$. Hence  $(\Gamma'', v_3)$ is the subgroup graph of $\left\langle a^2b \right\rangle < D_3$.
  
  Since $(\Gamma''', v_i)$ is isomorph to $(\Gamma''',v_j)$, the based graphs provide all the same language. The set $\{a^3, b^2, ab^2a^{-1}, a^2b^2a^{-2}, bab^{-1}a^{-2}, abab^{-1}, a^2bab^{-1}a^{-1}\}$ is a free basis for the language $L(\Gamma''', v_1)$. Since all these words are in the kernel of $\phi$, the graph $\Gamma'''$ is the subgroup graph $\Gamma( \{1_{D_3}\})$ of the trivial group. 
\end{example}

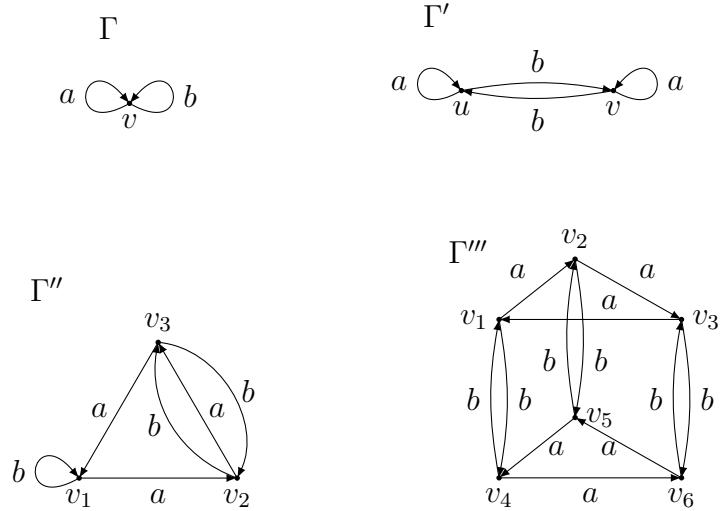
\begin{figure}[h]
\centering
  \begin{tikzpicture}
\coordinate[label=below:$v$](1) at (0,0);
\coordinate[label=left:$\Gamma$] (B) at (0,1);
 \filldraw[black](1) circle (0.8pt)
;
\path[->,min distance=1cm] (1) edge[in=130,out=210,left] node {$a$}(1);
\path[->,min distance=1cm] (1) edge[in=50,out=330,right] node {$b$}(1);

\end{tikzpicture}
\hspace{2cm}
\begin{tikzpicture}
\coordinate[label=below:$u$](1) at (0,0);
\coordinate[label=below:$v$](2) at (2,0); 
\coordinate[label=left:$\Gamma'$] (B) at (0,1);
 \filldraw[black](1) circle (0.8pt)
 (2) circle (0.8pt)
;
%\draw (1) to node[below] {$c$} (2);
%\draw (2) to node[right] {$b$} (3);
%\draw (1) to node[left] {$a$} (3);
\draw [->] (2) to[bend left=10] node[below] {$b$} (1);
\draw [->] (1) to[bend left=10] node[above] {$b$} (2);
\path[->,min distance=1cm] (1) edge[in=130,out=210,left] node {$a$}(1);
\path[->,min distance=1cm] (2) edge[in=50,out=330,right] node {$a$}(2);

\end{tikzpicture}

\vspace{1cm}
\begin{tikzpicture}     

\coordinate[label=below:$v_1$](1) at (210:1.2);
\coordinate[label=below:$v_2$](2) at (330:1.2); 
\coordinate[label=above:$v_3$](3) at (90:1.2);
\coordinate[label=left:$\Gamma''$] (A) at (120:2.2);
 \filldraw[black](1) circle (0.8pt)
 (2) circle (0.8pt)
 (3) circle (0.8pt);
\draw [->] (1) to node[below] {$a$} (2);
%\draw [->] (1) to[bend left=10] node[above] {$b$} (2);
\draw [->] (2) to node[right] {$a$} (3);
\draw [->] (3) to node[left] {$a$} (1);
\draw [->] (2) to[bend left=40] node[left] {$b$} (3);
\draw [->] (3) to[bend left=55] node[right] {$b$} (2);

\path[->,min distance=1cm] (1) edge[in=130,out=210,left] node {$b$}(1);
%\path[->,min distance=1cm] (2) edge[in=50,out=330,right] node {$c$}(2);
%\path[->,min distance=1cm] (3) edge[in=120,out=50,right] node {$b$}(3);
\end{tikzpicture}
\hspace{2cm}
\begin{tikzpicture}

\coordinate[label=left:$\Gamma'''$] () at (0,3);
\coordinate[label=left:$v_1$](1) at (0,2.1);
\coordinate[label=above:$v_2$](2) at (1,2.9); 
\coordinate[label=right:$v_3$](3) at (2.4,2.1);
\coordinate[label=below:$v_4$](4) at (0,0);
\coordinate[label=right:$v_5$](5) at (1,0.8); 
\coordinate[label=below:$v_6$](6) at (2.4,0);
 \filldraw[black](1) circle (0.8pt)
 (2) circle (0.8pt)
 (3) circle (0.8pt)
 (4) circle (0.8pt)
 (5) circle (0.8pt)
 (6) circle (0.8pt)
;
\draw [->]  (1) to node[above left] {$a$} (2);
\draw [->]  (2) to node[above right] {$a$} (3);
\draw [->]  (3) to node[above right] {$a$} (1);
\draw [->]  (4) to node[below] {$a$} (6);
\draw [->]  (5) to node[right] {$a$} (4);
\draw [->]  (6) to node[left] {$a$} (5);

\draw [->]  (1) to[bend left=10] node[right] {$b$} (4);
\draw [->]  (4) to[bend left=10] node[left] {$b$} (1);
\draw [->]  (2) to[bend left=10] node[below right] {$b$} (5);
\draw [->]  (5) to[bend left=10] node[below left] {$b$} (2);
\draw [->]  (3) to[bend left=10] node[right] {$b$} (6);
\draw [->]  (6) to[bend left=10] node[left] {$b$} (3);
\end{tikzpicture}
%\hspace{2mm}
	\caption{All $\{a, b\}$-regular graphs which fulfill the defining relators of the presentation $\left\langle\,  a, b \, | \, a^3, b^2, (ab)^2\,   \right\rangle$ of the dihedral group $D_3$.}
	\label{dieder}
\end{figure}

\begin{remark}
The example above shows that for isomorphic groups with different presentations we can have different subgroup graphs of isomorphic subgroups.
Take $S_3= \left\langle\,  s_1, s_2 \, | \, s_1^2, s_2^2, (s_1s_2)^3  \, \right\rangle \cong \left\langle \,  a, b \, | \, a^3, b^2, (ab)^2 \, \right\rangle=D_3$.
Comparing Figure \ref{S3} and \ref{dieder}, we observe that the subgroup graphs of the proper subgroups are different. Since the number of vertices of a subgroup graph is the index of the associated subgroup, we have nevertheless $|V(\Gamma(H))|=|V(\Gamma(H'))|$ for $S_3 \geq H\cong H'\leq D_3$.
 \end{remark}

\subsection{Conjugate subgroups}
\label{sec:ConjugacyClass}

With the next application we can detect the conjugacy class of a finite index subgroup of a finitely generated group. 

For finite index subgroups of the free group $F(X)$, with $X$ finite, Kapovich and Myasnikov  proved the following. 

\begin{lemma}\textup{(See \cite[7.5]{KM02})\\} \label{7.5}
Let $\Gamma$ be a folded core graph (with respect to one of its vertices). Let $v$ and $u$ be two vertices of $\Gamma$ and let $q$ be a reduced path in $\Gamma$ from $v$ to $u$ with label $g\in F(X)$. Let $H=L(\Gamma, v)$ and $K=L(\Gamma, u)$. Then $H=gKg^{-1}$. 
\end{lemma}

\begin{proposition}\textup{(See  \cite[7.7]{KM02})\\} \label{7.7}
Let $H$ and $K$ be finite index subgroups of $F(X)$. Then $H$ is conjugate to $K$ in $F(X)$ if and only if the graphs $\Gamma_X(H)$ and $\Gamma_X(K)$ are isomorphic as $X$-graphs. 
\end{proposition}

\begin{lemma}\textup{(See \cite[7.12]{KM02})\\} \label{7.12}
Let $H$ and $ K$ be finite index subgroups of $F(X)$. Then there exists an element $g\in F(X)$ with $gKg^{-1}\leq H$ if and only if there exists a morphism of (non-based) $X$-graphs $\pi\colon \Gamma_X(K) \rightarrow \Gamma_X(H)$.
\end{lemma}

We extend these results to finite index subgroups of finitely generated groups. 

\begin{lemma}\label{lcon2}\mbox{}\\
Let $G=\langle \,X\, |\, R\, \rangle$ be a group with $X$  finite and $R$  not necessarily finite.
Let $\Gamma$ be a finite $X$-regular graph which fulfills the relators $R$. Let $v$ and $u$ be two vertices of $\Gamma$ and let $p$ be a reduced path from $v$ to $u$ with label $g'\in F(X)$. Let $H=\phi(L(\Gamma,v))$ and $K=\phi(L(\Gamma,u))$. Then $H=gKg^{-1}$ for $g=g'N \in G$.  
\end{lemma}

\begin{proof}
Let $H'=L(\Gamma, v)$ and $K'=L(\Gamma, u)$. By Lemma \ref{7.5}, $H'=g'K'g'^{-1}$ for some $g'\in F(X)$. Since $\phi$ is a homomorphism,  $H=\phi(H')=\phi(g'K'g'^{-1})=gKg^{-1}$ for $g=g'N\in G$.
\end{proof}

Therefore the subgroup $H$ is conjugate to the subgroup $\phi(L(\Gamma_{X,R}(H), v))$ in $G$  for all $v \in V(\Gamma_{X,R} (H))$. 

\begin{proposition}\textup{(Conjugate Subgroups)}\label{con2}\mbox{}\\
Let $H$ and $K$ be subgroups of finite index in the group $G=\langle \,X\, |\, R\, \rangle$, where $X$ is finite and $R$ is not necessarily finite. Then $H$ is conjugate to $K$ in $G$ if and only if the subgroup graphs $\Gamma_{X,R}(H)$ and $\Gamma_{X,R}(K)$ are isomorphic as $X$-graphs. 
\end{proposition}

\begin{proof}
Let $H'=L(\Gamma_{X,R}(H), 1_H)$ and $K'=L(\Gamma_{X,R}(K), 1_K)$. Then $H'$ and $K'$ are finite index subgroups of $F(X)$ and $\Gamma_{X,R}(H)=\Gamma_X(H')$ and $\Gamma_{X,R}(K)=\Gamma_X(K')$.   Proposition \ref{7.7} completes the proof.
\end{proof}

\begin{lemma}\mbox{}\\
Let $H$ and $ K$ be finite index subgroups of the group   $G=\langle \,X\, |\, R\, \rangle$, where $X$ is finite and $R$ is not necessarily finite. Then there is $g\in G$ with $gKg^{-1}\leq H$ if and only if there exists a morphism of (non-based) $X$-graphs $\pi\colon \Gamma_{X,R}(K)\rightarrow \Gamma_{X,R}(H)$.
\end{lemma}

\begin{proof}
This follows from Lemma \ref{7.12}.
\end{proof}

\subsection{Normal subgroups and normalizer}
\label{sec:NormalSubgroup}
Proposition \ref{con2} states that for a finite index subgroup $H$ of a group $G$ the set $\{\phi(L(\Gamma(H),v)) \mid v\in V(\Gamma(H)) \}$ is the conjugacy class of $H$. Thus $H$  has at most $|V(\Gamma(H))|$ conjugate subgroups. This leads to the next result.

 \begin{theorem}\textup{(Normal Subgroups)}\label{norm2}\mbox{}\\
Let $H$ be a finite index subgroup of the  group $G= \langle \,X\, |\, R\, \rangle$, where $X$ is finite and $R$ is not necessarily finite. Then $H$ is normal in $G$ if and only if the based $X$-graphs $(\Gamma_{X,R}(H),1_H)$ and $(\Gamma_{X,R}(H),v)$ are isomorphic for all $v\in V( \Gamma_{X,R}(H))$. 
\end{theorem}

\begin{proof}
Assume that $H=\phi(L(\Gamma_{X,R}(H),1_H))$ is conjugate to $\phi(L(\Gamma_{X,R}(H),v))$. The subgroup $H$ is normal if and only if $H$ is conjugate only to itself. This is equivalent to $\phi(L(\Gamma_{X,R}(H),v))=H$ for all $v\in V(\Gamma_{X,R}(H))$. Hence $(\Gamma_{X,R}(H),1_H)$ and $(\Gamma_{X,R}(H),v)$ are isomorphic. 
\end{proof}

With the subgroup graph we can detect the normalizer of a subgroup. The normalizer of a subgroup $H$ in a group $G$ is the subgroup $$N_G(H)=\{g\in G \mid gHg^{-1}=H\}.$$

\begin{theorem}\textup{(Normalizer)}
\label{thmnormalizer}\mbox{}\\ 
Let $G =\langle \,X\, |\, R\, \rangle$ be a group with $X$ finite and $R$  not necessarily finite. Let $H$ be a finite index subgroup of $G$. Let $p_v$ be the reduced path in $\Gamma_{X,R}(H)$ from $1_H$ to $v$ with label $\mu(p_v)=g_v$. Then $g_v\in N_G(H)$ if and only if $(\Gamma_{X,R}(H),1_H)$ and $(\Gamma_{X,R}(H),v)$ are isomorphic as based $X$-graphs. Furthermore, let $V$ be the set of vertices of $\Gamma_{X,R}(H)$ with $(\Gamma_{X,R}(H),1_H)$ isomorph to $(\Gamma_{X,R}(H),v)$ as based $X$-graphs. Then
$$N_G(H)=\bigcup\limits_{v\in V} Hg_v. $$
\end{theorem}

\begin{proof} 
Let $g_v\in G$ and let $p_v$ be the reduced path with label $\mu(p_v)=g_v$, origin $1_H$ and terminus $v$ in $\Gamma(H):=\Gamma_{X,R}(H)$.  By Lemma \ref{lcon2}, we have $H=g_v K g_v^{-1}$ for $K=\phi(L(\Gamma(H), v))\leq G$. If $g_v\in N_G(H)$, then $K=g_v^{-1}Hg_v=H$. Therefore $(\Gamma(H), 1_H)$ and $(\Gamma(H), v)$ are isomorphic as based $X$-graphs. 

Let $(\Gamma(H), 1_H)$ and $(\Gamma(H), v)$ be isomorphic as based $X$-graphs. Then the subgroups $H=\phi(L(\Gamma(H), 1_H))$ and $K=\phi(L(\Gamma(H), v))$ are equal. Let $p_v$ be the reduced path from $1_H$ to $v$ with label $g_v$. Then $H=g_v K g_v^{-1}=g_v H g_v^{-1}$. Hence $g_v\in N_G(H)$. 

Let $V:=\{ v\in V(\Gamma(H)) \mid (\Gamma(H), 1_H)\cong (\Gamma(H), v)\}$. If $g_v\in N_G(H)$, then $Hg_v\subseteq N_G  (H)$.
Let $g\in N_G(H)$ and let $p$ be the reduced path with label $g$ origin $1_H$ and terminus $v'$. Since $g\in N_G(H)$,  the based graphs $(\Gamma(H),1_H)$ and $(\Gamma(H),v')$ are isomorphic. Therefore $v'\in V$.
\end{proof}

Theorem \ref{thmnormalizer} shows that if there is no symmetry in the subgroup graph $\Gamma(H)$ of a subgroup $H< G$ (that is  $(\Gamma(H),1_H)\ncong (\Gamma(H),v)$ for all $v\neq 1_H$), then $N_G(H)=H$.

\subsection{Intersection of subgroups}
\label{sec:IntersectionOfSubgroups}
We provide the subgroup graph of the intersection of two finite index subgroups.
In this subsection we use $\Gamma(H)$ for the subgroup graph of the subgroup $H$ for both $H \leq F(X)$ and $H\leq G$. This is less precise but clearer to read.  

\begin{definition}{(Product Graph, see \cite[9.1]{KM02})}\label{defproductgraph}\mbox{}\\
Let $\Gamma$ and $\Gamma'$ be $X$-graphs. We define the \textit{product graph} $\Gamma \times \Gamma'$ as follows. The vertex set of $\Gamma \times \Gamma'$ is the set $V(\Gamma) \times V(\Gamma')$. For a pair of vertices $(u,v),\ (u', v')$ in  $V(\Gamma \times \Gamma')$ (such that $u,u' \in V(\Gamma)$ and $v,v' \in V(\Gamma')$) and a letter $x \in X$ we introduce an edge, labeled $x$, with origin $(u,v)$ and terminus $(u',v')$, provided that there is an edge, labeled $x$, from $u$ to $u'$ in $\Gamma$ and there is an edge, labeled $x$, from $v$ to $v'$ in $\Gamma'$.
\end{definition}

Thus $\Gamma \times \Gamma'$ is an $X$-graph. We denote a vertex $(u,v)$ of the product graph $\Gamma \times \Gamma'$ by $u \times v$.  

For an example of a product graph see Figure \ref{schnitt}. 
 The graph in the second row is a product graph of the two graphs in the first row.

\begin{lemma}\textup{(See \cite[9.2]{KM02})}\label{lem9.2}\mbox{}\\
Suppose $\Gamma$ and $\Gamma'$ are folded $X$-graphs. Then $\Gamma \times \Gamma'$ is also a folded $X$-graph. 
\end{lemma}

\begin{proposition}\textup{(See \cite[9.4]{KM02})}\label{propschn} \mbox{}\\
Let $H$ and $K$ be two subgroups of $F(X)$. 
Let $\Gamma(H)\times_1 \Gamma(K)$ be the connected component of the product graph $\Gamma(H) \times \Gamma(K)$ containing $1_H\times 1_K$ and $\Delta$ be the core of $\Gamma(H) \times_1 \Gamma(K)$ with respect to $ 1_H \times 1_K$. 
 Then $(\Gamma(H\cap K), 1_{H\cap K})=(\Delta, 1_H \times 1_K)$.
\end{proposition}

 To generalize Proposition \ref{propschn} we need the following two lemmas.

\begin{lemma}\label{lem2}\mbox{}\\
Let $\Gamma$ and $\Gamma'$ be two finite connected $X$-regular graphs. Then the product graph $\Gamma \times \Gamma'$ is a finite $X$-regular graph.  
\end{lemma} 
 
 \begin{proof}
 Let $v \times v'$ be a vertex in $\Gamma \times \Gamma'$ and let $x\in X \cup X^{-1}$. Since $\Gamma$ and $\Gamma'$ are $X$-regular, there exists exactly one edge $e$ with label $x$ and origin $v$ in $\Gamma$ and exactly one edge $e'$ with label $x$ and origin $v'$ in $\Gamma'$.  By Definition \ref{defregular}, there exists exactly one edge in $\Gamma \times \Gamma'$ with label $x$ and origin $v \times v'$.  
  Hence $\Gamma \times \Gamma'$ is $X$-regular. Since $\Gamma$ and $\Gamma'$ are finite, the product graph is finite. 
 \end{proof}

\begin{lemma}\label{lem3}\mbox{}\\
Let $G=\langle \,X\, |\, R\, \rangle$ be a group with $X$ finite and $R$  not necessarily finite. Let $\Gamma$ and $\Gamma'$ be two finite $X$-regular graphs which fulfill the defining relators $R$. Then $\Gamma \times \Gamma'$ fulfills the defining relators $R$.
\end{lemma} 

\begin{proof}
By Lemma \ref{lem2}, $\Gamma \times \Gamma'$ is a finite $X$-regular graph. Let $v\times v'$ be a vertex in $\Gamma \times \Gamma'$ and $r\in R$. Then there exists exactly one reduced path $p_r$ with $\mu(p_r)=r$ and $o(p_r)=v\times v'$. Since $\Gamma$ and $\Gamma'$ fulfill the defining relators $R$, the reduced path $p$ in $\Gamma$ with $\mu(p)=r$ and $o(p)=v$  has terminus $v$ and the reduced path $p'$ in $\Gamma'$ with $\mu(p')=r$ and $o(p_r)=v'$ has terminus $v'$. Therefore $t(p_r)=v \times v'$ in $\Gamma \times \Gamma'$.
\end{proof}

\begin{proposition}\textup{(Intersection)}\label{propint}\mbox{}\\ 
Let $H$ and $K$ be finite index subgroups of the group $G=\langle \,X\, |\, R\, \rangle$, where $X$ is finite and $R$ is not necessarily finite. 
Let $\Gamma( H) \times_1 \Gamma( K)$ be the connected component of the product graph $\Gamma(H) \times \Gamma(K)$ containing $1_H \times 1_K$. Then $(\Gamma( H) \times_1 \Gamma( K), 1_{ H} \times 1_{ K})$ is the subgroup graph of  
$H \cap K< G$.
\end{proposition}

\begin{proof}
By Theorem \ref{thm2}, the graphs $\Gamma(H)$ and $\Gamma( K)$ are finite, connected, $X$-regular and fulfill the defining relators $R$. By Lemmas \ref{lem2} and \ref{lem3}, the product graph $\Gamma( H) \times \Gamma ( K)$ is finite, $X$-regular and fulfills the defining relators $R$. Therefore  $\Delta=Core(\Gamma(  H) \times_1 \Gamma(K), 1_{ H} \times 1_{ K})=\Gamma( H) \times_1 \Gamma( K)$ is $X$-regular and fulfills the defining relators $R$. 
Recall that $\Gamma(H)=\Gamma(H')$ and $\Gamma(K)= \Gamma(K')$ for $H'=L(\Gamma( H), 1_{ H})$ and $ K'=L(\Gamma(K), 1_{K})$. 
By Proposition \ref{propschn}, we know that $(\Delta, 1_{H} \times 1_{ K})$ is the subgroup graph of $ H' \cap K'$.
Thus $L(\Delta, 1_H\times 1_K)= H' \cap  K'$. Hence  $\phi(L(\Delta, 1_H \times 1_K))=\phi( H' \cap K')$. 
Since $\phi$ is a homomorphism and $\ker \phi \leq  H' \cap K'$, we have $\phi( H' \cap K')= \phi( H') \cap \phi( K')= H \cap K$. 
\end{proof}

\begin{figure}[h!]
	\centering
  \begin{tikzpicture}     

\coordinate[label=below:$1_{ H}$](1) at (210:1);
\coordinate[label=below:$v$](2) at (330:1); 
\coordinate[label=left:$u$](3) at (90:1);
\coordinate[label=left:$\Gamma( H)$] (A) at (120:1.8);
 \filldraw[black](1) circle (0.8pt)
 (2) circle (0.8pt)
 (3) circle (0.8pt);
%\draw (1) to node[below] {$b$} (2);
\draw [->] (2) to[bend left=10] node[below] {$b$} (1);
\draw [->] (1) to[bend left=10] node[above] {$b$} (2);
%\draw (2) to node[right] {$a$} (3);
%\draw (1) to node[left] {$c$} (3);
\draw [->] (2) to[bend left=10] node[left] {$a$} (3);
\draw [->] (3) to[bend left=10] node[right] {$a$} (2);
\draw [->] (3) to[bend left=10] node[below] {$\ c$} (1);
\draw [->] (1) to[bend left=10] node[left] {$c$} (3);
\path[->,min distance=1cm] (1) edge[in=130,out=210,left] node {$a$}(1);
\path[->,min distance=1cm] (2) edge[in=50,out=330,right] node {$c$}(2);
\path[->,min distance=1cm] (3) edge[in=120,out=50,right] node {$\ b$}(3);
\end{tikzpicture}
\hspace{1.5cm}
\begin{tikzpicture}
\coordinate[label=below:$1_{ K}$](1) at (210:1);
\coordinate[label=below:$v'$](2) at (330:1); 
\coordinate[label=left:$u'$](3) at (90:1);
\coordinate[label=left:$\Gamma( K)$] (B) at (120:1.8);
 \filldraw[black](1) circle (0.8pt)
 (2) circle (0.8pt)
 (3) circle (0.8pt)
;
%\draw (1) to node[below] {$c$} (2);
%\draw (2) to node[right] {$b$} (3);
%\draw (1) to node[left] {$a$} (3);
\draw [->] (2) to[bend left=10] node[below] {$c$} (1);
\draw [->] (1) to[bend left=10] node[above] {$c$} (2);
\draw [->] (2) to[bend left=10] node[left] {$b$} (3);
\draw [->] (3) to[bend left=10] node[right] {$b$} (2);
\draw [->] (3) to[bend left=10] node[below] {$\ a$} (1);
\draw [->] (1) to[bend left=10] node[left] {$a$} (3);
\path[->,min distance=1cm] (1) edge[in=130,out=210,left] node {$b$}(1);
\path[->,min distance=1cm] (2) edge[in=50,out=330,right] node {$a$}(2);
\path[->,min distance=1cm] (3) edge[in=120,out=50,right] node {$\ c$}(3);
\end{tikzpicture}
\vspace{8mm}

\begin{tikzpicture}
\coordinate[label=left:$\Gamma( H) \times \Gamma( K)$] () at (150:2.8);
\coordinate[label=left:$1_{H} \times 1_{ K}$](4) at (180:1.5);
\coordinate[label=left:$1_{ H}\times u'$](5) at (120:1.5); 
\coordinate[label=right:$v \times v'$](6) at (60:1.5);
\coordinate[label=right:$u \times v'$](7) at (0:1.5);
\coordinate[label=right:$u \times u'$](8) at (300:1.5); 
\coordinate[label=left:$v \times 1_{ K}$](9) at (240:1.5);
\coordinate[label=left:$c$](10) at (150:0.5);
\coordinate[label=above:$c$](11) at (105:0.46); 
\coordinate[label=right:$c$](12) at (56:0.52);
\coordinate[label=right:$c$](13) at (330:0.5);
\coordinate[label=below:$c$](18) at (280:0.46); 
\coordinate[label=left:$c$](19) at (230:0.52);
 \filldraw[black](4) circle (0.8pt)
 (5) circle (0.8pt)
 (6) circle (0.8pt)
 (7) circle (0.8pt)
 (8) circle (0.8pt)
 (9) circle (0.8pt)
;
\draw [->]  (4) to[bend left=10] node[left] {$a$} (5);
\draw [->]  (4) to[bend left=10] (7);
\draw [->]  (4) to[bend left=10] node[right] {$b$} (9);
\draw [->]  (6) to[bend left=10] node[below] {$\ b$} (5);
\draw [->]  (6) to[bend left=10] (9);
\draw [->]  (6) to[bend left=10] node[right] {$a$} (7);
\draw [->]  (8) to[bend left=10] node[left] {$b$} (7);
\draw [->]  (8) to[bend left=10] node[below] {$a$} (9);
\draw [->]  (8) to[bend left=10]  (5);

\draw [->]  (5) to[bend left=10] node[right] {$a$} (4);
\draw [->]  (7) to[bend left=10] (4);
\draw [->]  (9) to[bend left=10] node[left] {$b$} (4);
\draw [->]  (5) to[bend left=10] node[above] {$b$} (6);
\draw [->]  (9) to[bend left=10] (6);
\draw [->]  (7) to[bend left=10] node[left] {$a$} (6);
\draw [->]  (7) to[bend left=10] node[right] {$b$} (8);
\draw [->]  (9) to[bend left=10] node[above] {$a$} (8);
\draw [->]  (5) to[bend left=10] (8);
\end{tikzpicture}
\hspace{-4mm}
\begin{tikzpicture}     

\coordinate[label=below:$u \times 1_{ K}$](1) at (210:1);
\coordinate[label=below:$\ 1_{ H} \times v'$](2) at (330:1); 
\coordinate[label=left:$v \times u'$](3) at (90:1);

 \filldraw[black](1) circle (0.8pt)
 (2) circle (0.8pt)
 (3) circle (0.8pt)
;
%\draw (1) to node[below] {$c$} (2);
%\draw (2) to node[right] {$b$} (3);
%\draw (1) to node[left] {$a$} (3);
\draw [->] (2) to[bend left=10] node[below] {$c$} (1);
\draw [->] (1) to[bend left=10] node[above] {$c$} (2);
\draw [->] (2) to[bend left=10] node[left] {$b$} (3);
\draw [->] (3) to[bend left=10] node[right] {$b$} (2);
\draw [->] (3) to[bend left=10] node[below] {$\ a$} (1);
\draw [->] (1) to[bend left=10] node[left] {$a$} (3);
\path[->,min distance=1cm] (1) edge[in=130,out=210,left] node {$b$}(1);
\path[->, min distance=1cm] (2) edge[in=50,out=330,right] node {$a$}(2);
\path[->,min distance=1cm] (3) edge[in=120,out=50,above] node {$c$}(3);
\end{tikzpicture}
	\caption{Two subgroup graphs $\Gamma(H)$, $\Gamma(K)$ in the first row and the disconnected product graph $\Gamma(H) \times \Gamma(K)$ in the second row.}
	\label{schnitt}
\end{figure}
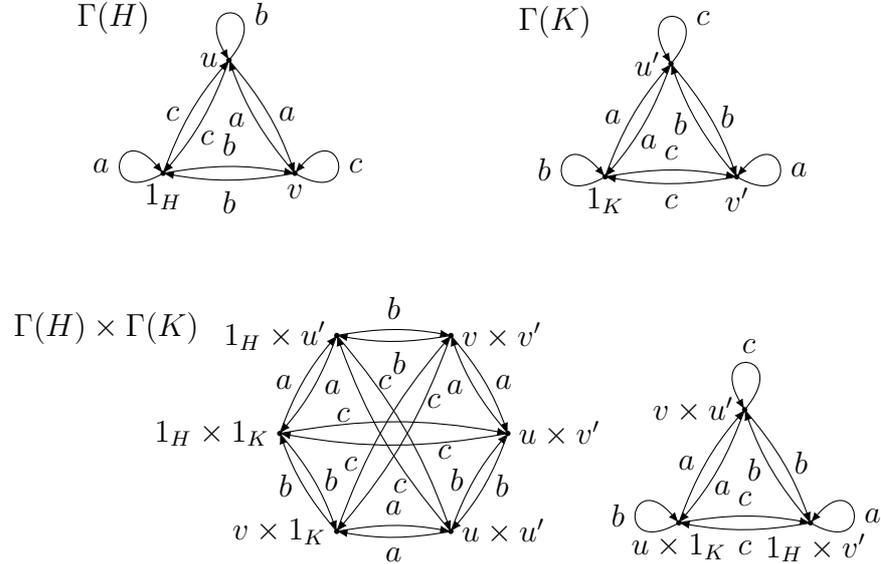

\begin{proposition}\textup{(Intersection of Cosets)\\}
Let $H$ and $K$ be finite index subgroups of the  group $G=\langle \,X\, |\, R\, \rangle$, where $X$ is finite and $R$ is not necessarily finite.  
If $Hw$ is the vertex $v \in V(\Gamma(H))$ and $Kw'$ the vertex $v'\in V(\Gamma(K))$, then $v \times v'$ is a vertex in $\Gamma(H)\times_1 \Gamma(K)$ if and only if the intersection of the cosets $Hw$ and $Kw'$ is not empty. 
\end{proposition}

\begin{proof}
Let $v\times v'$ be a vertex in $\Gamma(H) \times_1 \Gamma(K)$. Let $p$ be the reduced path from $1_H\times 1_{K}$ to $v \times v'$ and $\mu(p)=g$. By definition, there exists a reduced path $p'$ from $1_H$ to $v$ in $\Gamma(H)$ and a reduced path $p''$ from $1_{K}$ to $v'$ in $\Gamma(K)$. Therefore $\phi(g)$ is in $Hw$ and $Kw'$. 

Let  $g \in Hw\cap Kw'$. Hence there is a reduced path $p$ with $\mu(p)=g$ from $1_H$ to $v$ in $\Gamma(H)$  and a reduced path $p'$ with label $g$ from $1_K$ to $v'$ in $\Gamma(K)$. Thus the vertex $v\times v'$ is in $V(\Gamma(H) \times_1 \Gamma(K))$.
\end{proof}

\begin{example} 
The graph $(\Gamma(H), 1_H)$ of Figure \ref{schnitt} is the subgroup graph of the subgroup  $H=\left\langle a, cbc, cab \right\rangle$ and $(\Gamma(K), 1_K)$ is the subgroup graph of the subgroup $K= \left\langle b, aca, abc \right\rangle$  of the group  $\Delta (3,3,3)=\langle \, a, b, c\, |\, a^2, b^2, c^2, (ab)^3, (bc)^3, (ac)^3\, \rangle$ (Coxeter group of type $\widetilde A_2$). The connected component of $1_{ H} \times 1_{ K}$ is the subgroup graph of the intersection $H \cap K$. Hence $H \cap K$ is a normal subgroup of index $6$ in $\Delta (3,3,3)$. 
From the subgroup graphs $\Gamma(H)$ and $\Gamma(K)$ we get the cosets $H=1_H$, $Hb=v$, $Hc=u$ of the subgroup $H$ and the cosets $K=1_K$, $Kc=v'$, $Ka=u'$ of the subgroup $K$. Using the product graph $\Gamma(H) \times \Gamma(K)$, we get the following intersection of the cosets: 
 $H\cap K$;
$H \cap Ka = (H\cap K)a$;  $Hb\cap K= (H \cap K)b$  and $Hc \cap Kc=(H \cap K)c$. 
Furthermore, there exist paths  $1_H \times 1_K\xrightarrow{ab} v\times v'$  and $1_H\times 1_K\xrightarrow{ba} u \times u' $, therefore  $Hb \cap Kc= (H\cap K)ab$ and $Hc \cap Ka =(H \cap K)ba$. 
Since $1_H \times v'$, $u\times 1_K$ and $v\times u'$ are not in $\Gamma(H) \times_1\Gamma(K)$, the intersections $H \cap Kc$, $Hc \cap K$ and $Hb \cap Ka$ are empty.
 \end{example}

\subsection{Malnormal subgroups}
\label{sec:Malnormal}

A subgroup $H$ of a group $G$ is called \textit{malnormal} if for all $g\in G\setminus H$ 
$$ gHg^{-1} \cap H=\{1_G\}.$$
The groups $G$ and $\{1_G\}$ are always malnormal. 

The next two propositions are needed for the theorem for malnormal subgroups.  

\begin{proposition}\textup{(See \cite[9.7]{KM02})\\} \label{9.7}
Let $H$ and $K$ be subgroups of $F(X)$. Let $g\in F(X)$ be such that the double cosets $KgH$ and $KH$ are distinct. Suppose that $gHg^{-1}\cap K \neq \{1\}$. Then there is a vertex $v \times u$ in $\Gamma(H) \times \Gamma(K)$ which does not belong to the connected component of $1_H\times 1_K$ such that the subgroup $L(\Gamma(H) \times \Gamma(K), v\times u)$ is conjugate to $gHg^{-1}\cap K$ in $F(X)$. 
\end{proposition}

\begin{proposition}\textup{(See \cite[9.8]{KM02})\\} \label{9.8}
Let $H$ and $K$ be subgroups of $F(X)$. Then for any vertex $v\times u$ of $\Gamma(H)\times \Gamma(K)$ the subgroup $L(\Gamma(H) \times \Gamma(K), v\times u)$ is conjugate to a subgroup of the form $gHg^{-1}\cap K$ for some $g\in F(X)$. Moreover, if $v\times u$ does not belong to the connected component of $1_H \times 1_K$, then the element $g$ can be chosen such that $KgH\neq KH$.
\end{proposition}

\begin{theorem}\textup{(See  \cite[9.10]{KM02})\\} \label{malnorm1}
Let $H\leq F(X)$ be a subgroup of $F(X)$. Then $H$ is malnormal in $F(X)$ if and only if every component of $\Gamma(H)\times \Gamma(H)$, which does not contain $1_H\times 1_H$, is a tree. 
\end{theorem}

The intersection of two finite index subgroups has finite index. Hence no proper finite index subgroup of an infinite  finitely generated group is malnormal. Therefore the next application restricts to malnormal subgroups of finite groups.

\begin{theorem}\textup{(Malnormal Subgroups)}\label{malnorm2}\mbox{}\\
Let $H$ be a subgroup of the finite group  $G= \langle \,X\, |\, R\, \rangle$.  The subgroup $H$ is malnormal in $G$ if and only if $L(\Gamma(H) \times \Gamma(H), u \times v)=N$ for all vertices $u \times v$ not in the connected component of $1_H \times 1_H$.  
\end{theorem}

\begin{proof}
Let $H':=\phi^{-1}(H)$. Thus $\Gamma(H)=\Gamma(H')$.
Let $\Gamma \neq \Gamma(H) \times_1 \Gamma(H)$ be a connected component of $\Gamma (H) \times \Gamma(H)$
and let $u\times v$ be a vertex of $\Gamma$. 
By Proposition \ref{9.8}, $L(\Gamma, u\times v)$ is conjugate to  $g'H'g'^{-1}\cap H'$ for some $g'\in F(X)\setminus H'$.
Since $\ker \phi\leq H'$, we have 
$\phi(g'H'g'^{-1}\cap H')=gHg^{-1} \cap H$ for $g=g'N$. Thus $\phi(L(\Gamma, u\times v))$ is conjugate to $gHg^{-1} \cap H$ for some $g\in G\setminus H$. By Proposition \ref{lem3} and \ref{prop21},  $N \leq L(C, u \times v)$.  
Assume $L(\Gamma, u \times v)\neq N$. Then $g'H'g'^{-1} \cap H' \neq N$. Hence  $gHg^{-1}\cap H\neq \{1_G\}$ for some $g\in G \setminus H$. 

Let $g\in G\setminus H$, then $g\in  F(X)\setminus H'$ and we have $H'gH'\neq H'H'$.
Assume $gH'g^{-1}\cap H' \neq \{1_{F(X)}\}$. By Proposition \ref{9.7}, there exists a vertex $u\times v$ in $\Gamma(H)\times \Gamma (H)\setminus \Gamma(H)\times_1\Gamma(H)$ 
such that $L(\Gamma(H)\times \Gamma(H), u \times v)$ is conjugate to $gH'g^{-1} \cap H'$. Suppose that $L(\Gamma(H)\times \Gamma(H), u \times v)=N$. 
Then $gH'g^{-1}\cap H'=N$ for some $w\in F(X)$. Hence for all $g\in G\setminus H$ with $gH'g^{-1}\cap H'\neq \{1_{F(X)}\}$ we have $gHg^{-1}\cap H=\phi(N)=\{1_G\}$. 
For all $g\in G\setminus H$ with $gH'g^{-1}\cap H'= \{1_{F(X)}\}$ we have $gHg^{-1}\cap H=\{1_G\}$. 
\end{proof}

Figure \ref{malnormal} shows an example of a malnormal subgroup. 
\begin{figure}[h!]
	\centering
  \begin{tikzpicture}     

\coordinate[label=left:$\Gamma(H)$](0) at (0,1);
\coordinate[label=below:$1_H$] (1) at (0,0);
\coordinate[label=below:$v$] (2) at (1.5,0);
\coordinate[label=below:$u$] (3) at (3,0);

 \filldraw[black](1) circle (0.8pt)
 (2) circle (0.8pt)
 (3) circle (0.8pt)
  ;
\path[->,min distance=1cm] (1) edge[in=120,out=200,left] node {$s_1$}(1);
\path[->,min distance=1cm] (3) edge[in=340,out=60,right] node {$s_2$}(3);
\draw [->] (1) to[bend left=25] node[above] {$s_2$} (2);
\draw [->] (2) to[bend left=20] node[below] {$s_2$} (1);

\draw [->] (2) to[bend left=25] node[above] {$s_1$} (3);
\draw [->] (3) to[bend left=20] node[below] {$s_1$} (2);
\coordinate[label=above:$\bigcap$](0) at (5,0);
\coordinate[label=left:$\Gamma(H)$](0) at (7,1);
\coordinate[label=below:$1_H$] (4) at (7,0);
\coordinate[label=below:$v$] (5) at (8.5,0);
\coordinate[label=below:$u$] (6) at (10,0);

 \filldraw[black](4) circle (0.8pt)
 (5) circle (0.8pt)
 (6) circle (0.8pt)
  ;
\path[->,min distance=1cm] (4) edge[in=120,out=200,left] node {$s_1$}(4);
\path[->,min distance=1cm] (6) edge[in=340,out=60,right] node {$s_2$}(6);
\draw [->] (4) to[bend left=25] node[above] {$s_2$} (5);
\draw [->] (5) to[bend left=20] node[below] {$s_2$} (4);

\draw [->] (5) to[bend left=25] node[above] {$s_1$} (6);
\draw [->] (6) to[bend left=20] node[below] {$s_1$} (5);
\end{tikzpicture}
\vspace{10mm}

\begin{tikzpicture}     

\coordinate[label=above:$\Gamma(H) \times \Gamma(H)$](0) at (2,2);
\coordinate[label=below:$1_H\times 1_H\quad $] (1) at (0,0);
\coordinate[label=below:$v\times v$] (2) at (1.5,0);
\coordinate[label=below:$u\times u$] (3) at (3,0);

 \filldraw[black](1) circle (0.8pt)
 (2) circle (0.8pt)
 (3) circle (0.8pt)
 (0,-1.2) circle (0pt)
  ;
\path[->,min distance=1cm] (1) edge[in=120,out=200,left] node {$s_1$}(1);
\path[->,min distance=1cm] (3) edge[in=340,out=60,right] node {$s_2$}(3);
\draw [->] (1) to[bend left=25] node[above] {$s_2$} (2);
\draw [->] (2) to[bend left=15] node[below] {$s_2$} (1);

\draw [->] (2) to[bend left=25] node[above] {$s_1$} (3);
\draw [->] (3) to[bend left=15] node[below] {$s_1$} (2);
\end{tikzpicture}
\begin{tikzpicture}
\coordinate[label=left:$1_H \times v$](4) at (180:1.5);
\coordinate[label=above left:$1_H\times u$](5) at (120:1.5); 
\coordinate[label=above right:$v \times u$](6) at (60:1.5);
\coordinate[label=right:$u \times v$](7) at (0:1.5);
\coordinate[label=below right:$u \times 1_H$](8) at (300:1.5); 
\coordinate[label=below left:$v \times 1_H$](9) at (240:1.5);
 \filldraw[black](4) circle (0.8pt)
 (5) circle (0.8pt)
 (6) circle (0.8pt)
 (7) circle (0.8pt)
 (8) circle (0.8pt)
 (9) circle (0.8pt)
;
\draw [->]  (4) to[bend left=25] node[left] {$s_1$} (5);

\draw [->]  (4) to[bend left=25] node[right] {$s_2$} (9);
\draw [->]  (6) to[bend left=25] node[below] {$s_2$} (5);

\draw [->]  (6) to[bend left=25] node[right] {$s_1$} (7);
\draw [->]  (8) to[bend left=25] node[left] {$s_2$} (7);
\draw [->]  (8) to[bend left=25] node[below] {$s_1$} (9);

\draw [->]  (5) to[bend left=25] node[right] {$s_1$} (4);

\draw [->]  (9) to[bend left=25] node[left] {$s_2$} (4);
\draw [->]  (5) to[bend left=25] node[above] {$s_2$} (6);

\draw [->]  (7) to[bend left=25] node[left] {$s_1$} (6);
\draw [->]  (7) to[bend left=25] node[right] {$s_2$} (8);
\draw [->]  (9) to[bend left=25] node[above] {$s_1$} (8);

\end{tikzpicture}
	\caption{Malnormal subgroup $H=\left\langle s_1 \right\rangle $ of $S_3=\left\langle\, s_1, s_2 \, |\, s_1^2, s_2^2, (s_1s_2)^3\, \right\rangle$. }
	\label{malnormal}
\end{figure}
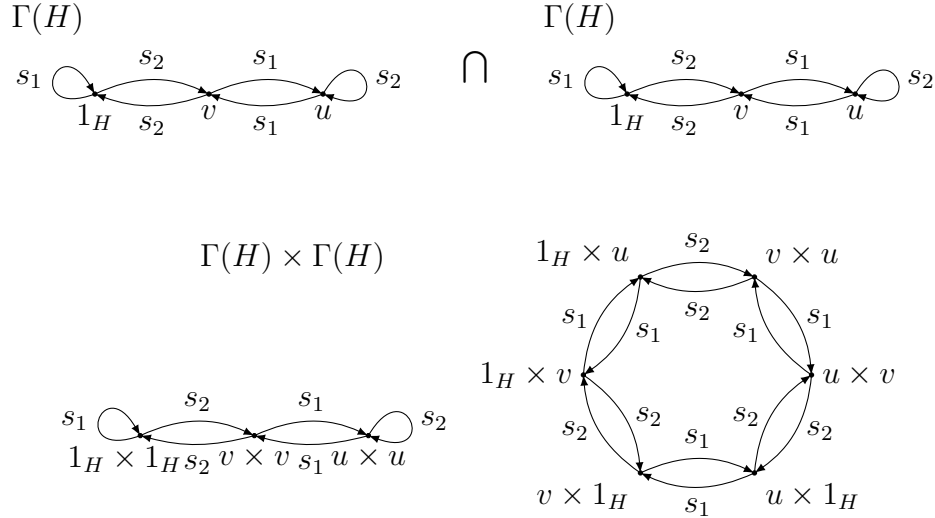
 
\begin{cor}\mbox{}\\
Let $G$ be a finite group and $H<G$ a subgroup of index $n$. If $H$ is malnormal, then $|G|$ divides $n^2-n$.  
\end{cor}

%%%%%%%%%%%%%%%%%%%%%%%%%%%%%%%%%%%%%%%

\section{Connectivity of the order complex $\Delta P_\f (G)$ for infinite finitely generated groups}
\label{secconnect}

%%%%%%%%%%%%%%%%%%%%%%%%%%%%%%%%%%%%%%%%%

In this section we use subgroup graphs to prove the connectivity of the order complex $\Delta P_\f (G)$ of the poset $(P_\f (G), \subseteq)$ with 
$$P_\f (G):=\{Hg \mid H< G, [G:H]<\infty, g\in G\}$$
 for a finitely generated group $G$. For this we use the nerve complex $\calnc(G, \scrh_\f )$ which is  homotopy equivalent to the order complex $\Delta P_\f (G)$. 
 If $G$ is a finite group, the poset $(P_\f (G), \subseteq)$ is the \textit{coset poset} $\mathscr C (G)$ (the poset of all left cosets of all proper subgroups of $G$, ordered by inclusion,  see \cite{B00}) and $\mathscr C (G)$ is not contractible, see \cite{SW14}. 
  In Subsection \ref{subsectype} we prove that $\Delta P_\f (G)$ is contractible  for some special classes of infinite finitely generated groups. These contain the free groups, free abelian groups,  Fuchsian groups of genus $g\geq 2$,  right angled Coxeter groups, Artin groups,  pure braid groups, Baumslag-Solitar groups, and infinite virtually cyclic groups.

\begin{definition}{(Order Complex, see \cite{SW14})}\label{deforder} \mbox{}\\
The \textit{order complex} $\Delta P$ of a poset $P$ is the simplicial complex whose $n$-simplices are the chains $C_0 \subset C_1 \subset ... \subset C_n$ of length $n$ (size $n+1$). 
\end{definition}

\begin{definition}{(Nerve Complex, see \cite{AH93})}\label{defnerve}\mbox{}\\
The \textit{nerve complex} $\calnc (G, \scrh)$ of a group $G$ with respect to a set $\scrh$ of subgroups of $G$ is a simplicial complex with an $n$-simplex for every set $\{H_0g_0,..., H_ng_n \}$ with $H_i\in \scrh$ such that $H_0g_0 \cap H_1g_1\cap... \cap H_ng_n\neq \emptyset$.
\end{definition}

Put $$\scrh_\f :=\{H \mid H< G, [G:H]<\infty\}.$$ Then the vertex set of $\calnc(G, \scrh_\f )$ is equal to the set $P_\f (G)$. In fact, the order complex $\Delta P_\f (G)$ is a subcomplex of the nerve complex $\calnc(G, \scrh_\f)$. The next theorem states that $\Delta P_\f (G)$ is homotopy equivalent to $\calnc (G, \scrh_\f)$. 
 It follows from the Theorem \cite[1.4]{AH93} by Abels and Holz.

\begin{theorem} \mbox{}\\
Let $\scrh$ be a set of subgroups of a group $G$ and let  $P$ be the set of all cosets of subgroups in $\scrh$. Suppose that $P$ has the following property: if $C$ and $C'$ are in $P$ and $C\cap C' \neq \emptyset$, then $C\cap C'\in P$. Then the nerve complex $\calnc(G, \scrh)$ and the order complex $\Delta P$ are homotopy equivalent. 
\end{theorem}

Now we consider a special subgroup graph.

\begin{proposition}\label{propGp}\mbox{}\\
Let $G= \langle \,X\, |\, R\, \rangle$ be a group with $X$ finite and $R$  not necessarily finite. Let   $(\Gamma_p, v)$ be  an $X$-regular connected graph which fulfills the defining relators $R$ and  has the following properties:  
 the $X$-graph $\Gamma_p$ has $p$ vertices  and there exists a freely reduced $X$-word $w_p$ such that $\{t(p_i) \mid o(p_i)=v, \mu(p_i)=w_p^i,  0\leq i < p\}$ and $V(\Gamma_p)$ are equal for the corresponding reduced paths $p_i$.  
Put  $H_p:=\phi(L(\Gamma_p,v))$. Let
 $H$ be another finite index subgroup of  $G$ with $w_p^m\in H$ for $0<m $. 
 If $m$ and $p$ are coprime, then $Hg\cap H_p \neq \emptyset $ for all $g\in G$.
\end{proposition}

\begin{proof}
The based $X$-graph $(\Gamma_p, v)$ is the subgroup graph of $H_p <G$. We denote $v$ by $1_{H_p}$. Let $p_g$ be the reduced path in $\Gamma_p$ with label $g$ and terminus $1_{H_p}$. Let $v_k:=o(p_g)$. By the assumption, there exists a reduced path $p_k$ with label $w^k$, origin $1_{H_p}$ and terminus $v_k$. The path $p_kp_g$ 
may not be reduced. But $\overline{w^kg}$ is the label of the reduced version of $p_kp_g$. Thus $\overline{w^kg} \in H_p$. Furthermore,  $\overline{w^{pz}w^kg} \in H_p$ for all $z \in \Z$. Since $p$ and $m$ are coprime, there exist integers $z, z'$ such that $pz+k=mz'$. Hence $\overline{w^{pz+k}g}=\overline{w^{mz'}g}\in Hg$.
\end{proof}

\begin{remark}
If $\Gamma_p$ is an $X$-graph as in Proposition \ref{propGp}, then $\{1, w_p, w_p^2, ..., w_p^{p-1}\}$ is a full set of coset representatives of $H_p=\phi(L(\Gamma_p,v))$ in $G$. In fact, if $H_p$ is normal in $G$, then $H_p\backslash G \cong \Z_p$.
\end{remark}

\begin{theorem}\label{thmcon}\mbox{}\\
Let $G= \langle \,X\, |\, R\, \rangle$ be a  group with $X$ finite and $R$  not necessarily finite. Suppose that 
there exists a collection $\{\Gamma_p\}_{p\in P}$  of $X$-graphs $\Gamma_p$  as in Proposition \ref{propGp} such that $P$ is a set of infinitely many primes. Suppose also that there exists a freely reduced $X$-word $w$ with $w_p=w$ uniformly for all $p\in P$. Then the nerve complex $\calnc (G, \scrh_\f)$ and  the order complex $\Delta P_\f (G)$ are contractible. 
\end{theorem}

\begin{proof}
A simplicial complex $K$ is contractible if and only if all finite subcomplexes are contractible in $K$. 
Let $U$ be a finite subcomplex of the nerve complex $\calnc (G, \scrh_\f)$. 
Then there exists a finite set $\Sigma$ consisting of all maximal simplices of $U$. Let $\bigcap \s$ be the intersection of the vertices of $\s\in \Sigma$. Thus $\bigcap \s=H_\s g_\s$. Let $m_\s >0$ be such that $w^{m_\s}\in H_\s$. Since $U$ is finite, we have only finitely many $m_\s$. Therefore we can find a prime $p\in P$ coprime to all $m_\s$.  
Hence $H_p \cap H_\s g_\s \neq \emptyset$, and consequently $\{H_p\}\cup \s$ is a simplex in $\calnc (G, \scrh_\f)$. Thus $H_p \ast U$ is a subcomplex  of $\calnc(G, \scrh_\f)$, and so $U$ is contractible in the nerve complex. 
\end{proof}

Since the collection $\{\Gamma_p\}_{p \in P}$ is infinite, the group $G$ has to be infinite. In the next subsections we give examples of groups satisfying these properties.

\subsection{The graph $\Gamma_p$}
\label{subsectype}

Assume we have two finitely generated groups $G= \langle \,X\, |\, R\, \rangle$  and $G'= \langle \,X\, |\, R'\, \rangle$. Let $\{\Gamma_p\}_{p\in P}$ and  $\{\Gamma'_p\}_{p\in P'}$ respectively be collections of $X$-graphs which satisfy the conditions of Theorem \ref{thmcon} for $G$ resp. $G'$.
 The collections may be equal, that is $P=P'$ and $\Gamma_p=\Gamma'_p$. In this case $\{1,w,w^2, ...,w^{p-1}\}$ is a full set of coset representatives for $H_p<G$ as well as $H_p'<G'$. If $ \left\langle \! \left\langle \, R\,
\right\rangle \!\right\rangle_{F(X)} \neq\left\langle \! \left\langle\, R'\,
\right\rangle \!\right\rangle_{F(X)}$, the subgroups $H_p$ and $H_p'$ may not be isomorphic.
Therefore we analyze the structure of the graph $\Gamma_p$, instead of the subgroup $H_p$.

\begin{definition}\mbox{}\\
Let $\Gamma$ be an $X$-graph. The graph $\Gamma|_Y$ with $Y\subset X$ is the subgraph of $\Gamma$ with the following properties:
 $e\in E(\Gamma|_Y)$ if and only if $e\in E(\Gamma)$ and $\mu(e)\in Y$, and $v\in V(\Gamma|_Y)$ if and only if there exists an edge $e\in E(\Gamma|_Y)$ such that $o(e)=v$ or $t(e)=v$. 
\end{definition}

If $\Gamma$ is $X$-regular then $V(\Gamma)=V(\Gamma|_Y)$ and $\Gamma|_Y$ is $Y$-regular for all $Y \subset X$. Furthermore,  $\Gamma|_{\{x\}}$ consists of $x$-circles of length $n_i$ with $|V(\Gamma)|=\sum n_i$. 

\subsubsection{Groups of Type I}
\begin{definition}{($(a,n)$-Circle)\\}
We call the $X$-regular graph with $n$ vertices and $n$ edges labeled $a$ for $a\in X$ the \textit{$a$-circle of length $n$ or the $(a,n)$-circle}. The graph is shown in Figure \ref{gamma_n}.
\end{definition}

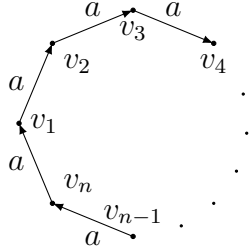
\begin{figure}[h!]
	\centering
 \begin{tikzpicture}%[>=stealth]

\coordinate[label=right:$v_1$] (1) at (180:1.5);
\coordinate[label=below right:$v_2$] (2) at (135:1.5); 
\coordinate[label=below:$v_3$] (3) at (90:1.5);
\coordinate[label=below:$v_4$] (4) at (45:1.5);
\coordinate[label=above:$v_{n-1}$] (5) at (270:1.5);
\coordinate[label=above right:$v_n$] (p) at (225:1.5);
%\coordinate[label=right:] (7) at (6,7);
\filldraw[black](1) circle (0.8pt) %Punkte
(2)circle(0.8pt)
(3)circle(0.8pt)
(4)circle(0.8pt)
(p)circle(0.8pt)
(295:1.5)circle(0.4pt)
(315:1.5)circle(0.4pt)
(335:1.5)circle(0.4pt)
(355:1.5)circle(0.4pt)
(5)circle(0.8pt)
(15:1.5)circle(0.4pt)

%(1,-1.2)circle(0.4pt)
;
\draw [->] (1) to node[left] {$a$} (2);
\draw [->] (2) to node[above] {$a$} (3);
 \draw[->] (3) to node[above] {$a$} (4);
\draw [->] (p) to node[left] {$a$} (1);
\draw [->] (5) to node[below] {$a$} (p);
%\path[->,min distance=1cm] (2) edge[in=90,out=190,above] node {$b$}(2);
%\path (1) edge [loop left] node{$b$} (1)
 %(2) edge [loop above ] node{$b$} (2)
% (3) edge [loop above] node{$b$} (3)
% (4) edge [loop above] node{$b$} (4)
%(p) edge [loop left] node{$b$} (p);
%(1) edge [bend right] node {a} (2)

\end{tikzpicture} 
	\caption{The $(a,n)$-circle is an $\{a\}$-regular connected core graph.}
	\label{gamma_n}
\end{figure}

\begin{definition}{(Groups of Type I)}\label{deftypI}\mbox{}\\
We call a group $G$ a \textit{group of Type I} if $G$ has a presentation $\left\langle\, X\, |\, R\, \right\rangle$, with $X$ finite and $R$  not necessarily finite, such that the following holds. 
There is an element $a\in X$ of infinite order. Moreover, there exists a collection $\{\Gamma_p\}_{p\in P}$ of $X$-graphs such that $P$ is a set of infinitely many primes, 
 each $\Gamma_p\in \{\Gamma_p\}_{p\in P}$ is a subgroup graph of a finite index subgroup of $G$, and each $\Gamma_p|_{\{a\}}$ is an  $(a,p)$-circle.  
\end{definition}

\begin{proposition}\mbox{}\\
If $G$ is a group of Type I, then the nerve complex $\calnc (G, \scrh_\f)$ and the order complex $\Delta P_\f(G)$ are contractible. 
\end{proposition}

\begin{proof}
The reduced path $p_k$ with label $a^k$ and origin $v_1$ has terminus $v_{k+1}$ in $\Gamma_p$ for $0\leq k<p$. Therefore Theorem \ref{thmcon} with $w=a$ completes the proof.
 \end{proof}
 
 Now we state examples of finitely generated groups of Type I. These contain the free groups, free abelian groups, Baumslag-Solitar groups,  Artin groups, pure braid groups, and infinite virtually cyclic groups $F\rtimes \mathbb Z$ with $F$ a finite group.

\begin{example}\label{extyp1}
(Groups of Type I)
\begin{itemize}
	\item \textit{The free group $F(X)$ with finite $X$}. For $a$ we can take any element of $X$. Let $\Gamma_p$ be an $(a,p)$-circle with a loop, labeled $x$, for all $a \neq x \in X$ at each vertex of $\Gamma_p$. The graph is shown in Figure \ref{gamma_p}. Since $F(X)=\left\langle \, X\, | \, \emptyset \, \right\rangle$, the graph $\Gamma_p$ is a subgroup graph for each $p\in  \mathbb P$. Thus $\{\Gamma_p\}_{p\in \mathbb P}$ satisfies the conditions of Definition \ref{deftypI}. 
	
	\item \textit{The groups $G\ast \mathbb Z$,  $G\times \mathbb Z$ and  $G \rtimes \mathbb Z$ with $Y$ finite, $X = Y \sqcup \{a\}$ and $G=\left\langle\, Y\, |\, R'\, \right\rangle$}. The collection $\{\Gamma_p\}_{p\in P}$ with $\Gamma_p$ as in Figure \ref{gamma_p} and $P=\mathbb P$ satisfies the conditions of Definition \ref{deftypI}.  We prove this as follows.
	
	$\left\langle\, X\, |\, R'\, \right\rangle$  is a presentation for $G \ast \mathbb Z$. All relators are words in $Y^{\pm 1}$ and each edge labeled $x \in Y$ is a loop in each $\Gamma_p\in\{\Gamma_p\}_{p\in \mathbb P} $. Consequently, the graph $\Gamma_p$ fulfills the relators $R'$ for each prime $p\in \mathbb P$. 
	
	We have $G\times \mathbb Z=\left\langle\, X\, |\, R\, \right\rangle$ with $R=R' \cup \{axa^{-1}x^{-1} \mid x \in Y\}$. Similarly as for $G \ast\Z$, the graph $\Gamma_p$ fulfills the relators $R'$. For the reduced path $v \xrightarrow{a} v'\xrightarrow{x} v' \xrightarrow{a^{-1}} v \xrightarrow{x^{-1}}v $ origin and terminus are equal for all $v\in V(\Gamma_p)$.  
	 Thus each $\Gamma_p$ fulfills the defining relators $R$.
	
 	$G\rtimes_\psi \mathbb Z=\left\langle\, X\, |\, R\, \right\rangle$, where $R=R' \cup \{axa^{-1}(\psi(a)(x))^{-1} \mid x \in Y\}$.  Since $\psi(a)(x)$ is a $Y$-word,  each graph $\Gamma_p$ fulfills the relators $R$, by the same argument as for $G\times \Z$. 
	
	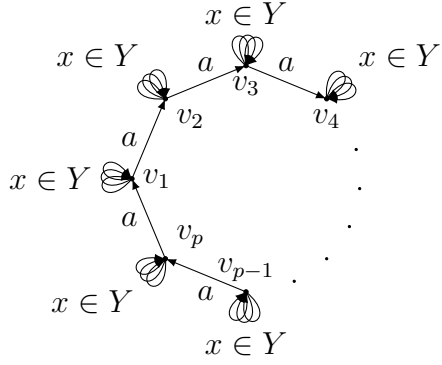
\begin{figure}[h!]
	\centering
 \begin{tikzpicture}%[>=stealth]

\coordinate[label=right:$v_1$] (1) at (180:1.5);
\coordinate[label=below right:$v_2$] (2) at (135:1.5); 
\coordinate[label=below:$v_3$] (3) at (90:1.5);
\coordinate[label=below:$v_4$] (4) at (45:1.5);
\coordinate[label=above:$v_{p-1}$] (5) at (270:1.5);
\coordinate[label=above right:$v_p$] (p) at (225:1.5);
%\coordinate[label=right:] (7) at (6,7);
\filldraw[black](1) circle (0.8pt) %Punkte
(2)circle(0.8pt)
(3)circle(0.8pt)
(4)circle(0.8pt)
(p)circle(0.8pt)
(295:1.5)circle(0.4pt)
(315:1.5)circle(0.4pt)
(335:1.5)circle(0.4pt)
(355:1.5)circle(0.4pt)
(5)circle(0.8pt)
(15:1.5)circle(0.4pt)

%(1,-1.2)circle(0.4pt)
;
\draw [->] (1) to node[left] {$a$} (2);
\draw [->] (2) to node[above] {$a$} (3);
 \draw[->] (3) to node[above] {$a$} (4);
\draw [->] (p) to node[left] {$a$} (1);
\draw [->] (5) to node[below] {$a$} (p);
\path[->,min distance=6mm]
 (1) edge[in=150,out=210,left] node {$x\in Y$}(1)
(2) edge[in=105,out=165,above left] node {$x\in Y\!$}(2)
(3) edge[in=60,out=120,above] node {$x\in Y$}(3)
(4) edge[in=15,out=75,above right] node {$x\in Y$}(4)
(5) edge[in=240,out=300,below] node {$x\in Y$}(5)
(p) edge[in=195,out=255,below left] node {$x\in Y$}(p)

(1) edge[in=170,out=230,above] (1)
(2) edge[in=125,out=185,above](2)
(3) edge[in=80,out=140,above] (3)
(4) edge[in=35,out=95,above] (4)
(5) edge[in=260,out=320,above] (5)
(p) edge[in=215,out=275,above] (p)

(1) edge[in=130,out=190,above] (1)
(2) edge[in=85,out=145,above] (2)
(3) edge[in=40,out=100,above] (3)
(4) edge[in=355,out=55,above] (4)
(5) edge[in=220,out=280,above] (5)
(p) edge[in=175,out=235,above] (p)
;

\end{tikzpicture} 
	\caption{The graph $\Gamma_p$ for the free group $F(X)$ with a loop, labeled $x$, for each $x\in Y=X\setminus \{a\}$ which is an $X$-regular connected core graph.}
	\label{gamma_p}
\end{figure} 	
 	
 	\item \textit{Free abelian groups $\mathbb Z^n$}, since $\mathbb Z^n=\mathbb Z^{n-1} \times \mathbb Z$ and $\Z=F(a)$.

 	\item \textit{The infinite virtually cyclic groups of the form $F\rtimes \mathbb Z$} with $F$ a finite group.

\item \textit{The orientation-preserving Fuchsian groups of genus one.} These are groups  $\left\langle\, X\, |\, R\, \right\rangle$ with  generators $X=\{a_1,b_1, x_1,..., x_d, y_1,...,y_s,z_1,...,z_t\}$ and relators $R=\{x_1^{m_1},..., x_d^{m_d}, x_1\cdots x_dy_1\cdots y_sz_1\cdots z_t[a_1,b_1]\}$ with $d,s,t\geq 0$ and $m_i \geq 2$, see \cite{LS04}.  Let $\Gamma_p$ be as in Figure \ref{gamma_p} with $a=a_1$. 
Hence $\Gamma_p$ fulfills the relators $x_i^{m_i}$. 
For $x_1 \cdots x_dy_1\cdots y_sz_1\cdots z_ta_1b_1a_1^{-1}b_1^{-1}$ we have $v \xrightarrow{x_1}v \xrightarrow{x_2} ...  \xrightarrow{z_t} v \xrightarrow{a_1} v' \xrightarrow{b_1} v' \xrightarrow{a_1^{-1}} v \xrightarrow{b_1^{-1}} v$ for different vertices $v, v' \in V(\Gamma_p)$. Therefore  $\{\Gamma_p\}_{p\in \mathbb P}$ is as in Definition \ref{deftypI}. 
	
\item \textit{Baumslag-Solitar groups $BS(m,n)=\left\langle\,  a,b \, |\, ab^na^{-1}b^{-m}\, \right\rangle$ for all integers $m$, $n$}.	Let $\Gamma_p$  be an $(a,p)$-circle with a loop labeled $b$ at each vertex of the circle. Therefore $v \xrightarrow{a} v' \xrightarrow{b} v' \xrightarrow{b} ... \xrightarrow{b} v' \xrightarrow{a^{-1}} v\xrightarrow{b^{-1}} v \xrightarrow{b^{-1}} v ... \xrightarrow{b^{-1}} v$. 
Hence the graph $\Gamma_p$ fulfills the relator for all $p\in \mathbb P$ and $ \{\Gamma_p\}_{p\in \mathbb P}$ satisfies the conditions of Definition \ref{deftypI}.

	\item \textit{Artin groups $A=\left\langle\, x_1,...,x_n \, |\, R\, \right\rangle$ (which include all braid groups)} such that $R=\{r_{i,j} \mid 1\leq i< j \leq n \}$ with $r_{i,j}=\left\langle x_i,x_j \right\rangle^{m_{i,j}}(\left\langle x_j,x_i \right\rangle^{m_{j,i}})^{-1}$,  $m_{i,j}=m_{j,i}$ for $i\neq j \in \{1,...,n\}$ and $\langle x_i, x_j \rangle^s$ an alternating product of $x_i$ and $x_j$ of length $s$ starting with $x_i$. 
	Let $\Gamma_p$ be an $X$-regular graph with $p$ vertices $v_1,..., v_p$ 
	such that the graph $\Gamma_p|_{\{x\}}$ is an $(x,p)$-circle for each $x\in X$. All $\Gamma_p|_{\{x\}}$ have the same direction. That is, for all $x\in X$ there is an edge labeled $x$ from $v_1$ to $v_2$. 
	The graph $\Gamma_p$ is shown in Figure \ref{artin}.
	 Starting in one vertex the reduced path $p_{r}$ with $\mu(p_r)=r_{i,j} $ leads $m_{i,j}$ vertices in the positive direction and then $m_{i,j}$ vertices in the negative direction. Hence $\Gamma_p$ fulfills the relators $R$ for every $p\in \mathbb P$. This gives us the collection $\{\Gamma_p\}_{p\in \mathbb P}$ as required in Definition \ref{deftypI}.

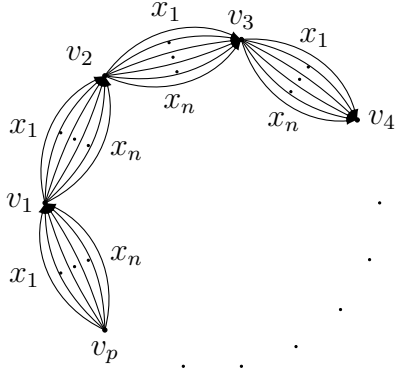
\begin{figure}[h]
	\centering
 \begin{tikzpicture}%[>=stealth]
\coordinate[label=left:$v_1$] (1) at (180:2.2);
\coordinate[label=above left:$v_2$] (2) at (130:2.2); 
\coordinate[label=above:$v_3$] (3) at (80:2.2);
\coordinate[label=right:$v_4$] (4) at (30:2.2);
%\coordinate[label=right:$5$] (5) at (1.5,0.5);
\coordinate[label=below:$v_p$] (p) at (230:2.2);
%\coordinate[label=right:] (7) at (6,7);
\filldraw[black](1) circle (0.8pt) %Punkte
(2)circle(0.8pt)
(3)circle(0.8pt)
(4)circle(0.8pt)
(p)circle(0.8pt)
(260:2.2)circle(0.4pt)
(280:2.2)circle(0.4pt)
(300:2.2)circle(0.4pt)
(320:2.2)circle(0.4pt)
(340:2.2)circle(0.4pt)
(0:2.2)circle(0.4pt)
(155:1.8)circle(0.4pt)
(105:1.8)circle(0.4pt)
(55:1.8)circle(0.4pt)
(205:1.8)circle(0.4pt)
(155:2.2)circle(0.4pt)
(105:2.2)circle(0.4pt)
(55:2.2)circle(0.4pt)
(205:2.2)circle(0.4pt)
(155:2.0)circle(0.4pt)
(105:2.0)circle(0.4pt)
(55:2.0)circle(0.4pt)
(205:2.0)circle(0.4pt)
;
\draw [->] (1) to[bend right=40] node[right] {$x_n$} (2);
\draw [->] (2) to[bend right=40] node[below] {$x_n$} (3);
 \draw[->] (3) to[bend right=40] node[below] {$x_n$} (4);
\draw [->] (p) to[bend right=40] node[right] {$x_n$} (1);
\draw [->] (1) to[bend left=40] node[left] {$x_1$} (2);
\draw [->] (2) to[bend left=40] node[above] {$x_1$} (3);
 \draw[->] (3) to[bend left=40] node[above] {$x_1$} (4);
\draw [->] (p) to[bend left=40] node[left] {$x_1$} (1);
\draw [->] (1) to [bend left=8]  (2);
\draw [->] (2) to [bend left=8] (3);
 \draw[->] (3) to [bend left=8]  (4);
\draw [->] (p) to [bend left=8]  (1);
\draw [->] (1) to [bend right=8]  (2);
\draw [->] (2) to [bend right=8] (3);
 \draw[->] (3) to [bend right=8]  (4);
\draw [->] (p) to [bend right=8]  (1);
\draw [->] (1) to [bend left=24]  (2);
\draw [->] (2) to [bend left=24] (3);
 \draw[->] (3) to [bend left=24]  (4);
\draw [->] (p) to [bend left=24]  (1);
\draw [->] (1) to [bend right=24]  (2);
\draw [->] (2) to [bend right=24] (3);
 \draw[->] (3) to [bend right=24]  (4);
\draw [->] (p) to [bend right=24]  (1);
%(1) edge [bend right] node {a} (2)

\end{tikzpicture} 
	\caption{The $X$-graph $\Gamma_p$ for the collections $\{\Gamma_p\}_{p\in \mathbb P}$ for Artin groups and pure braid groups with generators $X=\{x_1,...,x_n\}$. The graph $\Gamma_p$ is $X$-regular, connected and a core graph with respect to $v_i$.}
	\label{artin}
\end{figure}
	
	\item \textit{Pure braid groups $PB_m=\left\langle\, A_{ij}, 1\leq i< j \leq m \, |\, R_1, R_2, R_3, R_4\, \right\rangle$} with relators
	
 $\quad R_1=\{A_{rs}A_{ij}A_{rs}^{-1}A_{ij}^{-1}  \mid s< i \textup{ or } j <r\}$,
 
$\quad R_2=\{A_{rs}A_{ij}A_{rs}^{-1}A_{is}^{-1}A_{ij}^{-1}
A_{is}\mid  i< j=r < s \}$,

$\quad R_3=\{A_{rs}A_{ij}A_{rs}^{-1}A_{ij}^{-1}A_{ir}^{-1}A_{ij}^{-1}A_{ir}A_{ij}   \mid i<r<j =s\}$ and	

$\quad R_4=\{A_{rs}A_{ij}A_{rs}^{-1}A_{is}^{-1}A_{ir}^{-1}
A_{is}A_{ir}A_{ij}^{-1}
A_{ir}^{-1}A_{is}^{-1}A_{ir}A_{is}
 \mid i < r < j < s\}$. 
 
	Let $\Gamma_p$ be as in Figure \ref{artin}, where $n$ is the number of generators $A_{ij}$ and $X=\{A_{ij} \mid 1\leq i< j \leq m\}$. If we sum up the exponents of the $A_{ij}$ in each relator $r\in R_i$, it gives $0$. Therefore the graph $\Gamma_p$ fulfills the relators for all $p\in \mathbb P$ and $\{\Gamma_p\}_{p\in \mathbb P}$ is as required in Definition \ref{deftypI}.\\

\end{itemize}
\end{example}

\subsubsection{Group of Type II}

\begin{definition}{($(a,k,b,l)$-Graph)}\mbox{}\\
Let $k, l \geq 2$. We construct the following graph, which we call an \textit{$(a,k,b,l)$-graph}.   
We glue an $(a,k)$-circle with a $(b,l)$-circle over a single vertex $v$. We say that the circles share the vertex $v$. Then we glue this $(b,l)$-circle with a second $(a,k)$-circle over a different vertex. We glue the second $(a,k)$-circle  with a second $(b,l)$-circle. Repeating these steps we end with a $(b,l)$-circle. An $(a,k)$-circle and a $(b,l)$-circle share only one vertex.  
We add loops with label $a$ or $b$ such that the constructed graph is $\{a,b\}$-regular.
Thus to every not shared vertex of an $(a,k)$-circle we add a loop labeled $b$ and to every not shared vertex of a $(b,l)$-circle we add a loop  labeled $a$, see 
 Figure \ref{gamma_kl}. 
\end{definition}

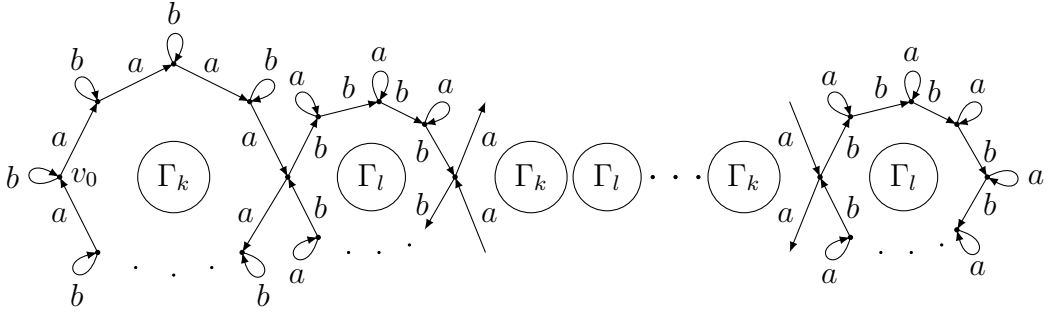
\begin{figure}[h]
	\centering
 \begin{tikzpicture}%[>=stealth]
  % \node(A) at (1,1) [circle, draw] {$\Gamma$};
%Gamma_p
\coordinate[label=right:$v_0$] (1) at (0,0);
\coordinate (2) at (0.5,1); 
\coordinate (3) at (1.5,1.5);
\coordinate(4) at (2.5,1);
\coordinate(42) at (2.4,-1);
\coordinate (p) at (0.5,-1);
\node (A) at (1.5,0) [circle, draw] {$\Gamma_k$};

%Gamma_q
\coordinate (5) at (3,0);
\coordinate(6) at (3.4,0.8); 
\coordinate (7) at (4.2,1);
\coordinate (8) at (4.8,0.7);
\coordinate (9) at (5.2,0);
\coordinate (10) at (4.8,-0.7);
\coordinate (q) at (3.4,-0.8);
\node (B) at (4.1,0) [circle, draw] {$\Gamma_l$};

\coordinate (11) at (5.6,1);
\coordinate (12) at (5.6,-1);

%Gamma_p
\filldraw[black](1) circle (0.8pt) %Punkte
(2)circle(0.8pt)
(3)circle(0.8pt)
(4)circle(0.8pt)
(42)circle(0.8pt)
(p)circle(0.8pt)
%(2.8,0.2)circle(0.4pt)
%(2.8,-0.2)circle(0.4pt)
%(2.7,-0.6)circle(0.8pt)
%(2.7,0.6)circle(0.4pt)
(1.5,-1.3)circle(0.4pt)
%(2.4,-1)circle(0.4pt)
(2,-1.2)circle(0.4pt)
(1,-1.2)circle(0.4pt);
\draw [->] (1) to node[left] {$a$} (2);
\draw [->] (2) to node[above] {$a$} (3);
 \draw[->] (3) to node[above] {$a$} (4);
\draw [->] (p) to node[left] {$a$} (1);
 \draw[->] (5) to node[left] {$a$} (42);% (2.7,-0.6);
%\path (1) edge [loop left] node{$b$} (1)
 %(2) edge [loop above] node{$b$} (2)
 %(3) edge [loop above] node{$b$} (3)
 %(4) edge [loop above] node{$b$} (4)
%(p) edge [loop left] node{$b$} (p);
\path[->,min distance=6mm] 
(1) edge[in=150,out=210,left] node {$b$}(1)
(2) edge[in=105,out=165,above] node {$b$}(2)
(3) edge[in=60,out=120,above] node {$b$}(3)
(4) edge[in=15,out=75,above] node {$b$}(4)
(p) edge[in=195,out=255,below] node {$b$}(p)
(42) edge[in=285,out=345,below] node {$b$}(42)
;
%(2.7,-0.6) edge [loop left] node{a} (2.7,-0.6);

\draw [->] (4) to node[left] {$a$} (5);

%Gamma_q
\filldraw[black](5) circle (0.8pt)
(6) circle (0.8pt)
(7) circle (0.8pt)
(8) circle (0.8pt)
(q) circle (0.8pt)
%(4.8,-0.7)circle(0.8pt)
(9)circle(0.8pt)
(3.8,-1)circle(0.4pt)
(4.2,-1)circle(0.4pt)
(4.6,-0.9)circle(0.4pt)
;
\draw [->] (5) to node[right] {$b$} (6);
\draw [->] (6) to node[above] {$b$} (7);
 \draw[->] (7) to node[above] {$b$} (8);
  \draw[->] (8) to node[left] {$b$} (9);
   \draw[->] (9) to node[left] {$b$} (10);
\draw [->] (q) to node[right] {$b$} (5);
\path[->,min distance=6mm] 
(6) edge[in=105,out=165,above] node {$a$}(6)
(7) edge[in=60,out=120,above] node {$a$}(7)
(8) edge[in=15,out=75,above] node {$a$}(8)
(q) edge[in=195,out=255,below] node {$a$}(q)
;

%Gamma_q
 \draw[->] (12) to node[right] {$a$} (9);
   \draw[->] (9) to node[right] {$a$} (11);
   
 % \filldraw[black] (11)circle(0.8pt)
  % (12)circle(0.8pt);
   \node () at (6.2,0) [circle, draw] {$\Gamma_k$};
     \node () at (7.2,0) [circle, draw] {$\Gamma_l$};
   \node () at (9,0) [circle, draw] {$\Gamma_k$};
      \node () at (11.1,0) [circle, draw] {$\Gamma_l$};
      
      \filldraw[black](7.8,0) circle (0.6pt) %Punkte
(8.1,0)circle(0.6pt)
(8.4,0)circle(0.6pt);
  % \coordinate[label=right:$\Gamma_q$]() at (5.8,0);
  % \coordinate[label=right:$\Gamma_p\ ...$]() at (7,0);
  % \coordinate[label=right:$\Gamma_q$]() at (8,0);
  % \coordinate[label=right:$\Gamma_p$]() at (9,0);
   %  \coordinate[label=right:$\Gamma_q$]() at (11.2,0);

     %Gamma_p
\coordinate (13) at (10,0);
\coordinate(14) at (10.4,0.8); 
\coordinate (15) at (11.2,1);
\coordinate (16) at (11.8,0.7);
\coordinate (17) at (12.2,0);
\coordinate (18) at (11.8,-0.7);
\coordinate (19) at (10.4,-0.8);

 \filldraw[black](13) circle (0.8pt)
(14) circle (0.8pt)
(15) circle (0.8pt)
(16) circle (0.8pt)
(17) circle (0.8pt)
%(4.8,-0.7)circle(0.8pt)
(18)circle(0.8pt)
(19)circle(0.8pt)
(10.8,-1)circle(0.4pt)
(11.2,-1)circle(0.4pt)
(11.6,-0.9)circle(0.4pt)
;
\draw [->] (13) to node[right] {$b$} (14);
\draw [->] (14) to node[above] {$b$} (15);
 \draw[->] (15) to node[above] {$b$} (16);
  \draw[->] (16) to node[right] {$b$} (17);
   \draw[->] (17) to node[right] {$b$} (18);
\draw [->] (19) to node[right] {$b$} (13);
\path[->,min distance=6mm] 
(14) edge[in=105,out=165,above] node {$a$}(14)
(15) edge[in=60,out=120,above] node {$a$}(15)
(16) edge[in=15,out=75,above] node {$a$}(16)
(17) edge[in=330,out=30,right] node {$a$}(17)
(18) edge[in=285,out=345,below] node {$a$}(18)
(19) edge[in=195,out=255,below] node {$a$}(19)
;

\draw [->] (9.6,1) to node[left] {$a$} (13);
\draw [->] (13) to node[left] {$a$} (9.6,-1);

\end{tikzpicture} 
	\caption{An $(a,k,b,l)$-graph.}
	\label{gamma_kl}
\end{figure}

\begin{definition}{(Groups of Type II)}\label{deftypII}\mbox{}\\
We call a group $G$ a \textit{group of Type II} if $G$ has a presentation $\left\langle\, X\, |\, R\, \right\rangle$, with $X$ finite and $R$  not necessarily finite,  such that the following holds. There exist two elements $a$ and $b$ in $X$ with $\ord(a)\in k\mathbb N_{>0}$, $\ord(b)\in l\mathbb N_{>0}$ and $\ord(ab)=\infty$. Moreover, there exists a collection $\{\Gamma_p\}_{p\in P}$ such that $P$ is a set of infinitely many primes, each $\Gamma_p$ is a subgroup graph of a subgroup of $G$, and $\Gamma_p|_{\{a,b\}}$ is an $(a,k,b,l)$-graph with $p$ vertices.  
\end{definition}

\begin{lemma}\label{lemkl}\mbox{}\\
Let $\Gamma$ be an $(a,k,b,l)$-graph with $m=k+l-1+(k+l-2)n$ vertices. Then $(ab)^m \in L(\Gamma,v_0)$ and $\{t(p_{m'}) \mid \mu(p_{m'})=(ab)^{m'}, o(p_{m'})=v_0, 0 \leq m'< m \}=V(\Gamma)$.
\end{lemma}

\begin{proof}
The proof is a part of the proof of Theorem \ref{thmast1}.
\end{proof}

\begin{remark}
An $(a,k,b,l)$-graph has $m=k+l-1+(k+l-2)n$ vertices with $n\in \mathbb N$. 
The numbers $k+l-1$ and $k+l-2$ are coprime for all $k,l \in \N_{>1}$. By Dirichlet's Theorem, there exist infinitely many $n$ such that $m$ is prime. 
\end{remark}

The next proposition follows from Theorem \ref{thmcon} and Lemma \ref{lemkl}.

\begin{proposition}\mbox{}\\
If $G$ is a group of Type II, then the nerve complex $\calnc (G, \scrh_\f)$ and the order complex $\Delta P_\f(G)$ are contractible. 
\end{proposition}

Now we state examples of groups of Type II. These contain free products of cyclic groups, infinite right angled Coxeter groups, Fuchsian groups of genus $g\geq 2$, and infinite virtually cyclic groups $A \ast_C B$.  

\begin{example}\label{extypII}
(Groups of Type II)
\begin{itemize}
\item \textit{The free product $\mathbb Z_s \ast \mathbb Z_t=\left\langle\,  a, b \, | \, a^s, b^t \, \right\rangle$}. Suppose that $k \mid s$ and $l \mid t$ and that $\Gamma_p$ is an $(a,k,b,l)$-graph with $p$  vertices. The graph $\Gamma_p$ consists of loops labeled $a$ or $b$,  $(a,k)$-circles and $(b,l)$-circles. 
Consequently,  the graph $\Gamma_p$ fulfills the relators. Since the set $P=\{p=k+l-1+(k+l-2)n \mid p \text{ prime} \}$ is infinite, the collection  $\{\Gamma_p\}_{p\in P}$ satisfies the conditions of Definition \ref{deftypII}.

\item \textit{The modular group} PSL$(2,\Z)$, since it is isomorphic to $\Z_2\ast \Z_3$.

\item \textit{The groups  $G \ast (\mathbb Z_s \ast \mathbb Z_t)$, $G \times (\mathbb Z_s \ast \mathbb Z_t)$ and $G \rtimes_\psi (\mathbb Z_s \ast \mathbb Z_t)$ with $X=Y\sqcup \{a,b\}$, $Y$ finite and $G=\left\langle\, Y\, |\, R'\, \right\rangle$}. Suppose that $k \mid s$ and $l \mid t$ and that $\Gamma_p $ is an $(a,k,b,l)$-graph with a loop, labeled $x$, for each $x\in Y$ at every vertex of $\Gamma_p$. The collection $\{\Gamma_p\}_{p\in P}$  with $P=\{p=k+l-1+(k+l-2)n \mid p \text{ prime} \}$ is as in Definition \ref{deftypII} required.  We prove this as follows.

$G \ast (\Z_s \ast \Z_t)= \left\langle\, X\, |\, R\, \right\rangle$ with $R= R' \cup \{a^s, b^t\}$. All edges labeled $x\in Y$ are loops. Consequently, each $\Gamma_p\in \{ \Gamma_p\}_{p\in P}$ fulfills the relators $R'$. 

$G \times (\Z_s \ast Z_t)=\left\langle\, X\, |\, R\, \right\rangle$ with $R=R' \cup \{a^s, b^t, axa^{-1}x^{-1}, bxb^{-1}x^{-1} \mid x\in Y \}$. By the same arguments as for the groups $G \ast (\Z_s \ast \Z_t)$ and $G \times \Z$, each graph $\Gamma_p\in \{ \Gamma_p\}_{p\in P}$ fulfills the relators $R$. 

The group $G \rtimes_\psi (\Z_s \ast Z_t)$  has a presentation of the form  $\left\langle\, X\, |\, R\, \right\rangle$ with relators $R=R' \cup \{a^s, b^t, axa^{-1}(\psi(a)(x))^{-1}, bxb^{-1}(\psi(b)(x))^{-1} \mid  x\in Y \}$. By the same arguments as for $G \ast (\Z_s \ast \Z_t)$ and $G \rtimes \Z$, each $\Gamma_p$ fulfills the relators $R$.

\item \textit{The infinite right angled Coxeter groups $W=\left\langle\, s_1,...,s_n \, |\, s_i^2,  (s_is_j)^{m_{i,j}} \, \right\rangle$}. Then $m_{i,j}\in \{2, \infty \}$ and at least one $m_{i,j}=\infty$.
For $m_{\alpha, \beta}=\infty$ let $a:=s_\alpha$ and $b:=s_\beta$. Let $\Gamma_p$ be an $(a,2,b,2)$-graph with a loop, labeled $s_i$, at every vertex for all $s_i$ with $i \neq \alpha, \beta$ (see Figure \ref{rcox}). 
 Consequently, $\Gamma_p$ fulfills the relators $s_i^2$ and $(s_is_j)^{m_{i,j}}$ with $i,j \neq \alpha, \beta$. 
 We have $v\xrightarrow{a} v' \xrightarrow{a} v$ and $v \xrightarrow{a} v' \xrightarrow{s_i} v'\xrightarrow{a} v \xrightarrow{s_i} v$ for all $s_i\neq a, b$ and analogously $v\xrightarrow{b^2} v$ and $v \xrightarrow{bs_ibs_i} v$ for all $s_i\neq a, b$.  Thus $\Gamma_p$ fulfills the relators  for all  $p\in P=\mathbb P_{>2}$ and $\{\Gamma_p\}_{p\in  P}$ satisfies the conditions of Definition \ref{deftypII}.

 \begin{figure}[h!]
	\centering
 \begin{tikzpicture}
\coordinate[label=below:$v_1$] (1) at (0,0);
\coordinate[label=below:$v_2$] (2) at (1.5,0); 
\coordinate[label=below:$v_3$] (3) at (3,0);
\coordinate[label=below:$v_4$] (4) at (4.5,0);
 \coordinate[label=below:$v_{p-1}$] (5) at (7.1,0);
 \coordinate[label=below:$v_p$] (p) at (8.6,0);
 \filldraw[black](1) circle (0.8pt)
 (2) circle (0.8pt)
 (3) circle (0.8pt)
 (4) circle (0.8pt)
 (5) circle (0.8pt)
 (p) circle (0.8pt)
 (5.6,0) circle (0.4pt)
 (5.8,0) circle (0.4pt)
 (6.0,0) circle (0.4pt)
 ;
 \draw [->] (1) to[bend left=12] node[above] {$a$} (2);
\draw [->] (2) to[bend left=12] node[above] {$b$} (3);
 \draw[->] (3) to[bend left=12] node[above] {$a$} (4);
\draw [->] (5) to[bend left=12] node[above] {$b$} (p);
 \draw [->] (2) to[bend left=12] node[below] {$a$} (1);
\draw [->] (3) to[bend left=12] node[below] {$b$} (2);
 \draw[->] (4) to[bend left=12] node[below] {$a$} (3);
\draw [->] (p) to[bend left=12] node[below] {$b$} (5);
\draw [->] (4) to[bend left=6] node[above right] {$b$} (5.3,0.1);
\draw [->] (5.3,-0.1) to[bend left=6] node[below right] {$b$} (4);
\draw [->] (6.3,0.1) to[bend left=6] node[above left] {$a$} (5);
\draw [->] (5) to[bend left=6] node[below left] {$a$} (6.3,-0.1);

\path [->, min distance=7mm]
 (1) edge[in=160,out=200,left]  node{$b$} (1)
 (1) edge[in=70,out=110,above] node{$s_i\in Y$} (1)
 (2) edge[in=70,out=110,above] node{$s_i\in Y$} (2)
 (3) edge[in=70,out=110,above] node{$s_i\in Y$} (3)
(4) edge[in=70,out=110,above] node{$s_i\in Y$} (4)
(5) edge[in=70,out=110,above] node{$s_i\in Y$} (5)
(p) edge[in=70,out=110,above] node{$s_i\in Y$} (p)
(1) edge[in=50,out=90] (1)
 (2) edge[in=50,out=90]  (2)
 (3) edge[in=50,out=90] (3)
(4) edge[in=50,out=90]  (4)
(5) edge[in=50,out=90]  (5)
(p) edge[in=50,out=90]  (p)
(1) edge[in=90,out=130] (1)
 (2) edge[in=90,out=130]  (2)
 (3) edge[in=90,out=130] (3)
(4) edge[in=90,out=130]  (4)
(5) edge[in=90,out=130]  (5)
(p) edge[in=90,out=130]  (p)
(p) edge[in=340,out=20,right] node{$a$} (p);
\end{tikzpicture} 
	\caption{The $X$-graph $\Gamma_p$ for an infinite right angled Coxeter group with $X$ finite and $Y=X\setminus \{a,b\}$. The graph $\Gamma_p$ is an $X$-regular connected core graph with respect to every vertex.}
	\label{rcox}
\end{figure}
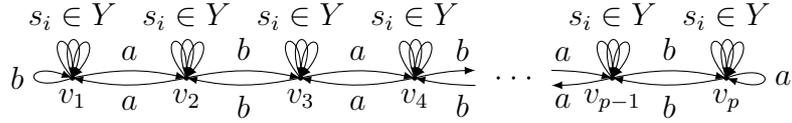

\item \textit{The infinite Coxeter groups $W=\left\langle \, s_1,...,s_n \, | \, s_1^2,...,s_n^2, (s_is_j)^{m_{i,j}} \, \right\rangle$ such that $m_{\alpha, \beta}=\infty$, $m_{\alpha,i}, m_{\beta,i} \in 2\mathbb N_{>0} \cup \infty$ and arbitrary $m_{i,j}$ for $i,j \neq \alpha, \beta$}. We use the same collection  $\{\Gamma_p\}_{p\in P}$  as for the right angled Coxeter groups. Therefore $\Gamma_p$ fulfills $s_l^2$ and $(s_is_j)^{m_{i,j}}$ for $i,j \neq \alpha, \beta$. Since $\Gamma_p$ fulfills $(s_\alpha s_i)^{m_{\alpha, i}}$ and $(s_\beta s_i)^{m_{\beta, i}}$ for $m_{\alpha,i}, m_{\beta,i} \in \{2, \infty\}$, it fulfills $m_{\alpha,i}, m_{\beta,i} \in 2\N_{>0} \cup \infty$.

\item \textit{The alternating subgroup $W^+$ of an infinite Coxeter group $W$ as in the previous example}. $W^+=\left\langle \, s_ks_i, i\neq k, 1\leq i\leq n \, |\, (s_ks_i)^{m_{k,i}},  ((s_ks_i)^{-1}s_ks_j)^{m_{i,j}}\, \right\rangle$ with fixed $k$ such that there are $ \alpha, \beta\neq k$ with $m_{\alpha, \beta}= \infty$.  Then let $a:=s_ks_\alpha$ and $b:=s_ks_\beta$. Let the graph $\Gamma_p$ be an $(a,2,b,2)$-graph with a loop, labeled $s_ks_i$, at every vertex for each $s_ks_i$ with $i \neq \alpha, \beta$.  Thus $\Gamma_p$ fulfills the relators $(s_ks_i)^{m_{k,i}}$ and  $((s_ks_i)^{-1}s_ks_j)^{m_{i,j}} $  for $i,j \neq \alpha, \beta$. Since $m_{k,\alpha}, m_{k,\beta}, m_{\alpha, i}, m_{\beta,i} \in 2\N_{>0}\cup \infty$, the graph $\Gamma_p$ fulfills the relators for all $p\in \mathbb P_{>2}$. This gives us the collection $ \{\Gamma_p\}_{p\in \mathbb P_{>2}}$. 

\item \textit{The orientation-preserving Fuchsian groups of genus $g\geq 2$}. These are  groups  $\left\langle\, X\, | \, R \, \right\rangle$ with $X=\{  a_1, b_1,...,a_g,b_g, x_1,...,x_d, y_1,...,y_s, z_1,..,z_t\}$ and $R=\{ x_1^{m_1}, ..., x_d^{m_d}, x_1 \cdots x_dy_1\cdots y_sz_1  \cdots z_t[a_1,b_1]\cdots [a_g,b_g]\}$ with
 $d,s, t \geq 0$,  $g\geq 2$ and $m_i\geq 2$, see \cite{LS04}. Let $\Gamma_p$ be an $(a_1,2,a_2,2)$-graph with a loop, labeled $x$, for each $x\in X\setminus \{a_1, a_2\}$ at every vertex of $\Gamma_p$. Thus $\Gamma_p$ fulfills the relators $x_i^{m_i}$. Since all other parts are just loop,
the important part of the relator $x_1 \cdots x_dy_1\cdots y_sz_1  \cdots z_t[a_1,b_1]\cdots [a_g,b_g]$ is $a_1a_1^{-1}a_2a_2^{-1}$, which $\Gamma_p$ fulfills. This gives the collection $ \{\Gamma_p\}_{p\in \mathbb P_{>2}}$ which satisfies the conditions of Definition \ref{deftypII}.  

\item \textit{The non-orientation-preserving Fuchsian groups of genus $g\geq 2$}. These are groups $\left\langle\, X\, | \, R \, \right\rangle$ with $X=\{  a_1,...,a_g, x_1,...,x_d, y_1,...,y_s, z_1,..,z_t\}$ and relators  $R=\{ x_d^{m_d},..., x_d^{m_d}, x_1 \cdots x_dy_1\cdots y_sz_1  \cdots z_ta_1^2\cdots a_g^2\}$ with $g\geq 2$, $d,s, t \geq 0$ and $m_i\geq 2$, see \cite{LS04}. Let $\Gamma_p$ be an 
$(a_1,2,a_2,2)$-graph with a loop, labeled $x$, for each $x\in X\setminus \{a_1, a_2\}$ at every vertex of $\Gamma_p$. Thus the important part of the relator $x_1 \cdots x_dy_1\cdots y_sz_1  \cdots z_ta_1^2\cdots a_g^2$ is $ v \xrightarrow{a_1} v' \xrightarrow{a_1} v  \xrightarrow{a_2} v'' \xrightarrow{a_2} v$. Moreover, $\Gamma_p$ fulfills all relators $x_i^{m_i}$. Consequently, $ \{\Gamma_p\}_{p\in \mathbb P_{>2}}$ is as in Definition \ref{deftypII} required.

\item \textit{Infinite virtually cyclic groups of the form $A \ast_{C_A=C_B} B$} with $A$ and $B$ finite groups and $C_A<A$, $C_B<B$ subgroups of index $2$. 
A presentation is $\left\langle \, a, b, c_1,..., c_k,\psi(c_1),...,\psi(c_k) \, | \, R_A, R_B,  c_i\psi(c_i)^{-1} \, \right\rangle$, where $\left\langle  c_1,...,c_k \right\rangle=C_A$,  $A= \left\langle\, a, c_1,...,c_k \, | \, R_A \, \right\rangle$ and $B= \left\langle\, b, \psi(c_1),...,\psi(c_k) \, | \, R_B \, \right\rangle$ for $\psi\colon C_A \rightarrow C_B$ an isomorphism.
 Let $\Gamma_p$ be an $(a,2,b,2)$-graph with a loop, labeled $c_i$,  for each $c_i$ and a loop, labeled $\psi(c_i)$, for each $\psi(c_i)$ at every vertex. Then $\Gamma_p$ fulfills the relators $c_i\psi(c_i)^{-1}$. The connected component of $\Gamma_p|_{\{ a, c_1,...,c_k\}}$ is either a graph with one vertex and a loop for $a$ and all $c_i$  or an $(a,2)$-circle with a loop for every $c_i$ at both vertices. 
 Thus each connected component of $\Gamma_p|_{\{ a, c_1,...,c_k\}}$ is either the subgroup graph of $A$ or of $C_A$ in $A$.
 Hence $\Gamma_p$ fulfills the relators $R_A$. Analogously, we prove that $\Gamma_p$ fulfills the relators $R_B$. 
  It follows that the graph $\Gamma_p$ fulfills each relator of the given presentation for all odd prime. Consequently, the collection $ \{\Gamma_p\}_{p\in \mathbb P_{>2}}$ satisfies the conditions of Definition \ref{deftypII}.

  \item \textit{The amalgamated product $A \ast_{D} B$ of two finitely generated groups $A$ and $B$ with $\{1\}\leq D\leq C_A$ and $C_A\cong C_B$ subgroups of index $2$ in $A$ and $B$.} We have $\left\langle \, a, b, c_1,..., c_k,\psi(c_1),...,\psi(c_k) \, | \, R_A, R_B,  d\psi(d)^{-1},   d \in D \, \right\rangle$  with 
   $A$, $B$, $C_A$ generated as in the previous example $A\ast_{C_A=C_B} B$. Furthermore, we use the same collection $ \{\Gamma_p\}_{p\in \mathbb P_{>2}}$.  Consequently, $\Gamma_p$ fulfills the relators $R_A$ and $R_B$.
   Since $D\leq C_A$, each $d\in D$ is a word in $\{c_1^{\pm 1}, ..., c_k^{\pm 1}\}$ and every $\psi(d)$ a word in $\{\psi(c_1)^{\pm 1},..., \psi(c_k)^{\pm 1}\}$. Every edge labeled $c_i$ or $\psi(c_i)$ is a loop in the graph $\Gamma_p$. Therefore each $\Gamma_p\in \{\Gamma_p\}_{p\in \mathbb P_{>2}}$ fulfills the relator $d\psi (d)^{-1}$ for all $d\in D$.   
  
  \end{itemize}
\end{example}

There are more types of groups that could be considered, but we will not develop these types here.  
 
%%%%%%%%%%%%%%%%%%%%%%%%%%%%%%%%%%%

\section{Sufficient conditions for the contractibility of $\Delta P_\f (G)$ }
\label{secsuf}

%%%%%%%%%%%%%%%%

After studying some special types of groups in Subsection \ref{subsectype}, we are now interested in more general statements. 
We prove a sufficient condition for an infinite finitely generated group $G_1$ such that the order complex $\Delta P_\f (H)$ of a finite index subgroup $H< G_1$ inherited the contractibility  of the order complex $\Delta P_\f (G_1)$.  Furthermore,  we prove sufficient conditions for finitely generated groups $G_1$ and $G_2$ such that the order complexes $\Delta P_\f (G_1 \ast G_2)$, $\Delta P_\f (G_1 \times G_2)$ and $\Delta P_\f (G_1 \ltimes G_2)$ are contractible. Each group of Subsection \ref{subsectype} can be  chosen for $G_1$ and $G_2$. We end this section with a condition for the connectivity of an amalgamated product $G_1 \ast_D G_2$.

\begin{theorem}\label{thmconsub}\mbox{}\\
Let $G=\left\langle\,  X \, | \, R\, \right\rangle$ be a  group with $X$ finite and $R$  not necessarily finite such that there is a collection $\{\Gamma_p\}_{p\in P}$ of $X$-graphs as in Theorem \ref{thmcon}. 
Let $H$ be a finite index subgroup of $G$. Then the nerve complex $\calnc(H, \scrh_\f )$ and the order complex $\Delta P_\f (H)$ are contractible.
\end{theorem}

\begin{proof}
The sets $\{H \cap Kh\neq H \mid h \in H, Kh \in P_\f (G) \}$ and  $P_\f (H)$ are equal since $H$ is of finite index in $G$. Hence the nerve complex $\calnc(H, \scrh_\f)$  is a subcomplex of $\calnc(G, \scrh_\f)$. 
Let $\Sigma_U$ be the set of all maximal simplices in $U$ and let $\bigcap \s =H_\s g_\s$. Let $m_\s>0$ be such that $w^{m_\s}\in H_\s$. 
 By Theorem \ref{thmcon}, there exists a $p\in P$ for every finite subcomplex $U$ of $\calnc(G, \scrh_\f)$ such that the join $H_p \ast U$ is a subcomplex of $\calnc(G, \scrh_\f)$ and $p$ and $m_\s$ are coprime for all $\s\in \Sigma_U$. 
 Therefore $H_p \cap H_\s g_\s \neq \emptyset$  for all  $\s \in \Sigma_U$. If $U$ is a finite  subcomplex of $\calnc(H, \scrh_\f)$, then $H_\s g_\s \subset H$. 
 Thus $H\cap (H_p \cap H_\s g_\s) \neq \emptyset$. If $H\cap H_p\neq H$, then $(H\cap H_p) \ast U$ is a contractible subcomplex of $\calnc(H, \scrh_\f)$. Suppose $H\leq H_p$. Then $H_\s < H_p$. Consequently, $w^{m_\s}\in H_p$. By the properties of $\Gamma_p$, we have $w^i\in H_p$ if and only if $i\in p\Z$. Therefore $p \mid m_\s$, contrary to $p$ and $m_\s$ are coprime. Thus $H\nleq H_p$, which completes the proof.
\end{proof}

Theorem \ref{thmast1} and \ref{thmast2} both deal with free products of groups satisfying  some conditions. The difference is that in the first case $G_1$ and $G_2$ have to satisfy the same condition while in the second case only $G_1$ has to satisfy some conditions.

\begin{theorem}\label{thmast1}\mbox{}\\
Let $G_1=\left\langle \, X_1 \, | \, R_1 \,\right\rangle$ and $G_2=\left\langle \,  X_2 \, | \, R_2 \,\right\rangle$ be groups with $X_1$, $X_2$ finite and $R_1$, $R_2$ not necessarily finite. Suppose that there exist proper subgroups $H_1<  G_1$, $H_2 < G_2$ such that $\{1, w_1, w_1^2, ..., w_1^{n_1-1}\}$, $\{1,w_2, w_2^2,..., w_2^{n_2-1}\}$ are full sets of coset representatives of $H_1$, $H_2$ in $G_1$, $G_2$ for some $w_1\in G_1$, $w_2\in G_2$. Then the nerve complex $\calnc (G_1 \ast G_2, \scrh_\f )$ and the order complex $\Delta P_\f (G_1 \ast G_2)$  are contractible. 
\end{theorem}

\begin{proof}
Let $\Gamma(H_i)=\Gamma_{X_i,R_i}(H_i)$ be the subgroup graph of $H_i< G_i$ for $i=1,2$. We consider the $(X_1\cup X_2)$-graph $\Gamma_p$ with $p=n_1+n_2-1+(n_1+n_2-2)n$ vertices. Since $n_1,n_2>1$, there exist infinitely many $n\in \mathbb N$ such that $p$ is prime. The graph $\Gamma_p$ is built from a copy of the graph $\Gamma(H_1)$ glued with a copy of $\Gamma(H_2)$ at a single vertex $v_{11}$. This $\Gamma(H_2)$ is glued with a copy of $\Gamma(H_1)$ at another single vertex $v_{12}$ and so on.  
We end up with a copy of $\Gamma(H_2)$, which is only glued with the previous $\Gamma(H_1)$ over the vertex $v_{nn}$. At every vertex $v\neq v_{ss}, v_{ss+1}$ of the copies of $\Gamma(H_i)$ we add a loop, labeled $x$, for each $x\in X_j$ with $i\neq j$. For $\Gamma(H_i)^{(s)}|_{X_i}=\Gamma(H_i)$ the graph $\Gamma_p$ looks like this:
$$1_{H_p}\Gamma(H_1)^{(1)}\, v_{11}\, \Gamma(H_2)^{(1)}\, v_{12}\, \Gamma(H_1)^{(2)}\, v_{22}\, ... \, \Gamma(H_2)^{(n-1)}\, v_{n-1n}\, \Gamma(H_1)^{(n)}\, v_{nn}\, \Gamma(H_2)^{(n)}$$ 
with $v_{kl}=V(\Gamma(H_i) ^{(k)})\cap V(\Gamma(H_j)^{(l)})$ and $1_{H_p}\in V(\Gamma(H_1)^{(1)})$ is the base-vertex of $\Gamma_p$ with $v_{11}\neq 1_{H_p}$.  (Figure \ref{bsp} shows an example.) We prove that for each vertex $v$ in $\Gamma_p$ there exists a reduced path with label $(w_1w_2)^k$ from $1_H$ to $v$.

Let $v\in \Gamma(H_1)^{(s)}$.  Then there exist $0 \leq m_1, m_1'< n_1$ such that $v_{s-1s} \xrightarrow{w_1^{m_1}} v$ and $v_{ss} \xrightarrow{w_1^{m_1'}} v$ are reduced paths. For $v_{01}$ we take the base-vertex $1_{H_p}$. 

Case $m_1 <m_1'$: Put $k_1:=m_1$. Then $v_{s-1s} \xrightarrow{(w_1w_2)^{k_1}} v$ is a reduced path. This follows from the facts that  an edge $e$ labeled $x\in X_2$ is a loop in  $\Gamma(H_1)^{(s)}$ if and only if $o(e)\neq v_{s-1s}, v_{ss}$, and, in this case, the paths in $\Gamma(H_1)^{(s)}$ with label $w_1^k$ for $0< k \leq k_1$ and origin $v_{s-1s}$ never have terminus $v_{s-1s}$ or $v_{ss}$.  
If $m_1=0$, then $v=v_{s-1s}$. 
 Since $\Gamma(H_2)^{(s-1)}|_{X_2}=\Gamma(H_2)$, there exists the reduced path $v_{s-1s-1} \xrightarrow{w_2^{k_2+1}} v_{s-1s}$.
 An edge $e$ labeled $x\in X_1$ is a loop in $\Gamma(H_2)^{(s-1)}$ as long as $o(e)\neq v_{s-1s-1}, v_{s-1s}$. Furthermore,  the reduced paths  in $\Gamma(H_2)^{(s-1)}$ with label $w_2^k$ for $0< k \leq k_2$ and origin $v_{s-1s-1}$ never have terminus $v_{s-1s-1}$ or $v_{s-1s}$. Consequently, $v_{s-1s-1}\xrightarrow{w_2(w_1w_2)^{k_2}} v_{s-1s}$ is a reduced path.
The vertices $v_{s-2s-1}$ and $v_{s-1s-1}$ are in $\Gamma (H_1)^{(s-1)}$. Hence  $v_{s-2s-1}\xrightarrow{w_1^{k_3+1}} v_{s-1s-1}$ is a reduced path. 
This gives the reduced path $v_{s-2s-1}\xrightarrow{(w_1w_2)^{k_3}w_1} v_{s-1s-1}$. 
Repeating these steps, we end in $\Gamma(H_1)^{(1)}$ at vertex $v_{11}$. There exists a reduced path $1_{H_p}\xrightarrow{w_1^{k_{\nu}+1}} v_{11}$. Thus $1_{H_p} \xrightarrow{(w_1w_2)^{k_{\nu}}w_1} v_{11}$ is a reduced path.  Therefore we have \\
 $1_{H_p}\!\xrightarrow{(w_1w_2)^{k_\nu}w_1}\!v_{11}\! \xrightarrow{w_2(w_1w_2)^{k_{\nu-1}}}\! v_{12}\ ...\ \!\xrightarrow{(w_1w_2)^{k_3}w_1}\!v_{s-1s-1}\!\xrightarrow{w_2(w_1w_2)^{k_2}}\! v_{s-1s}\! \xrightarrow{(w_1w_2)^{k_1}}\! v.$\vspace{2mm}
 Thus there is a reduced path with label $(w_1w_2)^k$ from $1_{H_p}$ to $v$. 
 
 Case $m_1>m_1'$: Put $l_1:=m_1'$. By the arguments above, $v_{ss}\xrightarrow{(w_1w_2)^{l_1}} v$ is a reduced path in $\Gamma_p$. If $m_1'=0$, then $v=v_{ss}$. 
 We need a path with terminus $v_{ss}$ whose label ends with $w_2$. Hence we consider a path in $\Gamma(H_2)^{(s)}$. There is the reduced  path $v_{ss+1}\xrightarrow{w_2^{l_2+1}}v_{ss}$. As above $v_{ss+1} \xrightarrow{w_2(w_1w_2)^{l_2}} v_{ss}$ is a reduced path in $\Gamma_p$. 
 In $\Gamma(H_1)^{(s+1)}$ exists the reduced path $v_{s+1s+1}\xrightarrow{w_1^{l_3+1}} v_{ss+1}$. Therefore we have $v_{s+1s+1} \xrightarrow{(w_1w_2)^{l_3}w_1} v_{ss+1}$. Repeating this, we reach the vertex $v_{nn}$  with the reduced  path $v_{nn} \xrightarrow{(w_1w_2)^{l_\tau} w_1} v_{n-1n}$. We have to find a path whose label ends with $w_2$. Since $\Gamma(H_2)$ has $n_2$ vertices, $v_{nn}\xrightarrow{w_2^{n_2}} v_{nn}$ is a reduced path in $\Gamma(H_2)^{(n)}$. Thus we have $v_{nn}\xrightarrow{w_2(w_1w_2)^{n_2-1}} v_{nn}$. At last there is the reduced path $v_{n-1n} \xrightarrow{(w_1w_2)^{n_1-l_{\tau}-2}w_1} v_{nn}$. 
 The vertex $v_{n-1n}$ is in $\Gamma(H_1)^{(n)}$. Hence we are in the case $m_1<m_1'$. 
 Consequently, there exist $k,l \in \N$ such that $1_{H_p}\xrightarrow{(w_1w_2)^k} v_{n-1n}\xrightarrow{(w_1w_2)^l} v$ is a reduced path.
 
 Let $v\in V(\Gamma(H_2)^{(s)})\setminus \{v_{ss}, v_{ss+1}\}$. Then there exist $0 < m_2, m_2'< n_2$ such that $v_{ss} \xrightarrow{w_2^{m_2}} v$ and $v_{ss+1} \xrightarrow{w_2^{m_2'}} v$ are reduced paths. 
 
 Case $m_2<m_2'$: There are paths $v_{ss}\xrightarrow{w_2(w_1w_2)^{m_2-1}} v$ and $v_{s-1s} \xrightarrow{(w_1w_2)^{\mu}w_1} v_{ss}$ in $\Gamma_p$. The vertex $v_{s-1s}$ is in $\Gamma(H_1)^{(s)}$. Thus we are in the case $m_1<m_1'$. Hence there is a reduced path as required. 

 Case $m_2>m_2'$: We get paths $v_{ss+1} \xrightarrow{w_2(w_1w_2)^{m_2'-1}} v$ and  $v_{s+1s+1} \xrightarrow{(w_1w_2)^{\nu}w_1} v_{ss+1}$ in $\Gamma_p$. The vertex $v_{s+1s+1}$ is in $\Gamma(H_1)^{(s+1)}$ hence we are in the case $m_1>m_1'$.

Therefore  $\{t(p_k) \mid \mu(p_i)=(w_1w_2)^k, o(p_k)=1_{H_p}, 0\leq k<p\}=V(\Gamma_p)$ holds. Every connected component of $\Gamma_p|_{X_i}$ is either $\Gamma_{X_i,R_i}(H_i)$ or $\Gamma_{X_i,R_i}(G_i)$. Hence $\Gamma_p$ fulfills the relators $R_1$ and $R_2$. By Theorem \ref{thmcon}, the nerve complex $\calnc(G_1  \ast G_2, \scrh_\f)$ and the order complex $\Delta P_\f (G_1 \ast G_2)$ are contractible and $\{1,w_1w_2, ..., (w_1w_2)^{p-1}\}$  is a full set of coset representatives of $H_p=\phi(L(\Gamma_p,1_{H_p}))$ in $G_1 \ast G_2$. 
\end{proof}

 Figure \ref{bsp} shows an example of a subgroup graph constructed as in the proof of Theorem \ref{thmast1} for $G_1=\left\langle \, a, b, c \, |\, a^2, b^2, c^2, (ab)^3, (bc)^3, (ac)^3\, \right\rangle$ with $H_1=\left\langle b, c, abcba \right\rangle$ and $w_1=abc$ and $G_2=\left\langle\, d \, | \, d^2 \, \right\rangle$ with $H_2=\{1_{G_2}\}$ and $w_2=d$. The graph $\Gamma_{13}$ is a subgroup graph of the subgroup $H=\phi(L(\Gamma_{13}, 1_H))$ of $G_1 \ast G_2$.  
   
 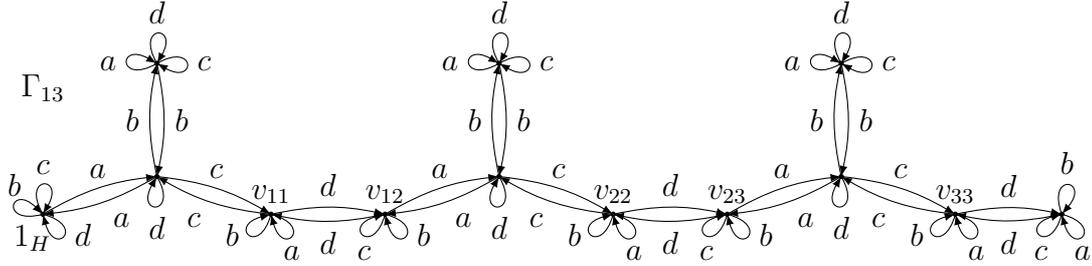
\begin{figure}[h!]
	\centering
 \begin{tikzpicture}%[>=stealth]
  % \node(A) at (1,1) [circle, draw] {$\Gamma$};
%Gamma_p
\coordinate[label=below:$\Gamma_{13}$] (0) at (0,2);
\coordinate[label=below:$1_H\ \ $] (1) at (0,0);
\coordinate (2) at (1.5,0.5); 
\coordinate (3) at (1.5,2);
\coordinate[label=above:$v_{11}$] (4) at (3,0);
\coordinate[label=above:$v_{12}$] (5) at (4.5,0);
\coordinate (6) at (6,0.5); 
\coordinate (7) at (6,2);
\coordinate[label=above:$v_{22}$] (8) at (7.5,0);
\coordinate[label=above:$v_{23}$] (9) at (9,0);
\coordinate (10) at (10.5,0.5);
\coordinate (11) at (10.5,2);
\coordinate[label=above:$v_{33}$] (12) at (12,0);
\coordinate (13) at (13.4,0);
%\coordinate[label=below:$\Gamma(H_1)^{(1)}$] (14) at (1.5,-0.6);
%\coordinate[label=below:$\Gamma(H_2)^{(1)}$] (14) at (3.8,-0.6);
%\coordinate[label=below:$\Gamma(H_1)^{(2)}$] (14) at (6,-0.6);
%\coordinate[label=below:$\Gamma(H_2)^{(2)}$] (14) at (8.3,-0.6);
%\coordinate[label=below:$\Gamma(H_1)^{(3)}$] (14) at (10.5,-0.6);
%\coordinate[label=below:$\Gamma(H_2)^{(3)}$] (14) at (12.8,-0.6);

\filldraw[black](1) circle (0.8pt) %Punkte
(2)circle(0.8pt)
(3)circle(0.8pt)
(4)circle(0.8pt)
(5)circle(0.8pt)
(6)circle(0.8pt)
(7)circle(0.8pt)
(8)circle(0.8pt)
(9)circle(0.8pt)
(10)circle(0.8pt)
(11)circle(0.8pt)
(12)circle(0.8pt)
(13)circle(0.8pt);

\path[->,min distance=6mm] 
(1) edge[in=135,out=195,above] node {$b$}(1)
(1) edge[in=60,out=120,above] node {$c$}(1)
(1) edge[in=285,out=345,right] node {$d$}(1)
(2) edge[in=240,out=300,below] node {$d$}(2)
(3) edge[in=150,out=210,left] node {$a$}(3)
(3) edge[in=330,out=30,right] node {$c$}(3)
(3) edge[in=60,out=120,above] node {$d$}(3)
(4) edge[in=285,out=345,below] node {$a$}(4)
(4) edge[in=195,out=255,left] node {$b$}(4)
(5) edge[in=285,out=345,right] node {$b$}(5)
(5) edge[in=195,out=255,below] node {$c$}(5)
(6) edge[in=240,out=300,below] node {$d$}(6)
(7) edge[in=150,out=210,left] node {$a$}(7)
(7) edge[in=330,out=30,right] node {$c$}(7)
(7) edge[in=60,out=120,above] node {$d$}(7)
(8) edge[in=285,out=345,below] node {$a$}(8)
(8) edge[in=195,out=255,left] node {$b$}(8)
(9) edge[in=285,out=345,right] node {$b$}(9)
(9) edge[in=195,out=255,below] node {$c$}(9)
(10) edge[in=240,out=300,below] node {$d$}(10)
(11) edge[in=150,out=210,left] node {$a$}(11)
(11) edge[in=330,out=30,right] node {$c$}(11)
(11) edge[in=60,out=120,above] node {$d$}(11)
(12) edge[in=285,out=345,below] node {$a$}(12)
(12) edge[in=195,out=255,left] node {$b$}(12)
(13) edge[in=195,out=255,below] node {$c$}(13)
(13) edge[in=285,out=345,below] node {$a$}(13)
(13) edge[in=50,out=110,above] node {$b$}(13)
;

\draw [->] (1) to[bend left=12] node[above] {$a$} (2);
\draw[->] (2) to[bend left=12] node[below right] {$a$} (1);
\draw[->] (3) to[bend left=12] node[right] {$b$} (2);
\draw[->] (2) to[bend left=12] node[left] {$b$} (3);
\draw[->] (4) to[bend left=12] node[below left] {$c$} (2);
\draw[->] (2) to[bend left=12] node[above] {$c$} (4);
\draw[->] (4) to[bend left=12] node[above] {$d$} (5);
\draw[->] (5) to[bend left=12] node[below] {$d$} (4);
\draw[->] (5) to[bend left=12] node[above] {$a$} (6);
\draw[->] (6) to[bend left=12] node[below right] {$a$} (5);
\draw[->] (7) to[bend left=12] node[right] {$b$} (6);
\draw[->] (6) to[bend left=12] node[left] {$b$} (7);
\draw[->] (6) to[bend left=12] node[above] {$c$} (8);
\draw[->] (8) to[bend left=12] node[below left] {$c$} (6);
\draw[->] (8) to[bend left=12] node[above] {$d$} (9);
\draw[->] (9) to[bend left=12] node[below] {$d$} (8);
\draw[->] (9) to[bend left=12] node[above] {$a$} (10);
\draw[->] (10) to[bend left=12] node[below right] {$a$} (9);
\draw[->] (11) to[bend left=12] node[right] {$b$} (10);
\draw[->] (10) to[bend left=12] node[left] {$b$} (11);
\draw[->] (12) to[bend left=12] node[below left] {$c$} (10);
\draw[->] (10) to[bend left=12] node[above] {$c$} (12);
\draw[->] (12) to[bend left=12] node[above] {$d$} (13);
\draw[->] (13) to[bend left=12] node[below] {$d$} (12);

%\node (A) at (1.5,0) [circle, draw] {$\Gamma_k$};

\end{tikzpicture} 
	\caption{An  $X$-graph as constructed in the proof of Theorem \ref{thmast1}.}
	\label{bsp}
\end{figure}

\begin{cor}\mbox{}\\
Let $W_1$ and $W_2$ be two finitely generated Coxeter groups. Then the nerve complex $\calnc (W, \scrh_\f )$ and the order  complex $\Delta P_\f (W)$   of the Coxeter group $W=W_1 \ast W_2$ are contractible. 
\end{cor}

\begin{proof} 
Take $H_1=W_1^+$ and $H_2=W_2^+$ the alternating subgroups of $W_1$ and $W_2$. 
\end{proof}

\begin{theorem}\label{thmast2}\mbox{}\\
Let $G_1=\left\langle \, X_1 \, | \, R_1 \,\right\rangle$ and $G_2=\left\langle \,  X_2 \, | \, R_2 \,\right\rangle$ be  groups with $X_1$, $X_2$ finite and $R_1$, $R_2$ not necessarily finite. Suppose that $G_1$ is a group such that there exists a collection $\{\Gamma_p\}_{p\in P}$  of $X_1$-graphs as in Theorem \ref{thmcon}.  
Then the nerve complex $\calnc(G, \scrh_\f )$ and the order complex $\Delta P_\f (G)$ are contractible for  $G=G_1\ast G_2$, $G=G_1\times G_2$ or $G=G_2 \rtimes G_1$.
\end{theorem}

\begin{proof}
We have $G_1 \ast G_2= \left\langle\,  X_1, X_2 \, | \, R_1, R_2\, \right\rangle$. 
We add to each graph $\Gamma_p \in \{\Gamma_p\}_{p\in P}$ a loop, labeled $y$, for each $y\in X_2$ at every vertex of $\Gamma_p$ and call the graph $\Gamma_p'$. 
Since $\Gamma_p'|_{X_1}=\Gamma_p$, each $\Gamma_p'$ fulfills the relators $R_1$. Since each edge labeled $y\in X_2$ is a loop, each $\Gamma_p'$ fulfills the relators $R_2$.  

For the groups
$ G_1 \times G_2= \left\langle \, X_1, X_2 \, | \, R_1, R_2, xyx^{-1}y^{-1}, x\in X_1, y\in X_2 \, \right\rangle$ and $G_2\rtimes_\psi G_1=\left\langle \, X_1, X_2 \, | \, R_1, R_2, xyx^{-1}(\psi(x)(y))^{-1}, x\in X_1, y\in X_2 \, \right\rangle$ we take the same collection $\{\Gamma_p'\}_{p\in P}$ of $(X_1\cup X_2)$-graphs $\Gamma_p'$ as above. It remains to show that each $\Gamma_p'$ fulfills the relators  $xyx^{-1}y^{-1}$ or $xyx^{-1}(\psi(x)(y))^{-1}$ for all $x\in X_1$ and $y\in X_2$. The image $\psi(x)(y)$ is an $X_2$-word. Thus a path in $\Gamma_p'$ with label $(\psi(x)(y))^{-1}$ consists only of loops. Therefore each $\Gamma_p'$ fulfills the relators. 

 Thus the collection $\{\Gamma_p'\}_{p\in P}$  of $(X_1\cup X_2)$-graphs satisfies the conditions of Theorem \ref{thmcon} for the groups $G_1 \ast G_2$, $G_1 \times G_2$ and $G_2 \rtimes G_1$.
\end{proof}

All groups of Type I and II satisfy the properties of the Theorems  \ref{thmconsub}, \ref{thmast1} and \ref{thmast2}. Therefore we have the following corollary.

\begin{cor}\label{corende}\mbox{}\\
Let $G$ be either the free product $G_1 \ast ...\ast G_n$, the direct product $G_1 \times ... \times G_n$, the semidirect product $G_1 \rtimes G_2$ or a finite index subgroup of the group $G_1$. Suppose that each $G_i$, for $1\leq i \leq n$, is one of the following finitely generated groups: \\
 a free group; a free abelian group; a Fuchsian group of genus $g \geq 2$;  an infinite right angled Coxeter group; an Artin group; a pure braid group; a Baumslag-Solitar group or an infinite virtually cyclic group. \\
Then the nerve complex $\calnc (G, \scrh_\f)$ and the order complex $\Delta P_\f (G)$ are contractible. 
\end{cor}

We now generalize the last point of Example \ref{extypII}, the amalgamated product $A\ast_D B$ of two finitely generated groups $A$ and $B$.

\begin{theorem}\label{thmamalgam}\mbox{}\\
Let $G_1$ and $G_2$ be finitely generated groups. Suppose there exist proper normal subgroups  $H_1 \lhd G_1$ and $H_2 \lhd G_2$  of finite index such that  $H_1 \cong H_2$ and $\{1,w_1,...,w_1^{n_1-1}\}$ and $\{1,w_2,...,w_2^{n_2-1}\}$ are full sets of coset representatives of $H_1< G_1$ and $H_2<G_2$ for some $w_1\in G_1$, $w_2\in G_2$. Then the nerve complex $\calnc(G, \scrh_\f)$ and the order complex $\Delta P_\f(G)$ are contractible for $G=G_1 \ast_D G_2$ with $\{1\}\leq D \leq H_1$.
\end{theorem}

\begin{proof}
Let $\psi\colon H_1 \rightarrow H_2$ be an isomorphism and  $H_1=\left\langle c_1,...,c_k \right\rangle$.   Then there exist presentations $G_1=\left\langle \, X_1 \, | \, R_1 \, \right\rangle$ and $G_2=\left\langle \, X_2 \, | \, R_2 \, \right\rangle$  with $X_1=\{ w_1, c_1,...,c_k \}$ and  $X_2=\{ w_2, \psi(c_1),...,\psi(c_k) \}$. 
 Let $\Gamma_{X_i,R_i}(H_i)$ be the subgroup graph of $H_i$ in $G_i$ for $i=1,2$. Every coset of $H_i$ in $G_i$ is of the form $H_iw_i^j$ for $0\leq j < n_i$. Therefore  $\Gamma_{X_i,R_i}(H_i)|_{\{w_i\}}$ is a $(w_i,n_i)$-circle. 
Let $v$ be the terminus of the edge $e$ with origin $1_{H_1}$ and label $c_1$ and let $H_1w_1^j$ be the coset corresponding to $v$. Then $c_1 \in H_1w_1^j$. Consequently, $j=0$ and $v=1_{H_1}$. Thus $e$ is a loop. Analogously, all edges with label $c_i$ and origin $1_{H_1}$    or with label $\psi(c_i)$ and origin $1_{H_2}$ are loops. Since $H_1$ and $H_2$ are normal, $(\Gamma_{X_1,R_1}(H_1), 1_{H_1})\cong (\Gamma_{X_1,R_1}(H_1), v)$ and $(\Gamma_{X_2,R_2}(H_2), 1_{H_2})\cong (\Gamma_{X_2,R_2}(H_2), v')$ for all vertices $v$ in $\Gamma_{X_1,R_1}(H_1)$ and $v'$ in $\Gamma_{X_2,R_2}(H_2)$. Thus every edge labeled $c_i$ respectively $\psi(c_i)$ is a loop  at each vertex of $\Gamma_{X_1,R_1}(H_1)$ respectively  $\Gamma_{X_2,R_2}(H_2)$. 

For the group $G_1 \ast_D G_2=\left\langle \, X_1, X_2\, | \, R_1, R_2, d\psi(d)^{-1}, d\in D\, \right\rangle$ we construct the collection  $\{\Gamma_p\}_{p\in P}$ of $(X_1\cup X_2)$-graphs $\Gamma_p$ as in Theorem \ref{thmast1}. Thus each $\Gamma_p$ is a $(w_1,n_1,w_2,n_2)$-graph with $p=n_1+n_2-1+(n_1+n_2-2)n$ vertices and a loop for each $c_i$ and $\psi(c_i)$ at every vertex.  
Since $D\leq H_1$, each $d\in D$ is a word in $\{c_1^{\pm 1},..., c_k^{\pm 1}\}$ and each $\psi(d)$ a word in $\{\psi(c_1)^{\pm 1},..., \psi(c_k)^{\pm 1}\}$. Therefore  each $\Gamma_p\in \{\Gamma_p\}_{p\in P}$ fulfills the relator $d\psi(d)^{-1}$ for all $d\in D$.  
Theorems  \ref{thmcon} and \ref{thmast1}  complete the proof.   
\end{proof}

\end{document}